\documentclass[11pt,reqno,a4paper]{amsart} 
\usepackage{amsmath}
\usepackage{amssymb}
\usepackage{euscript}
\usepackage{mathabx}
\usepackage{mathrsfs}
\usepackage[all]{xy}
\usepackage{graphicx}
\usepackage{color}
\usepackage[colorlinks=true,linktocpage=true,pagebackref=false, urlcolor=black, citecolor=black,linkcolor=black]{hyperref}
\usepackage{tikz}
\usepackage{pgfplots}
\usepackage{tikzducks} 

\pgfplotsset{soldot/.style={color=black,only marks,mark=*}} \pgfplotsset{holdot/.style={color=black,fill=white,only marks,mark=*}}

\newtheorem{thm}{Theorem}[section]

\newtheorem{lem}[thm]{Lemma}

\newtheorem{cor}[thm]{Corollary}

\newtheorem{prob}[thm]{Problem}

\theoremstyle{definition}
\newtheorem{defn}[thm]{Definition}
\newtheorem{defns}[thm]{Definitions}

\theoremstyle{remark}
\newtheorem{remark}[thm]{Remark}
\newtheorem{remarks}[thm]{Remarks}
\newtheorem{example}[thm]{Example}
\newtheorem{examples}[thm]{Examples}

\numberwithin{equation}{section}

\newcommand{\K}{{\mathbb K}} \newcommand{\N}{{\mathbb N}}
\newcommand{\Z}{{\mathbb Z}} \newcommand{\R}{{\mathbb R}}
 \newcommand{\C}{{\mathbb C}}

 \newcommand{\J}{{\mathcal J}}
 
 \newcommand{\I}{{\mathcal I}}

\newcommand{\ceros}{{\mathcal Z}}

\newcommand{\gtp}{{\mathfrak p}} 
\newcommand{\gtm}{{\mathfrak m}} \newcommand{\gtn}{{\mathfrak n}}
\newcommand{\gta}{{\mathfrak a}} \newcommand{\gtb}{{\mathfrak b}}

\newcommand{\an}{{\EuScript O}}
\newcommand{\Ff}{{\EuScript F}}

\newcommand{\Ss}{{\EuScript S}}

\newcommand{\Jj}{\J}
\newcommand{\Ii}{\I}
\newcommand{\Qq}{{\EuScript Q}}
\newcommand{\Mm}{{\EuScript M}}
\newcommand{\Nn}{{\EuScript N}}
\newcommand{\Gg}{{\EuScript G}}
\newcommand{\ZZ}{\mathcal{Z}}
\newcommand{\Cc}{{\EuScript C}}
\newcommand{\Bb}{{\EuScript B}}

\newcommand{\im}{\operatorname{im}}
\newcommand{\Reg}{\operatorname{Reg}}
\newcommand{\Sing}{\operatorname{Sing}}

\newcommand{\cl}{\operatorname{Cl}}
\newcommand{\supp}{\operatorname{supp}}

\newcommand{\x}{{\tt x}} \newcommand{\y}{{\tt y}} 
\newcommand{\z}{{\tt z}} \renewcommand{\t}{{\tt t}}

\newcommand{\veps}{\varepsilon}

\newcommand{\ol}{\overline}

\numberwithin{equation}{section}

\begin{document}
\title[A converse to Cartan's Theorem B]{A converse to Cartan's Theorem B: The extension property for real analytic and Nash sets}

\author{Jos\'e F.\ Fernando}
\address{Departamento de \'Algebra, Geometr\'\i a y Topolog\'\i a, Facultad de Ciencias Matem\'aticas, Universidad Complutense de Madrid, Plaza de Ciencias, 3, 28040 MADRID (SPAIN)}
\email{josefer@mat.ucm.es}

\author{Riccardo Ghiloni}
\address{Dipartimento di Matematica, Via Sommarive, 14, Universit\`a di Trento, 38123 Povo (ITALY)}
\email{riccardo.ghiloni@unitn.it}

\begin{abstract}
In 1957 Cartan proved his celebrated Theorem B and deduced that if $\Omega\subset\R^n$ is an open set and $X$ is a coherent real analytic subset of $\Omega$, then $X$ has the analytic extension property, that is, each real analytic function on $X$ extends to a real analytic function on $\Omega$. So far, the converse implication in its full generality remains unproven. As a matter of fact, in the literature only special cases of non-coherent real analytic sets $X\subset\Omega$ without the extension property appear, some of them due to Cartan himself: mainly real analytic sets $X\subset\Omega$ that have a visible `tail', that is, $X$ has a non-pure dimensional irreducible analytic component $Y$ such that the set of points of lower dimension of $Y$ is visible inside $X$ (this means that there exists a point $x\in Y$ such that $\dim(X_x)<\dim(Y)$). These examples generated the feeling that non-coherent real analytic sets without the analytic extension property may have visible `tails'.

In this article we prove the converse implication: If a subset $X$ of $\Omega$ has the analytic extension property, then it is a coherent real analytic subset of $\Omega$. Thus, the class of sets with the analytic extension property coincides with that of coherent real analytic sets. In fact, we prove that if $X\subset\Omega$ is a non-coherent global analytic set, there exist `many' meromorphic functions on $\Omega$ that are analytic on $X$, but have no analytic extension to $\Omega$, yielding an almost complete description of the possible non-extendability sets. To that end we analyze the first cohomology group of the sheaf of zero ideals of $X$.

We extend the previous characterization to the Nash case, which is somehow more demanding, because of its finiteness properties and its disappointing behavior with respect to cohomology of sheaves of Nash function germs. Let $\Omega\subset\R^n$ be an open semialgebraic set and let $X\subset\Omega$ be a set. We prove that each Nash function on $X$ extends to a Nash function on $\Omega$ if and only if $X\subset\Omega$ is a coherent Nash set. The `if' implication goes back to some celebrated results of Coste, Ruiz and Shiota, whereas the `only if' implication has been treated only for Nash sets $X$ that have visible `tails'. We prove that if $X\subset\Omega$ is a non-coherent Nash set, there exist `many' Nash meromorphic functions on $\Omega$ that are local Nash on $X$, but have no Nash extension to $\Omega$, providing once more an almost complete description of the possible non-extendability sets. 

If $M\subset\R^n$ is a Nash manifold, $\Cc^\infty$ semialgebraic functions on $M$ coincide with Nash functions on $M$. As an application of our previous strategies, we confront the coherence of a Nash set $X\subset\Omega$ with the fact that each $\Cc^\infty$ semialgebraic function on $X$ is a Nash function on $X$. Namely, we prove that these two properties are equivalent for Nash sets. More generally, we provide a full characterization of the semialgebraic sets $S\subset\Omega$ for which $\Cc^\infty$ semialgebraic functions on $S$ coincide with Nash functions on $S$.
\end{abstract}

\date{24/07/2025}

\maketitle 


\section{Introduction}\label{s1}

Let $\Omega$ be an open subset of $\R^n$ and let $X$ be a subset of $\Omega$. We denote $\N:=\{0,1,2,\ldots\}$ the set of natural numbers including $0$. A function $f:X\to\R$ is called \emph{(real) analytic} if for each point $a\in X$ there exists an open neighborhood $U$ of $a$ in $\Omega$ such that $f|_{X\cap U}$ extends to an analytic function defined on $U$ or, equivalently, if there exist an open neighborhood $U$ of $a$ in $\Omega$ and a (real) power series $\sum_{\alpha\in\N^n}a_\alpha (x-a)^\alpha$ that converges at the points of $U$ and such that the value $f(x)$ equals the sum of the convergent series $\sum_{\alpha\in\N^n}a_\alpha (x-a)^\alpha$ in $\R$ for each $x\in X\cap U$. We say that the set $X\subset\Omega$ has the \emph{analytic extension property} if each analytic function $f:X\to\R$ extends to an analytic function defined on $\Omega$. A main initial difficulty here is the absence of analytic partitions of unity.

In this article we solve the following extension problem.

\begin{prob}\label{prob:analytic}
How can one decide whether a set $X\subset\Omega$ has the analytic extension property?
\end{prob}

This problem goes back to the 1957 article \cite{c2} where Cartan proved several fundamental results of the theory of coherent sheaves, including the celebrated Theorem B. Recall that a set $X\subset\Omega$ is \emph{analytic} if it is closed in $\Omega$ and for each $a\in X$ there exist an open neighborhood $U$ of $a$ in $\Omega$ and finitely many analytic functions $f_1,\ldots,f_r$ defined on $U$ whose common zero set is $X\cap U$. The set $X\subset\Omega$ is \emph{global analytic} if there exist finitely many analytic functions $f_1,\ldots,f_r$ defined on the whole $\Omega$ whose common zero set is $X$. There exist analytic sets which are not global analytic \cite[\S11]{c2}. Global analytic sets were called \emph{$C$-analytic sets} by Whitney and Bruhat in \cite{wb}, probably as an abbreviation for `Cartan real analytic sets'. 

The global analyticity is an immediate necessary condition for the set $X\subset\Omega$ to have the analytic extension property. We borrow the suitable strategy from the proof of \cite[Thm.1]{nt}. Namely, suppose the set $X\subset\Omega$ has the analytic extension property, but it is not $C$-analytic. By \cite[\S6.Prop.7 \& Cor.2]{wb} or \cite[Cor.I.8]{fr}, the family of $C$-analytic subsets of $\Omega$ is stable under arbitrary intersections. Thus, there exists the smallest $C$-analytic subset $Y$ of $\Omega$ that contains $X$. Suppose there exists a point $y\in Y\setminus X$. Consider the analytic function $f:X\to\R,\ x\mapsto\frac{1}{\|x-y\|^2}$ and an analytic function $F:\Omega\to\R$ that extends $f$. Define the analytic function $G:\Omega\to\R$, $x\mapsto1-F(x)\|x-y\|^2$. As $G$ is identically zero on $X$, it vanishes identically on $\{y\}$, but this is a contradiction, because $y\in Y$ and $G(y)=1$.

Global analyticity is still not enough to assure the analytic extension property. Consider Whitney's umbrella $W:=\{(x,y,z)\in\R^3:y^2-zx^2=0\}$ and the function 
$$
f:W\to\R,\ (x,y,z)\mapsto\begin{cases}
\frac{x}{z+1}&\text{if $(x,y,z)\neq (0,0,-1)$},\\
0&\text{otherwise.}
\end{cases}
$$
Observe that $W\subset\R^3$ is a $C$-analytic set (even an algebraic set) and $f$ is an analytic function on $W$. The latter assertion can be proven as follows. In the following we denote ${\tt i}:=\sqrt{-1}$. Whitney's umbrella $W$ decomposes as the union of the singular ruled surface $C:=\{x\neq0,z=\frac{y^2}{x^2}\}\sqcup\{x=y=0,z\geq0\}$ (the cloth of the umbrella) and the $z$-semiaxis $H:=\{x=y=0,z\leq0\}$ (the handle of the umbrella). Observe that $C$ meets $H$ only at the origin $O$. As $p\in H\setminus\{O\}$ and $f$ vanishes identically around $p$, the function $f$ is analytic on $W$, because $\{z+1=0\}\cap W=\{p\}$. Suppose now that $f$ has an analytic extension $F$ to $\R^3$. Observe that $(z+1)F(x,y,z)=x$ on $W$. We can complexify the present geometric configuration using \cite[\S7]{wb} and obtain an open neighborhood $V$ of $\R^3$ in $\C^3$ and a holomorphic function $F_\C:V\to\C$ such that $F_\C$ extends $F$ and $(z+1)F_\C(x,y,z)=x$ on $W_\C:=\{(x,y,z)\in V:y^2-zx^2=0\}$. Thus, if $\epsilon>0$ is small enough, the point $q:=(\epsilon,\epsilon{\tt i},-1)$ belongs to $W_\C$ and $0=0F_\C(q)=\epsilon$, which is a contradiction. Consequently, the analytic extension $F$ of $f$ does not exist. 

The reader observes that this proof of the non-existence of $F$ is based on the following fact: if we move on $W$ from the cloth to the handle through the origin, the local (real) dimension of $W$ decreases from $2$ to $1$, but the local (complex) dimension of the complexification $W_\C$ of $W$ does not: it is constantly equal to $2$. This phenomenon of dimensional decrease allows to `hide in the handle of $W$' the poles of the meromorphic function $(x,y,z)\mapsto\frac{1}{z+1}$ after multiplying it by the analytic function $(x,y,z)\to x$, which vanishes identically on the handle but not on $W$. Thus, we obtain `by restriction to the real part $W$ of $W_\C$ an analytic function $f$ on $W$ with a unique `\emph{hidden pole}' at $p$. On the contrary, the meromorphic function $(x,y,z)\mapsto\frac{x}{z+1}$ has poles on $W_\C$ exactly at the points $\{(x,y,z)\in V:\ (x+{\tt i}y)(x-{\tt i}y)=0,z=-1\}$, which is the intersection with $W$ of the union of two (complex) lines.

Cartan's Theorem B provides a sufficient condition for a $C$-analytic set to have the analytic extension property. Such a sufficient condition is a local condition, the so-called \emph{coherence}. Let $X\subset\Omega$ be an analytic set and let $a\in X$. Consider an open neighborhood $U$ of $a$ in $\Omega$ and analytic functions $f_1,\ldots,f_r$ defined on $U$ such that their common zero set is $X\cap U$. The functions $f_1,\ldots,f_r$ are \emph{good equations for $X$ locally at $a$} if each local analytic equation of $X$ at $a$ is a local analytic combination of $f_1,\ldots,f_r$. More precisely, if $f$ is an analytic function defined on an open neighborhood $U_1$ of $a$ in $\Omega$ that vanishes on $X\cap U_1$, there exist an open neighborhood $U_2$ of $a$ in $U\cap U_1$ and analytic functions $g_1,\ldots,g_r$ defined on $U_2$ such that $f|_{U_2}=g_1f_1|_{U_2}+\ldots+g_rf_r|_{U_2}$ on $U_2$. The analytic set $X\subset\Omega$ is said to be \emph{coherent at $a\in X$} if there exist an open neighborhood $U$ of $a$ in $\Omega$ and analytic functions $f_1,\ldots,f_r$ defined on $U$ such that $f_1,\ldots,f_r$ are good equations for $X$ locally at each point $b\in X\cap U$. The analytic set $X\subset\Omega$ is \emph{coherent} if it is coherent at each of its points. Non-singular analytic subsets of $\Omega$ and singular analytic curves of $\Omega$ are examples of coherent analytic set. Cartan proved that each coherent analytic set is $C$-analytic and that the above phenomenon of dimensional decrease cannot occur locally at any coherent point, see \cite[\S9\&10]{c2} for further details. In particular, Whitney's umbrella is not coherent at the origin. Actually, Cartan's Theorem B \cite[\S6, Thm.3(B)]{c2} implies the following.

\begin{thm}[Cartan, {\cite[\S7(2)]{c2}}]\label{c-ext}
Each coherent analytic set $X\subset\Omega$ has the analytic extension property.
\end{thm}

Coherence of analytic sets is thus a sufficient local condition that guarantees the analytic extension property. In general, the problem of finding local obstructions to have the analytic extension property has been treated in the literature mainly for $C$-analytic sets for which the previously described phenomenon of dimensional decrease appears. As a matter of fact, in the mind of many researchers non-coherence of $C$-analytic sets concerns the existence of `\emph{tails}', as in the case of Whitney's umbrella. However, this is not the general situation. In Section \ref{s3} we present several examples of pure dimensional $C$-analytic sets (algebraic sets indeed) that are not coherent. The following theorem is our first main result: it asserts that coherence is also a necessary condition to have the analytic extension property.

\begin{thm}\label{thm1}
A set $X\subset\Omega$ has the analytic extension property if and only if it is a coherent analytic set.
\end{thm}

To that end we prove (in view of Theorem \ref{c-ext}) the following result (see also Theorem \ref{main11}). Recall that the ring $\Mm(\Omega)$ of meromorphic functions on an open subset $\Omega\subset\R^n$ is the total ring of fractions of the ring $\an(\Omega)$ of analytic functions on $\Omega$. The meaning of the word `many' in the following result will be explained in Remark \ref{many}(ii).

\begin{thm}\label{main1}
Let $X\subset\Omega$ be a non-coherent $C$-analytic set. Then there exist `many' meromorphic functions $\xi:\Omega\dashrightarrow\R$ whose restrictions $f:=\xi|_X$ are analytic on $X$, but have no analytic extensions to $\Omega$.
\end{thm}

Problem \ref{prob:analytic} has a natural counterpart in Nash geometry. Recall that a subset of $\R^n$ is \emph{semialgebraic} if it is a Boolean combination of sets defined by polynomial equalities and inequalities. Let $\Omega\subset\R^n$ be an open semialgebraic set. A function $f:\Omega\to\R$ is said to be \emph{Nash} if it is smooth and its graph is semi-algebraic, or equivalently, by \cite[Prop.8.1.8]{bcr} if it is analytic and its graph is semialgebraic. Let $X$ be a subset of $\Omega$. A function $f:X\to\R$ is called \emph{local Nash} if, for each point $a\in X$, there exists an open semialgebraic neighborhood $U$ of $a$ in $\Omega$ such that $f|_{X\cap U}$ extends to a Nash function defined on $U$. We say that the set $X\subset\Omega$ has the \emph{Nash extension property} if each local Nash function $f:X\to\R$ extends to a Nash function defined on $\Omega$. The set $X\subset\Omega$ is a \emph{Nash set} if it is the common zero set of finitely many Nash functions defined on $\Omega$. An adapted argument to the one exposed above \cite[Thm.1]{nt} for the analytic setting shows that the Nash extension property is reserved for Nash sets. Let $\Omega\subset\R^n$ be an open semialgebraic set and let $X\subset\Omega$. Suppose that $X$ has the Nash extension property, but it is not a Nash set. As the ring $\Nn(\Omega)$ is noetherian \cite[Thm.8.7.18]{bcr}, there exists the smallest Nash subset $Y$ of $\R^n$ that contains $X$. Observe that $Y$ is a $C$-analytic set. Suppose there exists a point $y\in Y\setminus X$ and consider the local Nash function $f:X\to\R,\ x\mapsto\frac{1}{\|x-y\|^2}$. As $f$ has no analytic extension to $\R^n$, it has no Nash extension to $\R^n$, as well.

The notion of coherent Nash set can be defined as in the analytic case: it is enough to replace `open neighborhood' with `open semialgebraic neighborhood' and `analytic' with `Nash', respectively. As we recall in Lemma \ref{ijn}, a Nash set $X\subset\Omega$ is coherent (as a Nash set) if and only if it is coherent as an analytic set (se also \cite[\S2.B]{bfr}). 

Similarly to what we have formulated in the analytic case, we arrive to the following Nash extension problem. 

\begin{prob}\label{prob:n}
Given an open semialgebraic set $\Omega$ and a Nash set $X\subset\Omega$, is it true that $X\subset\Omega$ has the Nash extension property if and only if $X\subset\Omega$ is coherent?
\end{prob}

The `if' implication was proven in the remarkable papers \cite{crs2,cs} by Coste, Ruiz and Shiota. So far, as in the analytical case, the `only if' implication has been treated mainly for Nash sets $X\subset\R^n$ with visible `tails'.

Problem \ref{prob:n} is solved affirmatively by our second main result.

\begin{thm}\label{thm2}
A subset $X$ of an open semialgebraic set $\Omega$ has the Nash extension property if and only if $X\subset\Omega$ is a coherent Nash set.
\end{thm}

To that end we prove the following result (see also Theorem \ref{main21}). Recall that the ring $\Mm^\bullet(\Omega)$ of meromorphic Nash functions on an open semialgebraic subset $\Omega\subset\R^n$ is the total ring of fractions of the ring $\Nn(\Omega)$ of Nash functions on $\Omega$. The meaning of the word `many' in the following result will be explained in Remark \ref{many}(ii).

\begin{thm}\label{main2}
Let $X$ be a non-coherent Nash subset of an open semialgebraic set $\Omega$. Then there exist `many' meromorphic Nash functions $\xi:\Omega\dashrightarrow\R$ on $\Omega$ whose restrictions $f:=\xi|_X$ to $X$ are local Nash functions, but have no Nash extensions to $\Omega$. 
\end{thm}

The proof of Theorem \ref{main2} is similar to the one of Theorem \ref{main1}, but requires to take into account the finiteness conditions imposed by the semialgebricity of Nash functions. In the Nash setting there exist difficulties to pass from local to global results, due to the unexpected cohomological behavior of Nash functions even in the simplest cases. Namely, Hubbard proved that $H^1(\R,\Nn_\R)\neq0$, see \cite{hb}. We explain the precise technicalities in \S\ref{nasht}.

The theory of sheaves is the appropriate setting where one can contextualize the concept of coherence and consequently the theorems stated above in the analytic case. Let $\Omega$ be an open subset of $\R^n$ and let $X\subset\Omega$ be an analytic set. The idea is to associate to each point $x\in\Omega$ the ring $\an_{\Omega,x}$ of germs $f_x$ at $x$ of the analytic functions $f$ on a neighborhood of $x$ in $\Omega$ and the ideal $\Jj_{X,x}$ of $\an_{\Omega,x}$ constituted by the germs $h_x$ that vanish identically on the germ $X_x$ of $X$ at $x$. Define the set $\an_\Omega:=\bigcup_{x\in\Omega}\an_{\Omega,x}$, the projection map $\pi:\an_\Omega\to\Omega,\ f_x\mapsto x$ and the subset $\Jj_X:=\bigcup_{x\in\Omega}\Jj_{X,x}$ of $\an_\Omega$. The set $\an_\Omega$ admits a natural topology such that each analytic functions $f:\Omega\to\R$ can be interpreted as a continuous section $\sigma:\Omega\to\an_\Omega$ of the projection map $\pi$ (that is, $\pi\circ\sigma=\mathrm{id}_\Omega$). Namely, if $f:\Omega\to\R$ is an analytic function, define the section $\sigma_f:\Omega\to\an_\Omega,\ x\mapsto f_x$ of $\pi$ induced by $f$, whereas if $\sigma:\Omega\to\an_\Omega$ is a continuous section of $\pi$, define the analytic function $f_\sigma:\Omega\to\R,\ x\mapsto(\sigma(x))(x)$ associated to $\sigma$. It holds that $f_{\sigma_f}=f$ and $\sigma_{f_\sigma}=\sigma$. If we endow $\an_\Omega$ with such a topology, we obtain the so-called \emph{sheaf $\an_\Omega$ of analytic functions on $\Omega$} and $\Jj_X$ endowed with the induced topology is the \emph{subsheaf of $\an_\Omega$ of ideals of germs of analytic functions vanishing on $X$}. The subsheaf $\Jj_X$ of $\an_\Omega$ is coherent at a point $a\in\Omega$ if there exist an open neighborhood $U$ of $a$ in $\Omega$ and analytic functions $f_1,\ldots,f_r$ defined on $U$ such that the germs $f_{1,b},\ldots,f_{r,b}$ generate the ideal $\Jj_{X,b}$ of the ring $\an_{\Omega,b}$ for each $b\in U$. The analytic set $X\subset\Omega$ is \emph{coherent} if $\Jj_X$ is coherent at each point of $\Omega$. This `sheaf version' of coherence of $X\subset\Omega$ coincides with the one given above.

\subsection{Whitney's extension problem and related results}\label{wet}
Close to Problems \ref{prob:analytic} and \ref{prob:n} appears another classic extension problem often called \emph{Whitney's extension problem}. Although the `differentiable nature' of the latter problem is essentially different from the analytical nature of Problems \ref{prob:analytic} and \ref{prob:n}, there are some subtle and deep connections between them that reveals the central role of sheaves and other similar tools.

In his pioneering 1934 articles \cite{w1,w2,w3} Whitney raised the question (known as {\em Whitney's extension problem}) of determining whether a function $f$ defined on a closed subset $X$ of $\R^n$ admits an extension to $\R^n$ of class $\Cc^p$ for some $p\in\N$. In \cite{w1} it is proved what is known as classical \emph{Whitney's extension theorem}. Namely, Whitney characterized the families $\{f_\alpha\}_{|\alpha|\leq p}$ of continuous functions on $X$ such that $f_0=f$ admits an extension $F:\R^n\to\R$ of class $\Cc^p$ whose Taylor polynomial at each point $a\in X$ is $\sum_{|\alpha|\leq p}f_\alpha(a)(x-a)^\alpha$. In addition, he proved that $F$ can be chosen analytic on $\R^n\setminus X$. 

The deeper question of finding an extension of $f$ to $\R^n$ of class $\Cc^p$ using only the values of $f$ was treated in \cite{w2}: here Whitney gave a solution in the case $n=1$ using limits of $p^{\mathrm{th}}$ divided differences of $f$. In \cite{g} Glaeser introduced a `paratangent bundle' of $X$ using limits of secants and solved the question in the case $p=1$. Bierstone, Milman and Paw{\l}ucki \cite{bmp2} generalized Glaeser's construction and defined the \emph{paratangent bundle $\tau^p(X)$ of $X$ of any order $p\in\N$}. The idea is similar (dual in some sense) to the one used to construct the subsheaf $\Jj_X$ of $\an_\Omega$. Namely, repeating a limit procedure (Glaeser operation) a finite number of times, one associates to each point $x\in X$ a linear subspace $\tau^p_x(X)$ of the dual space $\EuScript{P}_p(\R^n)^*$ of the real vector space $\EuScript{P}_p(\R^n)$ of polynomial functions on $\R^n$ of degree $\leq p$. Then, one defines $\tau^p(X):=\bigcup_{x\in X}\tau^p_x(X)$. Applying this construction to the graph of any function $f:X\to\R$, one obtains a bundle $\nabla^pf$ contained in $\tau^p(X)\times\R$. In the same article \cite{bmp2} the authors conjectured that \emph{$f$ has an extension to $\R^n$ of class $\Cc^p$ if and only if $\nabla^pf$ is a `function'} (more precisely, the graph of a function $\tau^p(X)\to\R$), and they proved a version of their conjecture: \emph{if $X\subset\R^n$ is a compact subanalytic set and $\nabla^qf$ is a `function' for some $q\geq p$ sufficiently large, then $f$ has an extension to $\R^n$ of class $\Cc^p$}. The proof of this remarkable result is based on a $\Cc^p$ composite function property of (the uniformizations of) compact subanalytic sets established in \cite{bmp1}.

The $\Cc^p$ composite function property is a version of the $\Cc^\infty$ composite function property of closed subanalytic sets previously introduced in \cite{bm1}. Suppose $X\subset\R^n$ is a closed subanalytic set. Let $\Cc^k(X)$ be the set of restrictions to $X$ of functions defined on $\R^n$ of class $\Cc^k$ for each $k\in\N\cup\{\infty\}$ and define $\Cc^{(\infty)}(X):=\bigcap_{k\in\N}\Cc^k(X)$. Evidently, $\Cc^\infty(X)\subset\Cc^{(\infty)}(X)$. An important result asserts that the $\Cc^\infty$ composite function property for $X$ is equivalent both to the equality $\Cc^\infty(X)=\Cc^{(\infty)}(X)$ and to a `stratified coherence' property of $X$, the so-called \emph{semicoherence} of $X$, widely studied and developed in \cite{bm3,bm4}. Thus, \emph{$\Cc^\infty(X)=\Cc^{(\infty)}(X)$ if and only if $X\subset\R^n$ is a semicoherent subanalytic set} \cite{bmp1}. Our Theorems \ref{thm1} and \ref{thm2} are the analytic and Nash counterparts of the latter differential subanalytic result. Examples of semicoherent subanalytic sets are the closed Nash subanalytic sets, which include closed semianalytic sets \cite{bm1}. In \cite{p1} Paw{\l}ucki provided an example of $3$-dimensional compact subanalytic set $P\subset\R^5$ that is not semicoherent, so $\Cc^\infty(P)\subsetneq\Cc^{(\infty)}(P)$ (see also \cite{p2}). In his 2006 article \cite{f3} Fefferman provided a complete solution to Whitney's extension problem by proving the above Bierstone-Milman-Paw{\l}ucki conjecture via the use of a natural variant of the paratangent bundle of order $p$, see also \cite{f1,f2,bmp3}. We refer the reader to Fefferman's expository paper \cite{f4} for further details and references.
 
Whitney's extension theorem has been generalized to the o-minimal setting by Kurdyka, Paw{\l}ucki and Thamrongthanyalak in \cite{kp1,kp2,th}. The statements are natural o-minimal reformulations of the classical theorem. The proofs however are quite different, because o-minimality involves always finiteness restrictions that prevent the use of certain strong techniques from the differential case.

Let $S\subset\R^n$ be a closed semialgebraic set and let $f:S\to\R$ be a semialgebraic function. Suppose $f$ extends to a function $F:\R^n\to\R$ of class $\Cc^p$ for some $p\in\N$. A question proposed by Bierstone and Milman \cite{z} asks whether it is possible to take $F$ semialgebraic. In the continuous case, the affirmative solution is contained in the article \cite{dk} by Delfs and Knebusch. In \cite{at}, Aschenbrenner and Thamrongthanyalak solved affirmatively this question in the case $p=1$. Recently, Fefferman and Luli \cite{fl} settled the case $n=2$ for an arbitrary $p\geq1$. In \cite{bcm} Bierstone-Campesato-Milman provided a weak solution to the previous problem with a {\em loss of differentiability} via a function $t:\N\to\N$ such that $t(p)\geq p$ for each $p\geq1$. Namely, if the semialgebraic function $f:S\to\R$ extends to a function $F:\R^n\to\R$ of class $\Cc^{t(p)}$ for some $p\in\N$, then $f$ extend to a $\Cc^p$ differentiable semialgebraic function $F':\R^n\to\R$.

In Section \ref{s7} and Appendix \ref{B} we relate smooth semialgebraic functions on a semialgebraic set $S\subset\R^n$ (see \S\ref{ssf}) with semialgebraic functions on $S$ that are local Nash at the points of $S$. We characterize in Theorem \ref{as} the semialgebraic sets $S$ such that the smooth semialgebraic functions on $S$ coincide with the Nash functions on $S$. In case $S=X$ is a Nash subset of an open semialgebraic subset $U$ of $\R^n$, the previous coincidence is provided exactly when $X$ is coherent (see Theorem \ref{main2} and Lemma \ref{ens}).

\subsection*{Structure of the article}
The article is organized as follows. In Section \ref{s2} we recall and analyze some results concerning coherence in both the real analytic and the Nash cases. In Section \ref{s3} we study the set $T(X)$ of `tails' of a $C$-analytic set $X\subset\R^n$, we recall a description of its set $N(X)$ of points of non-coherence provided in \cite{abf2} and provide a good understanding of both sets. We show as well that both $T(X)$ and $N(X)$ are semialgebraic subsets of $X$ in case $X$ is a Nash set. We also discuss in Section \ref{s3} some enlightening examples that the reader can have in mind to illustrate the strategy followed to prove in Section \ref{s6} (stronger versions of) the main results of this article: Theorems \ref{main1} and \ref{main2}. In Sections \ref{s4} and \ref{s5} we present the key Theorems \ref{xy} and \ref{h}: (1) $\cl(T(X)\setminus N(X))=T(X)\cup N(X)$ and (2) there exists special analytic equations (resp. Nash equations) of $X\setminus N(X)$ in $\R^n\setminus N(X)$ that encode the `tails' of $X$ outside $N(X)$. Finally, in Section \ref{s7} we prove that the ring of $\Cc^\infty$ semialgebraic functions on a Nash set $X\subset\R^n$ coincide with the ring of Nash functions on $X$ if and only if $X$ is coherent. More generally, we characterize the semialgebraic sets $S\subset\R^n$ for which smooth semialgebraic functions on $S$ coincide with Nash functions on $S$. The article finishes with Appendix \ref{A}, where we prove the semialgebraicity of certain constructions (complexification and its normalization) in the Nash setting, and Appendix \ref{B}, where we prove the equivalence between the smooth semialgebraic functions on $S$ and the semialgebraic functions on $S$ that are local Nash on $S$.

\subsubsection*{Suggestion to read the proof of Theorem {\em\ref{main1}}}
If the reader is only interested in finding a meromorphic function on $\Omega$ that has a single pole, it is analytic on $X$ and has no analytic extension to $\Omega$, the proof provided in this article can be substantially shortened (see also Remark \ref{summ}). The reader can skip Sections \ref{s3}, \ref{s4} and \ref{s5}, go directly to Section \ref{s6} (Theorem \ref{main11}) and consider the case of a singleton $Y:=\{y\}$ following the construction contained in {\sc Steps 1} and {\sc 2} and skipping the arguments that involve a $Y$ of dimension $\geq1$. The explicit construction of such function is done in {\sc Step 3}. Sections \ref{s3}, \ref{s4} and \ref{s5} are included to show that if $\dim(X)\geq3$ and the set of (maybe hidden) `tails' of $X$ has dimension $\geq2$, one can produce many different meromorphic function on $\Omega$, that are analytic functions on $X$, have no analytic extension to $\Omega$ (and are of different nature to the ones already known for the case of visible `tails'). Alternatively, as a first approach to Theorem \ref{main1}, we refer the reader to the lecture of Kollar \cite{ko}, which provides a slightly different proof of the case of a single pole. Its main result (Theorem 2) is somehow the union of our Corollary \ref{ij}, Lemma \ref{dot} and the suggested simplified version of Theorem \ref{main11}. 

\subsubsection*{Suggestion to read the proof of Theorem {\em\ref{main2}}}
As it happens in the analytic case, if the reader is interested in a local Nash function on $X$ that is meromorphic with a single pole and does not have any Nash extension on $\Omega$, the proof of Theorem \ref{main2} can be substantially shortened and Sections \ref{s3}, \ref{s4} and \ref{s5} can be skipped again. In this case cohomological arguments do not work properly and {\sc Steps 3} and 4 of the proof of Theorem \ref{main2} (Theorem \ref{main21}) are needed to get around this inconvenience. 

\section{Coherence of analytic and Nash sets}\label{s2}

The content of this section is scattered in the literature. We include it (with suitable references to the literature when available) in order to provide a full vision of the tools we have at our disposal to prove the main results of this article. Given a function $f$ or a set ${\mathfrak F}$ of functions on a set $X$, we denote $\ZZ(f)=\{x\in X:\ f(x)=0\}$ and $\ZZ({\mathfrak F})=\bigcap_{f\in {\mathfrak F}}\ZZ(f)$ their respective zero sets. If $X$ is a topological space and $f_{1,x},\ldots,f_{r,x}$ are function germs at a point $x\in X$ represented by functions $f_1,\ldots,f_r$ defined on a common open neighborhood $V^x\subset X$ of $x$, we define $\ZZ(f_{1,x},\ldots,f_{r,x})$ as the set germ $\{y\in V^x:\ f_1(y)=0,\ldots,f_r(y)=0\}_x$ at $x$. Let $\K:=\R$ or $\C$ and let us endow $\K^n$ with the Euclidean topology.

\subsection{Coherence in the analytic setting}
We compare the concept of coherence for complex analytic sets and real analytic sets and analyzing its main consequences. Let $\Omega\subset\K^n$ be an open set and let $X\subset\Omega$ be a closed set. Recall that $X$ is an \emph{analytic subset of $\Omega$} if for each $x\in X$ there exists analytic function germs $f_{1,x},\ldots,f_{r,x}\in\an_{\K^n,x}$ such that $X_x=\ZZ(f_{1,x},\ldots,f_{r,x})$.

\subsubsection{Coherence of complex analytic sets.}\label{complex}
Suppose first that we are in the complex case $\K=\C$ and let $X\subset\Omega$ be a complex analytic subset of and open set $\Omega\subset\C^n$. Pick a point $x\in\Omega$ and denote $\Jj_{X,x}:=\{f_x\in\an_{\C^n,x}:\ X_x\subset\ZZ(f_x)\}\}$ the zero ideal of the complex analytic germ $X_x$. Let $f_{1,x},\ldots,f_{r,x}\in\Jj_{X,x}$ be a system of generators of the ideal $\Jj_{X,x}$ of $\an_{\C^n,x}$ and let $U^x\subset\Omega$ be an open neighborhood of $x$ such that there exist analytic functions $f_i\in\an(U^x)$ that are representatives of the analytic germs $f_{i,x}$. Oka's coherence theorems \cite[Ch.IV]{n1} imply that if we shrink $U^x$, then $\{f_{1,y},\ldots,f_{r,y}\}$ is a system of generators of $\Jj_{X,y}$ for each $y\in U^x$. This means that every complex analytic set $X\subset\Omega$ is {\em coherent}. If in addition the open set $\Omega\subset\C^n$ is Stein, then the coherence of $X$ (that we have by Oka's coherence theorems) implies: 
\begin{itemize}
\item[(i)] $X$ is the zero set of finitely many analytic functions on $\Omega$ (see \cite[Proof of Prop.15]{c2}).
\item[(ii)] The ideal of analytic functions on $\Omega$ that are identically zero on $X$ generates the ideal of analytic germs on $\Omega_x$ that vanishes identically on $X_x$ for each $x\in X$ (as a consequence of Cartan's Theorem A).
\item[(iii)] The analytic functions on $X$ (that is, the global sections of the sheaf $\an_{\C^n}/\Jj_X$) are restrictions to $X$ of analytic functions on $\Omega$ (as a consequence of Cartan's Theorem B).
\end{itemize}
Observe that the limitation to have the previous properties is of global nature, that is, $\Omega$ is a Stein space, or equivalently, $\Omega$ has a suitable shape to have `enough' analytic functions.

\subsubsection{Coherence of real analytic sets}
We provide next a more systematic approach to the concept of {\em coherence of real analytic sets} than the one provided in the Introduction. Assume first $\Omega=\R^n$ and let $X\subset\R^n$ be a $C$-analytic subset of $\R^n$, that is, the common zero set of finitely many analytic functions on $\R^n$. We will comment below in \S\ref{general} (via Whitney's analytic immersion theorem) that this assumption is not restrictive. For each open subset $U\subset\R^n$ denote the ring of analytic functions on $U$ with $\an(U)$. Let $\Ii(X):=\{f\in\an(\R^n):\ X\subset\ZZ(f)\}$ be the ideal of $\an(\R^n)$ of analytic functions vanishing identically on $X$. Consider the sheaves $\Jj_X$ and $\Ii_X$ of ideals on $\R^n$ given by the following formulas:
$$
\begin{cases}
\Jj_{X,x}:=\{f_x\in\an_{\R^n,x}:\ X_x\subset\ZZ(f_x)\},\\
\Ii_{X,x}:=\Ii(X)\an_{\R^n,x}.
\end{cases}
$$
If $U\subset\R^n$ is an open subset, then $H^0(U,\Jj_{X})=\Ii(X\cap U)=\{f\in\an(U):\ X\cap U\subset\ZZ(f)\}$, whereas $H^0(U,\Ii_X)=\Ii(X)\an(U)$. In general, $\Ii_{X,x}\subset\Jj_{X,x}$ for each $x\in X$ and $\Ii(X)\an(U)\subset\Ii(X\cap U)$ for each open subset $U\subset\R^n$.

A real analytic set $X$ is \em coherent \em if the sheaf of ideals $\Jj_X$ is coherent. Recall that a sheaf $\Ff$ of $\an_{\R^n}$-modules is \em coherent \em if:

\begin{itemize}
\item[(i)] $\Ff$ is of finite type, that is, for each $x\in\R^n$ there exist an open neighborhood $U\subset\R^n$ of $x$, $m\in\N^*$ and a surjective morphism $\an_{\R^n}^m|_U\rightarrow\Ff|_U$, and 
\item[(ii)] the kernel of each homomorphism $\an_{\R^n}^p|_V\rightarrow\Ff|_V$ is of finite type for each $p\geq1$ and each open subset $V$ of $\R^n$.
\end{itemize}

As the sheaf of rings $\an_{\R^n}$ (see \cite[\S1.2.15.Def.3]{se} for the definition of sheaf of rings) is by \cite[Prop.4]{c2} coherent, a sheaf of ideals $\Ff$ of $\an_{\R^n}$ is coherent if and only if $\Ff$ is of finite type \cite[IV.B.Prop.8]{gr}. The famous Cartan's Theorems A and B (\cite[Thm.3]{c2}) describe the local-global behavior of coherent sheaves $\Ff$ of $\an_{\R^n}$-modules: 
\begin{itemize}
\item[(A)] \em The stalks of a coherent sheaf $\Ff$ are spanned by the global sections.\em
\item[(B)] \em Each $p$-cohomology group of a coherent sheaf $\Ff$ is trivial for each $p>0$\em. 
\end{itemize}

\begin{remark}
The previous results for real analytic sets are derived from Cartan's Theorems A and B for Stein spaces and the following fact: {\em each open set $U\subset\R^n$ has by \cite[\S1]{c2} a basis of `very nice' open neighborhoods in $\C^n$}, which: are Stein open sets, are invariant under conjugation in $\C^n$ and have $U$ as a deformation retract. Thus, in the real case a sufficient condition to have the properties (i), (ii), (iii) described above in \S\ref{complex} for complex analytic sets is \em coherence\em. The substantial difference between complex and real analytic sets concerns coherence: complex analytic sets are always coherent (Oka's coherence theorems), whereas real analytic sets behave as complex analytic sets (in a Stein open subset of $\C^n$) when they are coherent.
$\hfill\sqbullet$
\end{remark}

Recall that a real analytic set $X\subset\R^n$ is pure dimensional if the analytic germs $X_x$ have the same dimension as $X$ for each $x\in X$. For each real analytic set $X\subset\R^n$ we have the following property concerning coherence:

\begin{lem}[{\cite[Prop.14\&15]{c2}}]\label{fact:cohpure} 
Each coherent analytic set $X\subset\R^n$ is a $C$-analytic set. In addition, if $X\subset\R^n$ is a $C$-irreducible coherent analytic set, then it is pure dimensional.
\end{lem}

A non-empty $C$-analytic set $X\subset\R^n$ is {\em $C$-irreducible} if it is not the union of two proper $C$-analytic subsets. We refer the reader to \cite{fe1,wb} for a detailed study about $C$-irreducibility. As a consequence of the previous lemma and Cartan's Theorem A, we have the following crucial characterization of coherent real analytic sets, which will be useful along the sequel.

\begin{cor}\label{ij}
A real analytic set $X\subset\R^n$ is coherent if and only if it is a $C$-analytic set and $\Ii_{X,x}=\Jj_{X,x}$ for each $x\in X$.
\end{cor}
\begin{remark}
Let $X\subset\R^n$ be a coherent $C$-analytic set and $U\subset\R^n$ an open subset. Then $\Ii(X\cap U)=\Ii(X)\an(U)$.

As $X\cap U$ is a coherent $C$-analytic subset of $U$, we have $\Ii(X\cap U)\an_{\R^n}|_U=\Jj_X|_U=\Ii_X|_U=\Ii(X)\an_{\R^n}|_U$, so $\Ii(X\cap U)=H^0(U,\Ii(X\cap U)\an_{\R^n}|_U)=H^0(U,\Ii(X)\an_{\R^n}|_U)=\Ii(X)\an(U)$.\hfill$\sqbullet$
\end{remark}

\subsubsection{Analytic extension property}
Let $X\subset\R^n$ be a $C$-analytic subset. The \em sheaf of analytic functions germs on $X$ \em is $\Cc^\omega_X:=\an_{\R^n}/\Jj_X$, whereas $\an_X:=\an_{\R^n}/\Ii_X$ is the \em sheaf of global analytic functions germs on $X$\em. The name of the latter follows from Cartan's Theorem B. The sheaf of ideals $\Ii_X$ is generated by global sections, so it is a coherent $\an_{\R^n}$-sheaf of ideals on $\R^n$ whose support is $X$ and in fact: {\em $\Ii_X$ is the biggest coherent $\an_{\R^n}$-sheaf of ideals whose support is $X$}.
\begin{proof} 
Let $\Ff$ be a coherent $\an_{\R^n}$-sheaf of ideals whose support is $X$. By Cartan's Theorem A the stalks $\Ff_x$ are generated by global sections of $\Ff$ and $H^0(X,\Ff)\subset\Ii(X)$. Thus, $\Ff=H^0(X,\Ff)\an_{\R^n}\subset\Ii(X)\an_{\R^n}=\Ii_X$.
\end{proof}
In particular, the $\an_{\R^n}$-sheaf of quotient rings $\an_{\R^n}/\Ii_X$ is coherent. By Cartan's Theorem B the cohomology group $H^1(\R^n,\Ii_X)=0$, so the exact sequence 
$$
0\to H^0(\R^n,\Ii_X)\to H^0(\R^n,\an_{\R^n})\to H^0(\R^n,\an_{\R^n}/\Ii_X)\to 0
$$
is short and consequently each global section of $\an_{\R^n}/\Ii_X$ is the image of a global analytic function on $\R^n$. We denote $\an(X):=H^0(\R^n,\an_{\R^n}/\Ii_X)\equiv H^0(X,(\an_{\R^n}/\Ii_X)|_X)$ the {\em ring of global analytic functions on $X$}, whereas $\Cc^\omega(X):=H^0(X,(\an_{\R^n}/\Jj_X)|_X)$ is the {\em ring of analytic functions on $X$}. At this point we reformulate what means the analytic extension property for $C$-analytic sets.

\begin{defn}[Analytic extension property]
A $C$-analytic set $X\subset\R^n$ has the \em analytic extension property \em if the homomorphism of rings $\an(\R^n)\to\Cc^\omega(X)$ is surjective, that is, each analytic function on $X$ is the restriction to $X$ of a analytic function on $\R^n$. 
\end{defn}

\subsubsection{Reduction to the case of affine spaces}\label{general}
We comment briefly why our choice of considering $C$-analytic subsets of $\R^n$ instead of $C$-analytic subsets of an open subset $\Omega\subset\R^n$ is not restrictive. 
\begin{proof}
Assume that $M$ is a real analytic manifold of dimension $m$ (this includes the case of an open subset of an affine space) and let $X\subset M$ be a $C$-analytic subset, that is, the common zero set of finitely many analytic functions $f_1,\ldots,f_r$ on $M$. By Whitney's analytic immersion theorem there exists an analytic immersion $\varphi:M\hookrightarrow\R^{2m+1}$ such that $\varphi(M)$ is closed in $\R^{2m+1}$, that is, we may assume that $M$ is a closed real analytic submanifold in $\R^n$ for some positive integer $n$. Let $(U,\rho)$ be an analytic tubular neighborhood of $M$ in $\R^n$ and define $F_i:=f_i\circ\rho$ for $i=1,\ldots,r$. As $M$ is a coherent analytic subset of $\R^n$ and $F_1,\ldots,F_r$ define global sections of the coherent sheaf $\an_{\R^n}/\Jj_M$, there exist analytic functions $G_1,\ldots,G_r\in\an(\R^n)$ such that $G_i|_M=f_i$. Let $G_{r+1}$ be an analytic equation of $M$ in $\R^n$. We have $X=\ZZ(G_1,\ldots,G_r,G_{r+1})$ is a $C$-analytic subset of $\R^n$. Thus, in the following we only treat the case of $C$-analytic subsets of $\R^n$.
\end{proof}

\subsection{Coherence in the Nash setting}\label{nasht}
Next, we turn out to the Nash case. We have already defined in the Introduction semialgebraic sets, Nash functions and Nash sets. A semialgebraic set $M\subset\R^n$ is called an \em (affine) Nash manifold \em if it is in addition a smooth submanifold of (an open semialgebraic subset of) $\R^n$. By \cite[Prop.8.1.8 \& Cor.9.3.10]{bcr} Nash manifolds coincide with analytic submanifolds of (an open semialgebraic subset of) $\R^n$ that are theirselves semialgebraic sets. Let $U\subset\R^n$ be an open semialgebraic set and let us denote the ring of Nash functions on $U$ with $\Nn(U)$. Let $S\subset\R^n$ be a semialgebraic set. We say that a function $f:S\to\R$ is a \em Nash function \em if there exist an open semialgebraic neighborhood $U\subset\R^n$ of $S$ and a Nash extension $F:U\to\R$ of $f$. We denote $\Nn(S)$ the ring of Nash functions on $S$ (see also Definition \ref{nashs}, Remark \ref{nashsr}, \S\ref{sss} and \eqref{ssseq}). At this point we prove (see also \cite[Prop.2.8(ii)]{fgr}) that Nash subsets of an open semialgebraic subset $U\subset\R^n$ coincide with those $C$-analytic subset of $U$ that are semialgebraic subsets of $\R^n$. This means that the Nash category corresponds to the amalgamation of the (global) analytic and semialgebraic categories.

\begin{lem}\label{ncas}
Let $U\subset\R^n$ be an open semialgebraic set and let $X\subset U$. The following conditions are equivalent:
\begin{itemize}
\item[(i)] $X$ is a Nash subset of $U$.
\item[(ii)] $X$ is a $C$-analytic subset of $U$ that is a semialgebraic subset of $\R^n$. 
\end{itemize}
\end{lem}
\begin{proof}
The implication (i) $\Longrightarrow$ (ii) is clear, because Nash functions are analytic functions \cite[Prop.8.1.8]{bcr}, so let us prove the converse implication. We claim: {\em We may assume that $X$ is a $C$-irreducible $C$-analytic subset of $U$.} 

By \cite[Prop.2.9.10]{bcr} $X$ is a finite pairwise disjoint union of connected Nash manifolds $M_i$, each one Nash diffeomorphic to $(0,1)^{\dim(M_i)}$. By \cite[\S6.Prop.8 \& Prop.9]{wb} there exist the locally finite family $\{X_j\}_{j\in J}$ of $C$-irreducible $C$-analytic components of $X$. By the identity principle each Nash manifold $M_i$ is contained in some $X_j$. As the $M_i$ are finitely many, the $X_j$ are finitely many and each one is a finite union of some of the Nash manifolds $M_i$. Consequently, each $X_j$ is a semialgebraic set. Thus, we may assume that $X$ is a $C$-irreducible $C$-analytic subset of $U$.

Let $Y$ be the smallest Nash set that contains $X$, which exists, because $\Nn(U)$ is noetherian \cite[Thm.8.7.18]{bcr}. As $X$ is $C$-irreducible, $Y$ is an irreducible Nash set, so the zero ideal $\Ii^\bullet(Y):=\{f\in\Nn(U): X\subset\ZZ(f)\}$ of $Y$ in $\Nn(U)$ is prime. By \cite[Prop.2.8(i)]{fgr} the ideal $\Ii(Y)$ is generated by $\Ii^\bullet(Y)$ in $\an(U)$, that is, $\Ii(Y)=\Ii^\bullet(Y)\an(U)$. By \cite[Cor.2]{cs} the ideal $\Ii(Y)=\Ii^\bullet(Y)\an(U))$ is a prime ideal of $\an(U)$, so $Y$ is a $C$-irreducible $C$-analytic subset of $U$. As $X$ is semialgebraic, we have $\dim(X)=\dim(Y)$, so $Y=X$, because both are irreducible Nash sets. Consequently, $X$ is a Nash subset of $U$, as required. 
\end{proof}

\subsubsection{Coherence of Nash sets}
Assume first we work in the open semialgebraic set $\Omega=\R^n$. We will comment below in \S\ref{omegarn} (via Mostowski's Lemma \cite[Lem.6]{m2}) that this assumption is not restrictive. We recover the definition of Nash function on a Nash subset $X$ of $\R^n$ via global sections of a certain sheaf \cite[\S2.B]{bfr}. Let $\Jj_X^\bullet$ be the sheaf of $\Nn_{\R^n}$-ideals given by $\Jj_{X,x}^\bullet:=\{f_x\in\Nn_{\R^n,x}:\ X_x\subset\ZZ(f_x)\}$, which is the ideal of germs of Nash functions in $\R^n_x$ vanishing identically on $X_x$ for each $x\in\R^n$ and whose support is $X$. As in the analytic case, the sheaf $\Nn_{\R^n}/\Jj_X^\bullet$ has in general a bad local-global behavior and the ring of its global sections does not coincide necessarily with the ring of the restrictions to $X$ of Nash functions on $\R^n$.

As $\Nn_{\R^n}$ is a coherent sheaf of rings (see \cite[\S1.2.15.Def.3]{se} and \cite[(I.6.6)]{s}), a sheaf $\Ff$ of $\Nn_{\R^n}$-ideals is coherent if and only if $\Ff$ is of finite type \cite[IV.B.Prop.8]{gr}. When dealing with sheaves of $\Nn_{\R^n}$-ideals, coherence is not enough in general to have results in the vein of Cartan's Theorems (A) and (B) (see the Introduction of \cite{crs3}), because semialgebraicity involves a finiteness phenomenon that coherence alone does not capture. We say that a sheaf $\Ff$ of $\Nn_{\R^n}$-ideals is \em finite \em if there exists a finite open semialgebraic covering $\{U_i\}_{i=1}^r$ of $\R^n$ such that each restriction $\Ff|_{U_i}$ is generated by finitely many Nash functions over $U_i$. As a finite sheaf of $\Nn_{\R^n}$-ideals is of finite type, we deduce that it is coherent. In \cite{cs,crs2} the authors prove that this is the correct notion for sheaves of $\Nn_{\R^n}$-ideals to obtain results in the vein of Cartan's Theorems (A) and (B) (in general for $\Nn_{\R^n}$-modules the right notion is \em strong coherence \em \cite{crs3}). Recall that a coherent sheaf $\Mm$ of $\Nn_{\R^n}$-modules is {\em strongly coherent} if there exists an exact sequence $\Nn_{\R^n}^q\to\Nn_{\R^n}^p\to\Mm\to 0$. We suggest the survey \cite{crs1} as a general reference.

\begin{thm}[{\cite[Thm.1, Cor.2]{cs}, \cite[Prop.0.6, Prop.0.7]{crs2}}]\label{factfinitesheaf}
Let $\Ff$ be a finite sheaf of $\Nn_{\R^n}$-ideals and $x\in\R^n$. Then: 
\begin{itemize}
\item[(A)] The stalk $\Ff_x$ is generated by global sections of $\Ff$. 
\item[(B)] Every global section of $\Nn_{\R^n}/\Ff$ is represented by a Nash function on $\R^n$. 
\end{itemize}
\end{thm}

Let $X\subset\R^n$ be a Nash set and denote the (finitely generated) ideal $\{f\in\Nn(\R^n): X\subset\ZZ(f)\}$ of Nash functions vanishing identically on $X$ with $\Ii^\bullet(X)$. In the following the sheaf of $\Nn_{\R^n}$-ideals $\Ii_{X,x}^\bullet:=\Ii^\bullet(X)\Nn_{\R^n,x}$ will play a central role, because it is the both the biggest finite and the biggest coherent $\Nn_{\R^n}$-sheaf of ideals whose support is $X$ (Corollary \ref{bfbc}). The proof of this fact takes advantage of the faithfully flatness of the inclusion homomorphism $\Nn_{\R^n,x}\hookrightarrow\an_{\R^n,x}$ for each $x\in\R^n$. Let us recall how one prove this latter fact, which will be useful several times. 

\begin{lem}\label{ff}
Let $(A,\gtm_A)$ and $(B,\gtm_B)$ be local excellent rings and let $\varphi:A\hookrightarrow B$ be an inclusion homomorphism such that $\varphi(\gtm_A)\subset\gtm_B$. Assume that the completions $\widehat{A}$ and $\widehat{B}$ coincide. Then $\varphi$ is faithfully flat.
\end{lem}
\begin{proof}
As $A$ and $B$ are excellent local rings, the inclusions homomorphisms $\phi:A\hookrightarrow\widehat{A}$ and $\psi:B\hookrightarrow\widehat{B}$ are by \cite[(33.A), (34.A)]{mt} regular. As $\widehat{A}=\widehat{B}$, we may assume $\phi=\psi\circ\varphi$. As both $\phi$ and $\psi$ are regular homomorphisms, we deduce by \cite[(33.B) Lem.1]{mt} that $\varphi$ is regular, so it is flat. As $A$ and $B$ are local rings and $\varphi(\gtm_A)\subset\gtm_B$, we deduce by \cite[(4.D) Thm.3]{mt} that $\varphi$ is faithfully flat, as required.
\end{proof}
\begin{cor}\label{ffna}
The inclusion homomorphism $\Nn_{\R^n,x}\hookrightarrow\an_{\R^n,x}$ is faithfully flat.
\end{cor}
\begin{proof}
As $\an_{\R^n,x}\cong\R\{\x_1,\ldots,\x_n\}$ and $\Nn_{\R^n,x}$ is by \cite[Cor.8.1.6]{bcr} $\R$-isomorphic to the ring of algebraic formal series $\R[[\x_1,\ldots,\x_n]]_{\rm alg}$, we deduce by \cite[Ch.VII.Ex.2.3]{abr} that both rings are local excellent rings. Their common completion is $\R[[\x_1,\ldots,\x_n]]$ and their maximal ideals are generated by the variables $\x_1,\ldots,\x_n$. By Lemma \ref{ff} we conclude that the inclusion homomorphism $\Nn_{\R^n,x}\hookrightarrow\an_{\R^n,x}$ is faithfully flat, as required. 
\end{proof}

\begin{cor}\label{bfbc}
The sheaf of $\Nn_{\R^n}$-ideals given by $\Ii_{X,x}^\bullet:=\Ii^\bullet(X)\Nn_{\R^n,x}$ for each $x\in X$ is both: 
\begin{itemize}
\item[(i)] The biggest finite sheaf of $\Nn_{\R^n}$-ideals whose support is $X$.
\item[(ii)] The biggest coherent sheaf of $\Nn_{\R^n}$-ideals whose support is $X$. 
\end{itemize}
In addition, $\Ii_X=\Ii_X^\bullet\an_{\R^n}$.
\end{cor}
\begin{proof}
(i) Let $\Ff$ be a finite sheaf of $\Nn_{\R^n}$-ideals whose support is $X$. By Theorem \ref{factfinitesheaf}(A) the stalks $\Ff_x$ are generated by global sections of $\Ff$. As the global sections of $\Ff$ vanish identically on $X$, we have $H^0(X,\Ff)\subset\Ii^\bullet(X)$, so $\Ff_x=H^0(X,\Ff)\Nn_{\R^n,x}\subset\Ii^\bullet(X)\Nn_{\R^n,x}=\Ii_{X,x}^\bullet$ for each $x\in\R^n$.

(ii) Let $\Ff$ be a coherent sheaf of $\Nn_{\R^n}$-ideals whose support is $X$. Then $\Ff\an_{\R^n}$ is a coherent sheaf of $\an_{\R^n}$-ideals whose support is $X$ as well, so $H^0(X,\Ff\an_{\R^n})\subset\Ii(X)=\Ii^\bullet(X)\an(\R^n)$ by \cite[Prop.2.8(i)]{fgr}. By Cartan's Theorem A 
$$
\Ff_x\an_{\R^n,x}=H^0(X,\Ff\an_{\R^n})\an_{\R^n,x}\subset(\Ii^\bullet(X)\Nn_{\R^n,x})\an_{\R^n,x}=\Ii^\bullet_{X,x}\an_{\R^n,x}
$$ 
for each $x\in\R^n$. As the monomorphism $\Nn_{\R^n,x}\hookrightarrow\an_{\R^n,x}$ is by Corollary \ref{ffna} faithful flat, we conclude $\Ff_x\subset\Ii^\bullet_{X,x}$ for each $x\in X$, as required. 

The latter assertion holds, because by \cite[Prop.2.8(i)]{fgr}
$$
\Ii_X=\Ii(X)\an_{\R^n}=\Ii^\bullet(X)\an(\R^n)\an_{\R^n}=\Ii^\bullet(X)\an_{\R^n}=\Ii^\bullet(X)\Nn_{\R^n,x}\an_{\R^n}=\Ii_X^\bullet\an_{\R^n},
$$
as required.
\end{proof} 

\subsubsection{Nash coherence versus analytic coherence}
In this framework it is natural to ask when $X$ is \em coherent as a Nash set\em, that is, when $\Jj_X^\bullet$ is a coherent sheaf of $\Nn_{\R^n}$-ideals. The following result borrowed from \cite[\S2.B]{bfr} establishes the equivalence of Nash coherence and analytic coherence for Nash sets. 

\begin{lem}\label{ijn}
Let $X\subset\R^n$ be a Nash set. The following conditions are equivalent:
\begin{itemize}
\item[(i)] $X$ is coherent as a Nash set.
\item[(ii)] $X$ is coherent as a $C$-analytic set.
\item[(iii)] $\Jj_X^\bullet=\Ii_X^\bullet$.
\end{itemize}
\end{lem}
\begin{proof}
(i) $\Longrightarrow$ (ii) Suppose first $X$ is a coherent Nash set. Recall that $\Jj_{X,x}=\Jj_{X,x}^\bullet\an_{\R^n,x}$ (see \cite[Proof of Prop.8.6.9]{bcr}) and the ideal $\Ii(X\cap U)$ is generated by $\Ii^\bullet(X\cap U)$ for each open semialgebraic set $U$ of $\R^n$ (see \cite[Prop.2.8(i)]{fgr}). As $\Jj_X^\bullet$ is of finite type, for each $x\in\R^n$ there exists an open semialgebraic neighborhood $U^x\subset\R^n$ such that $\Jj_{X}^\bullet|_{U^x}=\Ii^\bullet(X\cap U^x)\Nn_{\R^n}|_{U^x}$. The sheaf of $\an_{\R^n}$-ideals $\Ff$ given by $\Ff|_{U^x}:=\Ii(X\cap U^x)\an_{\R^n}|_{U^x}=\Ii^\bullet(X\cap U^x)\an_{\R^n}|_{U^x}$ for each $x\in\R^n$ is of finite type (because the ideal $\Ii^\bullet(X\cap U^x)$ is finitely generated, as $\Nn(U^x)$ is a Noetherian ring) and its support is $X$. Thus, $\Ff$ is coherent and by Cartan's Theorem A $\Ff_x=\Ii(X)\an_{\R^n,x}=\Ii_{X,x}$. Putting all together:
$$
\Jj_{X,x}=\Jj_{X,x}^\bullet\an_{\R^n,x}=(\Ii^\bullet(X\cap U^x)\Nn_{\R^n,x})\an_{\R^n,x}=\Ii^\bullet(X\cap U^x)\an_{\R^n,x}=\Ff_x=\Ii_{X,x}
$$
for each $x\in X$, so $X$ is a coherent analytic set.

(ii) $\Longrightarrow$ (iii) Suppose next $X$ is a coherent $C$-analytic set. Consider the ideal $\Ii(X)$ of all analytic functions in $\an(\R^n)$ vanishing identically on $X$. By Cartan's Theorem A the stalks $\Jj_{X,x}$ are generated by the global sections $H^0(\R^n,\Jj_X)=\Ii(X)$. As $\Ii(X)$ is generated by $\Ii^\bullet(X)$ in $\an(\R^n)$ (see \cite[Prop.2.8(i)]{fgr}), it follows 
$$
\Ii_{X,x}^\bullet\an_{\R^n,x}=\Ii^\bullet(X)\an_{\R^n,x}=\Ii(X)\an_{\R^n,x}=\Jj_{X,x}\supset\Jj^\bullet_{X,x}\an_{\R^n,x}\supset\Ii_{X,x}^\bullet\an_{\R^n,x},
$$ 
so $\Ii_{X,x}^\bullet\an_{\R^n,x}=\Jj_{X,x}^\bullet\an_{\R^n,x}$. The monomorphism $\Nn_{\R^n,x}\hookrightarrow\an_{\R^n,x}$ is by Corollart \ref{ffna} faithful flat, so $\Ii_{X,x}^\bullet=\Jj_{X,x}^\bullet$ for each $x\in\R^n$, that is, $\Jj_X^\bullet=\Ii_X^\bullet$.

(iii) $\Longrightarrow$ (i) As $\Jj_X^\bullet=\Ii_X^\bullet$ is a coherent sheaf of $\Nn_{\R^n}$-ideals, $X$ is coherent as a Nash set. 
\end{proof}

\begin{remark}\label{generalnash}
If $X\subset\R^n$ is a coherent Nash set, then $\Ii^\bullet(X)\Nn(U)=\Ii^\bullet(X\cap U)$ for each semialgebraic open subset $U\subset\R^n$, where $\Ii^\bullet(X\cap U)=\{f\in\Nn(U):\ X\cap U\subset\ZZ(f)\}$. 

As $X\cap U$ is a coherent Nash subset of $U$, we have $\Ii^\bullet(X\cap U)\Nn_{\R^n}|_U=\Jj^\bullet_{X\cap U}=\Jj^\bullet_X|_U=\Ii_X^\bullet|_U=\Ii^\bullet(X)\Nn_{\R^n}|_U$. By \cite[I.6.5]{s} we have 
$$
\Ii^\bullet(X\cap U)=H^0(U,\Ii^\bullet(X\cap U)\Nn_{\R^n}|_U)=H^0(U,\Ii^\bullet(X)\Nn_{\R^n}|_U)=\Ii^\bullet(X)\Nn(U),
$$ 
as required.\hfill$\sqbullet$
\end{remark}

\subsubsection{Nash extension property}
Roughly speaking, the coherent sheaf $\Nn_{\R^n}/\Ii_X^\bullet$ captures the global notions concerning Nash functions, whereas $\Nn_{\R^n}/\Jj_X^\bullet$ the local ones. Define $\Cc^{\scriptscriptstyle\Nn}(X):=H^0(X,(\Nn_{\R^n}/\Jj^\bullet_X)|_X)$, which is the ring of local Nash functions on a Nash set $X$ of $\R^n$. We will see in Lemma \ref{ens} that the elements of $\Cc^{\scriptscriptstyle\Nn}(X)$ are semialgebraic functions on $X$. We reformulate what means the Nash extension property for Nash sets.

\begin{defn}
A Nash set $X\subset\R^n$ has the \em Nash extension property \em if the homomorphism $\Nn(X)\to\Cc^{\scriptscriptstyle\Nn}(X)$ is surjective, that is, each local Nash function on $X$ is the restriction to $X$ of a Nash function on $\R^n$.
\end{defn}

\subsubsection{Reduction to the case of affine spaces}\label{omegarn}
As in the real analytic case, we comment briefly why our choice of considering Nash subsets of $\R^n$ instead of Nash subsets of an open semialgebraic subset $\Omega\subset\R^n$ is not restrictive. 
\begin{proof}
Let $M\subset\R^n$ be a Nash manifold and let $X\subset M$ be a Nash subset, that is, the common zero set of finitely many Nash functions $f_1,\ldots,f_r$ on $M$. As $M$ is locally compact, $U:=\R^n\setminus(\cl(M)\setminus M)$ is an open semialgebraic subset of $\R^n$. By Mostowski's Lemma \cite[Lem.6]{m2} there exists a continuous semialgebraic function $h:\R^n\to\R$ such that $h|_U$ is a strictly positive Nash function and $\ZZ(h)=\R^m\setminus U=\cl(M)\setminus M$. The image of the Nash map $\varphi:U\to\R^{m+1},\ x\mapsto\big(x,\frac{1}{h(x)}\big)$ is the closed Nash manifold $T:=\varphi(U)=\{(x,y)\in\R^{n+1}:\ yh(x)=1\}$ and the projection $\pi:\R^{n+1}\to\R^n,\ (x,x_{n+1}):=(x_1,\ldots,x_n,x_{n+1})\to x$ induces by restriction a Nash map $\rho:=\pi|_T:T\to U$. Observe that $\varphi:U\to T$ and $\rho:T\to U$ are mutually inverse diffeomorphisms. Thus, as $M$ is closed in $U$, also $N:=\varphi(M)$ is closed in $\R^{n+1}$, so we may assume from the beginning that $M$ is a closed Nash submanifold in $\R^n$ for some positive integer $n$. Let $(U,\rho)$ be a Nash tubular neighborhood of $M$ in $\R^n$ and define $F_i:=f_i\circ\rho$ for $i=1,\ldots,r$. As $M$ is a coherent Nash subset of $\R^n$ and $F_1,\ldots,F_r$ define global sections of the coherent sheaf $\Nn_{\R^n}/\Jj_M^\bullet$, there exist by Theorem \ref{factfinitesheaf}(i), Corollary \ref{bfbc} and Lemma \ref{ijn} Nash functions $G_1,\ldots,G_r\in\Nn(\R^n)$ such that $G_i|_M=f_i$. Let $G_{r+1}$ be a Nash equation of $M$ in $\R^n$. We have $X=\ZZ(G_1,\ldots,G_r,G_{r+1})$ is a Nash subset of $\R^n$.
\end{proof}

\subsubsection{Cartan's Theorems {\em(A)} and {\em(B)} in the Nash setting.}
Let us summarize a part of the information above in the following Cartan's style result:
\begin{thm}\label{summer}
Let $\Omega\subset\R^n$ be an open semialgebraic set and let $X\subset\Omega$ be a coherent Nash subset of $\Omega$. Then:
\begin{itemize}
\item[(A)] The ideal $\Ii^\bullet(X)$ of Nash functions on $\Omega$ that are identically zero on $X$ generates the ideal $\Jj^\bullet_{X,x}$ of Nash germs on $\Omega_x$ that vanishes identically on $X_x$ for each $x\in X$.
\item[(B)] The Nash functions on $X$, that is, the global sections of the sheaf $\Nn_{\R^n}/\Jj^\bullet_X$ are restrictions to $X$ of Nash functions on $\Omega$.
\end{itemize}
\end{thm}

We recall here (by the sake of completeness) the identity $\Nn(X)=H^0(\R^n,\Nn_{\R^n}/\Ii_X^\bullet)$.
\begin{cor}\label{iso}
There exist a natural isomorphism between the ring $\Nn(X)$ of Nash functions on $X$ and the ring of global sections of the sheaf $\Nn_{\R^n}/\Ii_X^\bullet$, that is, $\Nn(X)=H^0(\R^n,\Nn_{\R^n}/\Ii_X^\bullet)$.
\end{cor}
\begin{proof}
By Remark \ref{generalnash} we have $\Ii^\bullet(X\cap V)=\Ii^\bullet(X)\Nn(V)$ for each open semialgebraic subset $V\subset\R^n$. Then by Theorem \ref{summer}(B) we deduce $H^0(V,(\Nn_{\R^n}/\Ii_X^\bullet)|_V)=\Nn(V)/\Ii^\bullet(X)\Nn(V)$ for each open semialgebraic subset $V\subset\R^n$. The support of $\Nn_{\R^n}/\Ii_X^\bullet$ is $X$, so for each open semialgebraic neighborhood $U$ of $X$ in $\R^n$ we have $H^0(X,(\Nn_{\R^n}/\Ii_X^\bullet)|_X)\equiv H^0(U,(\Nn_{\R^n}/\Ii_X^\bullet)|_U)\equiv H^0(\R^n,\Nn_{\R^n}/\Ii_X^\bullet)$. Fix an open semialgebraic neighborhood $U$ of $X$ in $\R^n$. As 
$$
\Nn(\R^n)/\Ii^\bullet(X)\equiv H^0(\R^n,\Nn_{\R^n}/\Ii_X^\bullet)\equiv H^0(U,\Nn_{\R^n}/\Ii_X^\bullet|_U)\equiv\Nn(U)/\Ii^\bullet(X)\Nn(U),
$$ 
we deduce that for each Nash function $g:U\to\R$ there exists a Nash function $h:\R^n\to\R$ such that $g-h|_U\in\Ii^\bullet(X)\Nn(U)$, so $g=h$ on $X$. In particular, \em each function $f\in\Nn(X)$ has a Nash extension to $\R^n$\em, so that the restriction homomorphism $\Nn(\R^n)\to\Nn(X),\ f\mapsto f|_X$ is surjective. Consequently, $\Nn(X)=\Nn(\R^n)/\Ii^\bullet(X)\equiv H^0(X,\Nn_{\R^n}/\Ii_X^\bullet)$, as required.
\end{proof}

\section{Set of tails and set of points of non-coherence. Examples}\label{s3}

In this section we recall the description of the set $N(X)$ of points of non-coherence of a $C$-analytic set $X$ provided in \cite[\S5.1]{abf2} and we analyze the set $T(X)$ of `tails' of $X$. The characterization of $N(X)$ is inspired by \cite{t2,tt} and the description of both $N(X)$ and $T(X)$ involve an invariant complexification $\widetilde{X}$ of $X$ and an invariant normalization $(Y,\pi)$ of $\widetilde{X}$. We also refer the reader to \cite{fn,gal} for further information concerning $N(X)$. In addition, we analyze the corresponding concepts in the Nash setting. 

Denote $\sigma:=\sigma_n:\C^n\to\C^n,\ z:=(z_1,\ldots,z_n)\mapsto\ol{z}:=(\ol{z}_1,\ldots,\ol{z}_n)$ the complex conjugation in $\C^n$. Recall that a subset $A\subset\C^n$ is {\em invariant} (under conjugation) if $\sigma_n(A)=A$ and a function $F:\Omega\to\C$ (defined on an open subset $\Omega\subset\C^n$) is {\em invariant} (under conjugation) if $\Omega$ is invariant and $\ol{F\circ\sigma_n}=F$. We recall some concepts involved in the descriptions of the set $N(X)$ of points of non-coherence and the set $T(X)$ of `tails' of a $C$-analytic set.

\subsection{Involved concepts for the analytic setting}
We analyze first the concept of {\em complexification} of a $C$-analytic set $X$. 

\subsubsection{Complexification of a \texorpdfstring{$C$}{C}-analytic set}\label{ccas0}
Let us recall briefly how the {\em complexification} of $X$ is constructed. Consider the coherent sheaves of $\an_{\C^n}$-ideals $\Ii_X\otimes_\R\C$ on $\R^n$. By \cite[Prop.2]{c2} there exists an open neighborhood $\Omega\subset\C^n$ of $\R^n$ and a coherent sheaf $\Ff$ of $\an_{\C^n}$-ideals on $\Omega$ such that $\Ff|_{\R^n}=\Ii_X\otimes_\R\C=\Ii(X)\an_{\C^n}$. Shrinking $\Omega$ if necessary, we may assume that the support of $\Ff$ is an invariant (under conjugation in $\C^n$) complex analytic set $\widetilde{X}\subset\Omega$ such that its fixed part (with respect to conjugation in $\C^n$) is $X$. Recall that $\widetilde{X}$ is called a \em complexification \em of $X$. By \cite{c2} (see also \cite[Thm.1.2.12]{abf5}) two complexifications of $X$ coincide in an open neighborhood of $\R^n$ in $\C^n$ and shrinking $\Omega$ we may always assume that it is an invariant Stein open subset of $\C^n$ and $\widetilde{X}\subset\C^n$ is an invariant Stein complex analytic set. 

Denote the sheaf of zero $\an_{\C^n}$-ideals of $\widetilde{X}$ with $\Jj_{\widetilde{X}}^\C$ (we use the superindex $\C$ to stress that we are considering a complex analytic set), that is, $\Jj_{\widetilde{X},x}^\C$ is the ideal of complex analytic germs $f_x\in\an_{\C^n,x}$ such that $\widetilde{X}_x\subset\ZZ(f_x)$. Denote $\Jj^\C(\widetilde{X}):=H^0(\Omega,\Jj_{\widetilde{X}}^\C)$ and as $\widetilde{X}$ is a Stein complex analytic set, by Cartan's Theorem A $\Jj_{\widetilde{X},x}^\C=\Jj^\C(\widetilde{X})\an_{\C^n,x}$ for each $x\in\widetilde{X}$. Observe that $\widetilde{X}\cap\R^n=X$, because $\widetilde{X}$ is the support of $\Ff$ and $\Ff|_{\R^n}=\Ii(X)\an_{\C^n}|_{\R^n}$. As $\widetilde{X}$ is invariant, $\Jj^\C(\widetilde{X})$ is generated by invariant analytic functions that vanish identically on $X$. By R\"uckert's Nullstellensatz \cite[Thm.3.4.4]{jp}
$$
\Ii(X)\an_{\C^n,x}=\Ff_x\subset\sqrt{\Ff_x}=\Jj_{\widetilde{X},x}^\C=\Jj^\C(\widetilde{X})\an_{\C^n,x}\subset\Ii(X)\an_{\C^n,x}
$$
for each $x\in X$. Consequently, all the previous inclusion are equalities
\begin{equation}\label{icc}
\Jj^\C(\widetilde{X})\an_{\C^n,x}=\Jj_{\widetilde{X},x}^\C=\Ii(X)\an_{\C^n,x}=\Ff_x
\end{equation} 
for each $x\in X$. This means in particular that $\Ff_x$ is a radical ideal for each $x\in X$ and that we may assume that the sheaf of $\an_{\C^n}$-ideals $\Ff$ (chosen at the begining) is $\Jj_{\widetilde{X}}$.

Denote $\an_{\widetilde{X}}:=\an_{\C^n}/\Jj_{\widetilde{X}}^\C$ and define $\an(\widetilde{X}):=H^0(\Omega,\an_{\Omega}/\Jj_{\widetilde{X}}^\C)$ and $\an(\Omega)=H^0(\Omega,\an_{\Omega})$. As $\widetilde{X}$ is a Stein complex analytic set, Cartan's Theorem B implies that $\an(\widetilde{X})=\an(\Omega)/\Jj^\C(\widetilde{X})$.

We refer to \cite{abf5,c2,t1,wb} for a deeper study of the complexification of a $C$-analytic set. A crucial property of the complexification of a $C$-analytic subset of $\R^n$ is the following.
\begin{lem}\label{fx}
Let $X\subset\R^n$ be a $C$-analytic set and let $\widetilde{X}\subset\C^n$ be a complexification of $X$. Let $f:\R^n\to\R$ be an analytic function such that $f|_X=0$. Let $F$ be an invariant analytic extension of $f$ to an invariant open neighborhood $\Omega\subset\C^n$ of $\R^n$. Then there exists an invariant open neighborhood $\Theta\subset\Omega$ of $\R^n$ such that $F|_{\widetilde{X}\cap\Theta}=0$.
\end{lem}
\begin{proof}
Shrinking $\Omega$ if necessary, we may assume $\widetilde{X}\cap\Omega$ is a closed subset of $\Omega$. The complex analytic set $Z:=\{z\in\widetilde{X}\cap\Omega:\ F(z)=0\}$ satisfies $\Jj_{Z,x}^\C=\Ii(X)\an_{\R^n,x}\otimes_\R\C=\Jj_{\widetilde{X},x}^\C$ for all $x\in X$, because $f=F|_X\in\Ii(X)$. Consequently, by Lemma \ref{fx} there exists an open neighborhood $\Theta\subset\Omega$ of $\R^n$ such that $Z\cap\Theta=\widetilde{X}\cap\Theta$. Consequently, $F|_{\widetilde{X}\cap\Theta}=0$, as required.
\end{proof}

\subsubsection{Invariant normalization in the \texorpdfstring{$C$}{C}-analytic setting}
Let $X\subset\R^n$ be a $C$-analytic set. Denote the restriction to $\widetilde{X}$ of the complex conjugation on $\C^n$ with $\sigma:\widetilde{X}\to\widetilde{X}$. It holds $d:=\dim(X)=\dim_\C(\widetilde{X})$ and $X=\widetilde{X}^\sigma:=\{x\in\widetilde{X}:\ \sigma(x)=x\}$. Let $\pi:Y\to\widetilde{X}$ be the normalization of $\widetilde{X}$. As $\widetilde{X}$ is Stein, also $Y$ is Stein \cite{n3}. The complex conjugation of $\widetilde{X}$ extends to an anti-involution $\widehat{\sigma}$ on $Y$ that makes the following diagram commutative \cite[IV.3.10]{gmt}
$$
\xymatrix{
&Y^{\widehat{\sigma}}\ar[d]_{\pi|_{Y^{\widehat{\sigma}}}}\ar@{^{(}->}[r]&Y\ar[r]^{\widehat{\sigma}}\ar[d]_{\pi}&Y\ar[d]^{\pi}\\
X\ar@{=}[r]&\widetilde{X}^{\sigma}\ar@{^{(}->}[r]&\widetilde{X}\ar[r]^{\sigma}&\widetilde{X}
}
$$
where $Y^{\widehat{\sigma}}:=\{y\in Y:\ \widehat{\sigma}(y)=y\}$ is the set of fixed points of $\widehat{\sigma}$.

\subsubsection{Regular and singular points in the analytic setting}\label{rspas}
We say that $x$ is a \em regular point of a $C$-analytic set $X\subset\R^n$ \em if $\an(X)_{\gtm_x}$ is a regular local ring (where $\gtm_x$ is the maximal ideal of $\an(\R^n)$ associated to $x$). Pick a point $x$ of $X$ and assume for simplicity $x=0$. Thus, $\an_{\R^n,x}$ coincides with the ring $\R\{\x\}$ of real convergent power series in $n$ variables. Its completion is the ring $\R[[\x]]$ of real formal power series in $n$ variables, which is also the completion of the local ring $\an(\R^n)_{\gtm_x}$. The completions of the local rings $\an_{\R^n,x}/\Ii_{X,x}=\an_{\R^n,x}/(\Ii(X)\an_{\R^n,x})$ and $\an(X)_{\gtm_x}=(\an(\R^n)/\Ii(X))_{\gtm_x}$ are by \cite[(23.I) Thm.54, (23.L) Thm.55, Cor.3]{mt} both isomorphic to $\R[[\x]]/(\Ii(X)\R[[\x]])$. Let $\widetilde{X}$ be a complexification of $X$. The completion of the local ring $\an(\widetilde{X})_{\gtn_x}=(\an(\Omega)/\Jj^\C(\widetilde{X}))_{\gtn_x}$ (where $\gtn_x$ is the maximal ideal of $\an(\Omega)$ associated to $x$) is by \cite[(23.I) Thm.54, (23.L) Thm.55, Cor.3]{mt} and \eqref{icc}
$$
\C[[\x]]/(\Jj^\C(\widetilde{X})\C[[\x]])=\C[[\x]]/(\Ii(X)\C[[\x]])=(\R[[\x]]/(\Ii(X)\R[[\x]]))\otimes_\R\C,
$$
whereas the completion of the local ring $\an_{\widetilde{X},x}=\an_{\C^n,x}/\Jj_{\widetilde{X},x}^\C$ is by \cite[(23.I) Thm.54, (23.L) Thm.55, Cor.3]{mt} and \eqref{icc}
\begin{equation*}
\C[[\x]]/(\Jj_{\widetilde{X},x}^\C\C[[\x]])=\C[[\x]]/(\Ii(X)\C[[\x]])=(\R[[\x]]/(\Ii(X)\R[[\x]]))\otimes_\R\C.
\end{equation*}
Thus, the completions of the local rings $\an(\widetilde{X})_{\gtn_x}=(\an(\Omega)/\Jj^\C(\widetilde{X}))_{\gtn_x}$ and $\an_{\widetilde{X},x}$ coincide. In addition, they are obtained tensoring by $\C$ the common completions of $\an(X)_{\gtm_x}$ and $\an_{\R^n,x}/\Ii_{X,x}$.

All the local rings above are by \cite[Ch.VII.Ex.2.3]{abr} excellent, so the inclusions of such local rings in their respective completions are faithfully flat, regular morphisms \cite[(33.A), (34.A)]{mt}. In addition, the homomorphism $\R\hookrightarrow\C$ is faithfully flat and regular, so if one of the excellent local rings $
\an(X)_{\gtm_x}$, $\an_{\R^n,x}/\Ii_{X,x}$, $\an_{\widetilde{X},x}$, $\an(\widetilde{X})_{\gtn_x}$ is regular, all the other are regular too \cite[(33.B), Lem.2, (33.E), Lem.4]{mt}. Thus, the \em set of regular points \em of $X$ is $\Reg(X)=\Reg(\widetilde{X})\cap X$ and the \em singular locus of $X$ \em is $\Sing(X):=X\setminus\Reg(X)=\Sing(\widetilde{X})\cap X$. The connected components of the open subset $\Reg(X)$ of $X$ is a finite union of real analytic manifolds (possibly of different dimensions). As $\Sing(\widetilde{X})$ is a complex analytic set of strictly smaller dimension than $\widetilde{X}$, we deduce $\Sing(X)$ is a $C$-analytic set of strictly smaller dimension than $X$.

\subsection{Involved concepts for the Nash setting}\label{icns}
In the Nash setting the constructions above are slightly more intricate, because we need to take care of the semialgebraicity of the involved objects. We analyze the concept of (semialgebraic) {\em complexification} of a Nash set $X\subset\R^n$. To lighten the presentation we postpone the proofs of this part until Appendix \ref{A}.

\subsubsection{Semialgebraic complexification of a Nash set}\label{scns0}
Let $X\subset\R^n$ be a Nash set. Then there exists an open semialgebraic neighborhood $W\subset\C^n$ (endowed with its underlying real structure) of $X$ and an invariant complexification of $X\subset W$, which is both a complex analytic subset of and a semialgebraic subset of $W$.

\subsubsection{Invariant normalization in the Nash setting: semialgebraic structure}\label{innsss0}
For each $k\geq1$ we identify $\C^k$ with its underlying real structure $\R^{2k}$ when referring to its semialgebraic subsets and denote $\sigma_k:\C^k\to\C^k$ the complex conjugation of $\C^k$. Let $X\subset\R^n$ be a Nash set. Then there exist $m\geq0$, open semialgebraic subsets $\Omega\subset\C^n$, $\Theta\subset\C^{n+m}$, complex analytic sets $\widetilde{X}\subset\Omega$ and $Y\subset\Theta$ such that:
\begin{itemize}
\item[(i)] $\widetilde{X}$ is an invariant complexification of $X$ and it is a semialgebraic subset of $\C^n$.
\item[(ii)] $Y$ is a semialgebraic subset of $\C^{n+m}$.
\item[(iii)] $\rho:Y\to\widetilde{X},\ (x,w)\mapsto x$ is an invariant normalization of $\widetilde{X}$ (with respect to the usual conjugation of $\C^{n+m}$), where $\rho$ is the restriction to $Y$ of the projection $P:\C^{n+m}=\C^n\times\C^m\to\C^n$ onto the first factor. Observe that $\rho\circ\sigma_{n+m}=\sigma_n\circ\rho$.
\end{itemize} 

\subsubsection{Regular and singular points in the Nash setting}\label{rspns}
We say that $x$ is a \em regular point of a Nash set $X\subset\R^n$ \em if $\Nn(X)_{\gtn_x}$ is a regular local ring (where $\gtn_x$ is the maximal ideal of $\Nn(\R^n)$ associated to $x$). Let us check that regular points in this case coincide with the regular points of $X$ endowed with its analytic structure. Pick a point $x$ of the Nash set $X\subset\R^n$ and assume for simplicity $x=0$. Thus, $\Nn_{\R^n,x}$ coincides with the ring $\R[[\x]]_{\rm alg}$ of real algebraic power series in $n$ variables \cite[Cor.8.1.6]{bcr}. Its completion is the ring $\R[[\x]]$ of real formal power series in $n$ variables, which is also the completion of the local ring $\Nn(\R^n)_{\gtn_x}$. The completions of the local rings $\Nn_{\R^n,x}/\Ii^\bullet_{X,x}=\Nn_{\R^n,x}/(\Ii^\bullet(X)\Nn_{\R^n,x})$ and $\Nn(X)_{\gtn_x}=(\Nn(\R^n)/\Ii^\bullet(X))_{\gtn_x}$ are by \cite[(23.I) Thm.54, (23.L) Thm.55, Cor.3]{mt} both isomorphic to $\R[[\x]]/(\Ii^\bullet(X)\R[[\x]])$. 

If we consider $X$ as a $C$-analytic set, the completions of the local excellent rings 
$$
\begin{cases}
\an(X)_{\gtm_x}=(\an(\R^n)/\Ii(X))_{\gtm_x}=(\an(\R^n)/\Ii^\bullet(X)\an(\R^n))_{\gtm_x},\\
\an_{\R^n,x}/\Ii_{X,x}=\an_{\R^n,x}/(\Ii^\bullet_{X,x}\an_{\R^n,x})=\an_{\R^n,x}/(\Ii^\bullet(X)\an_{\R^n,x})
\end{cases}
$$ 
are by \cite[(23.I) Thm.54, (23.L) Thm.55, Cor.3]{mt} both isomorphic to $\R[[\x]]/(\Ii^\bullet(X)\R[[\x]])$. This means that the completions of the rings $\Nn(X)_{\gtn_x},\Nn_{\R^n,x}/\Ii^\bullet_{X,x},\an_{X,x}$ and $\an(X)_{\gtm_x}$ coincide.

All the local rings above are by \cite[Ch.VII.Ex.2.3]{abr} excellent, so the inclusions of such local rings in their respective completions are faithfully flat, regular morphisms \cite[(33.A), (34.A)]{mt}. This means that if one of the excellent local rings $\Nn(X)_{\gtn_x}$, $\Nn_{\R^n,x}/\Ii^\bullet_{X,x}$, $\an_{X,x}$, $\an(X)_{\gtm_x}$ is regular, then all the others are regular as well, see \cite[(33.B), Lem.2]{mt} and \cite[(33.E), Lem.4]{mt}. Consequently, the regular (resp. singular) points of $X$ are the same if we consider either the Nash structure or the analytic structure of the Nash set $X$. Using Jacobian criterion for Nash manifolds, we deduce that if $X\subset\R^n$ is a Nash set, then $\Sing(X)$ is a Nash subset of $X$ of strictly smaller dimension and $\Reg(X)=X\setminus\Sing(X)$ is a semialgebraic subset of $X$.

\subsection{Set of `tails' of a \texorpdfstring{$C$}{C}-analytic set}
Let us characterize, using the normalization of the complexification of a $C$-analytic set $X\subset\R^n$, the set $T(X):=\{x\in X:\ \Jj_{X,x}\neq\Ii_{X,x}\}$ of `tails' of $X$, that is, the set of points of $X$ where the sheaves $\Jj_X$ and $\Ii_X$ differ. We call $T(X)$ in this way, because it generalizes the visible `tails' of celebrated $C$-analytic sets like Cartan's or Whitney's umbrellas (even if in the general case the `tails' of $X$ could be hidden). 
\begin{remark}
Let $X\subset\R^n$ be a Nash set. As the inclusion homomorphism $\Nn_{\R^n,x}\hookrightarrow\an_{\R^n,x}$ is by Corollary \ref{ffna} faithfully flat, $\Jj_{X,x}=\Jj^\bullet_{X,x}\an_{\R^n,x}$ and $\Ii_{X,x}=\Ii^\bullet_{X,x}\an_{\R^n,x}$ for each $x\in X$, we deduce that $\Jj^\bullet_{X,x}\neq\Ii^\bullet_{X,x}$ if and only if $\Jj_{X,x}\neq\Ii_{X}$. This means that $T(X)=\{x\in X:\ \Jj^\bullet_{X,x}\neq\Ii^\bullet_{X,x}\}$.\hfill$\sqbullet$
\end{remark}

We provide next a geometric property of the `tails' of a $C$-analytic set.

\begin{lem}\label{tx0}
Let $X\subset\R^n$ be a $C$-analytic set and let $\widetilde{X}$ be an invariant complexification of $X$. Then for each $x\in T(X)$ there exists an irreducible component $T_x$ of $\widetilde{X}$ (depending on $x$) such that $\dim_\R(T_x\cap\R^n)<\dim_\C(T_x)$.
\end{lem}
\begin{proof}
Pick a point $x\in X$. As $\Ii_{X,x}\subsetneq\Jj_{X,x}$, we also have $\Jj^\C_{\widetilde{X},x}=\Ii_{X,x}\otimes_\R\C\subsetneq\Jj_{X,x}\otimes_\R\C$ (because the homomorphism $\R\hookrightarrow\C$ is faithfully flat). The germ $Z_x:=\ZZ(\Jj_{X,x}\otimes_\R\C)$ is the complexification of the germ $X_x=\ZZ(\Jj_{X,x})$ and it holds $Z_x\neq\widetilde{X}_x$. Let $\widetilde{X}_{1,x},\ldots,\widetilde{X}_{s,x}$ be the irreducible components of $\widetilde{X}_x$. The irreducible components of $X_x$ are however by \cite[Ch.V.Prop.2]{n1} the intersections with $\R^n_x$ of the irreducible components of $Z_x$. We have $Z_x\subsetneq\widetilde{X}_x$ and $Z_x\cap\R^n_x=X_x=\widetilde{X}_x\cap\R^n_x$. Suppose that for each $j=1,\ldots,s$ there exists an irreducible component $S_j$ of $\widetilde{X}_{j,x}\cap\R^n_x$ such that $\dim_\R(S_j)=\dim_\C(\widetilde{X}_{j,x})$. Thus, as $\widetilde{X}_{j,x}$ is irreducible, $\widetilde{X}_{j,x}$ is the complexification of $S_j\subset X$, so $\widetilde{X}_x$ is the complexification of $\bigcup_{j=1}^sS_j\subset X$. Consequently, $\widetilde{X}_x\subset Z_x$, which is a contradiction. Consequently, the statement holds, as required.
\end{proof}
\begin{remark}
We prove in Corollary \ref{txca} that this property characterizes the points of $T(X)$.\hfill$\sqbullet$
\end{remark}

Recall that a {\em $C$-semianalytic subset} $S$ of $\R^n$ is a locally finite union in $\R^n$ of {\em basic $C$-semianalytic sets}, that is, subsets of $\R^n$ of the type $\{f=0,g_1>0,\ldots,g_r>0\}$ where $r\geq1$ and $f,g_i\in\an(\R^n)$ (see \cite[\S3]{abf2}). We will use the following remark freely along this article.

\begin{remarks}\label{dense}
(i) If $S\subset\R^n$ is a $C$-semianalytic set, there exists a countable subset $D\subset S$ that is dense in $S$.

As $S$ is a locally finite union of set of the type $T:=\{f=0,g_1>0,\ldots,g_r>0\}$, we deduce using an exhaustion of $R^n$ by compact sets that the previous locally finite union is countable. Thus, it is enough to consider the case of basic $C$-semianalytic sets $T:=\{f=0,g_1>0,\ldots,g_r>0\}$. Denote $A:=\{g_1>0,\ldots,g_r>0\}$ and $X:=\{f=0\}$, so $T=X\cap A$. As $T$ is an open subset of $X$, if $D$ is a dense subset of $X$, then $D\cap A$ is a dense subset of $T$, because $\cl(D\cap A)\cap A=\cl(D)\cap A=T\cap A$. Consider the stratification of $X$ (as a finite union of analytic manifolds) given by $X=\bigcap_{k\geq0}\Sing_k(X)\setminus\Sing_{k+1}(X)$ where $\Sing_k(X):=\Sing_{k-1}(\Sing(X))$ and $\Sing_0(X):=X$. Thus, it is enough to prove that analytic manifolds admit dense countable subsets, which follows from their local structure diffeomorphic to $\R^n$ and their paracompactness.

(ii) In the Nash setting $C$-semianalytic sets correspond to semialgebraic sets. As semialgebraic sets admit finite stratification with strata Nash manifolds \cite[\S9]{bcr} and Nash manifolds can be covered by finitely many open subsets that are Nash diffeomorphic to $\R^n$, each semialgebraic set admit a dense countable subset.
\hfill$\sqbullet$
\end{remarks}

\begin{lem}[Description of tails]\label{dot}
Let $X\subset\R^n$ be a $C$-analytic set (resp. a Nash set), let $\widetilde{X}$ be an invariant Stein complexification of $X$ and let $(Y,\pi)$ be an invariant Stein normalization of $\widetilde{X}$ endowed with an involution $\widehat{\sigma}$ compatible with the complex conjugation $\sigma$ of $\C^n$. Then $T(X)=\pi(\pi^{-1}(X)\setminus\cl(Y^{\widehat{\sigma}}\setminus\Sing(Y^{\widehat{\sigma}})))$ and it is a $C$-semianalytic set (resp. a semialgebraic set) of dimension $\dim(T(X))<\dim(X)$ contained in $\Sing(X)$.
\end{lem}
\begin{proof}
Pick a point $x\in T(X)$. Write $\pi^{-1}(x)=\{y_1,\ldots,y_s\}$ and recall that $\pi(Y_{y_1}),\ldots,\pi(Y_{y_s})$ are the irreducible components of $\widetilde{X}_x$. By Lemma \ref{tx0} there exists an irreducible component $\pi(Y_{y_j})$ of $\widetilde{X}_x$ such that $\dim_\R(\pi(Y_{y_j})\cap\R^n_x)<\dim_\C(\pi(Y_{y_j}))$. As $\pi$ is a proper analytic map with finite fibers, we have $\dim_\R(\pi^{-1}(X)_{y_j})<\dim_\C(Y_{y_j})$. As $\pi\circ\widehat{\sigma}=\sigma\circ\pi$, we deduce $Y^{\widehat{\sigma}}_{y_j}\subset\pi^{-1}(X)_{y_j}$, so $\dim_\R(Y^{\widehat{\sigma}}_{y_j})<\dim_\C(Y_{y_j})$. Thus, $y_j\in\pi^{-1}(X)\setminus\cl(Y^{\widehat{\sigma}}\setminus\Sing(Y^{\widehat{\sigma}}))$, because $\cl(Y^{\widehat{\sigma}}\setminus\Sing(Y^{\widehat{\sigma}}))=\{y\in Y^{\widehat{\sigma}}:\ \dim_\R(Y^{\widehat{\sigma}}_y)=\dim_\C(Y_{y})\}$.

We prove next the converse. Let $x\in X$ be such that $\pi(y)=x$ for some $y\in\cl(Y^{\widehat{\sigma}}\setminus\Sing(Y^{\widehat{\sigma}}))$. 

\noindent{\sc Case 1}. Suppose first $x\in\pi^{-1}(X)\setminus Y^{\widehat{\sigma}}$. We have $\pi(\widehat{\sigma}(y))=\sigma(\pi(y))=\sigma(x)=x$ and $y\neq\widehat{\sigma}(y)$. Thus, $\pi(Y_y)$ and $\pi(Y_{\widehat{\sigma}(y)})$ are two (different) irreducible components of $\widetilde{X}_x$ that are mutually conjugated, because $\sigma(\pi(Y_y))=\pi(\widehat{\sigma}(Y_y))=\pi(Y_{\widehat{\sigma}(y)})$. The germ $Z_x:=\ZZ(\Jj_{X,x}\otimes_\R\C)$ is the complexification of the germ $X_x=\ZZ(\Jj_{X,x})$. If no irreducible component of $Z_x$ is contained in $\pi(Y_y)$, then $\Jj_{X,x}\otimes_\R\C\neq\Ii_{X,x}\otimes_\R\C$, so $\Jj_{X,x}\neq\Ii_{X,x}$. Thus, suppose $Z_{1,x}$ is an irreducible component of $Z_x$ contained in $\pi(Y_y)$ and denote $X_{1,x}:=Z_{1,x}\cap\R^n_x$, which is by \cite[Ch.V.Prop.2]{n1} an irreducible component of $X_x$. As $X_{1,x}\subset\pi(Y_y)\cap\R^n_x$, we have $X_{1,x}\subset\pi(Y_y)\cap\pi(Y_{\widehat{\sigma}(y)})\subset\Sing(\widetilde{X})_x$, so the complexification $Z_{1,x}$ of $X_{1,x}$ is contained in $\pi(Y_y)\cap\pi(Y_{\widehat{\sigma}(y)})\subset\Sing(\widetilde{X})_x$. Thus, $Z_x\neq\widetilde{X}_x$ and $\Jj_{X,x}\neq\Ii_{X,x}$, so $x\in T(X)$. 

\noindent{\sc Case 2}. Suppose $y\in\Sing(Y^{\widehat{\sigma}})\setminus\cl(Y^{\widehat{\sigma}}\setminus\Sing(Y^{\widehat{\sigma}}))$ with $\pi(x)=y$. Then $Y^{\widehat{\sigma}}_y=\Sing(Y^{\widehat{\sigma}})_y$ and $\pi(Y_y)$ is an irreducible component of $\widetilde{X}_x$. As $\pi|_{Y\setminus\pi^{-1}(\Sing(\widetilde{X}))}:Y\setminus\pi^{-1}(\Sing(\widetilde{X}))\to\widetilde{X}\setminus\Sing(\widetilde{X})$ is an analytic diffeomorphism and $\sigma\circ\pi=\pi\circ\widehat{\sigma}$, we deduce
$$
\pi(Y_y)\cap\R^n_x=\pi(Y_y)\cap X_x\subset\Sing(\widetilde{X})_x\cup\pi(Y^{\widehat{\sigma}}_y)=
\Sing(\widetilde{X})_x\cup\pi(\Sing(Y^{\widehat{\sigma}})_y)\subset\Sing(\widetilde{X})_x.
$$
Thus, the complexification of $\pi(Y_y)\cap\R^n_x$ is contained in $\Sing(\widetilde{X})_x$, so $\pi(Y_y)$ is not the complexification of $\pi(Y_y)\cap\R^n_x$. Consequently, $\Ii_{X,x}\neq\Jj_{X,x}$ and $x\in T(X)$.

Observe that in both cases $x\in\Sing(\widetilde{X})$. Thus, $T(X)\subset\Sing(\widetilde{X})\cap X=\Sing(X)$ and $\dim(T(X))\leq\dim(\Sing(X))<\dim(X)$.

The $C$-semianalyticity of $T(X)$ follows from \cite[Thm.1.1, Prop.3.5]{abf2}. We follow the description of $T(X)$, when $X$ is a Nash set, and we observe that the involved objects are of semialgebraic nature: complexification (semialgebraic, see \S\ref{scns0}), normalization of the complexificacion (semialgebraic, see \S\ref{innsss0}), sets of regular and singular points (semialgebraic, see \S\ref{rspns}). The remaining operations that appear in the construction of $T(X)$ are: boolean operations, images under semialgebraic maps, inverse images under semialgebraic maps, topological closures, which are of semialgebraic nature. Thus, $T(X)$ is a semialgebraic set, as required. 
\end{proof}

\subsection{Description of the set of points of non-coherence.}\label{dspnc}
Let $X\subset\R^n$ be a $C$-analytic set and define $N(X)$ as the set of points $x\in X$ such that $\Jj_{X}$ is not of finite type at $x$ (or equivalently, the analytic germ $X_x$ is not coherent), that is, for each open neighborhood $U\subset X$ of $x$ the restriction sheaf $\Jj_{X}|_U$ is not of finite type. We describe next the set $N(X)$ of points of non-coherence of a $C$-analytic set $X\subset\R^n$ (following \cite[\S5]{abf2}) and then we show that it $X$ is in addition a semialgebraic set, then $N(X)$ is semialgebraic as well. It can be enlightening for the reader to have the following examples in mind when reading \S\ref{asnx}. The description of $N(X)$ recalled bellow is crucial to prove Theorem \ref{xy}.
\begin{examples}
(i) The normalization of the complexification $\widetilde{W}\subset\C^3$ of Whitney's umbrella $W:=\{y^2-zx^2=0\}\subset\R^2$ is $\pi:\C^2\to\widetilde{W},\ (u,v)\mapsto(u,uv,v^2)$ and the inverse image $\pi^{-1}(W)=\R^2\cup\{(0,\pm{\tt i}w):\ w\in\R\}$.

(ii) The set of singular points of the complexification $\widetilde{X}\subset\C^4$ of the $3$-dimensional $C$-irreducible $C$-analytic hypersuface $X:=\{w^2-z(x^2+y^2)=0\}\subset\R^4$ is 
$$
\Sing(\widetilde{X})=\{x=0,y=0,w=0\}\cup\{x^2+y^2=0,z=0,w=0\}\subset\C^4, 
$$
which has codimension $2$ in $\widetilde{X}$. The irreducible complex analytic hypersurface $\widetilde{X}$ is by \cite{ok} normal (so it is locally irreducible), whereas $X$ is non-coherent, because it has a visible `tail' (the germs $X_{p}$ at the points $p:=(0,0,z,0)\in X$ with $z<0$ have dimension $1$). We refer the reader to \cite[Esempio, p. 211]{t1} and \cite[Ex.2.4]{abf2} for further details.

(iii) Let $X:=X_1\cup X_2\subset\R^4$ where $X_1=\{(x^2-(z+1)y^2)^2z-u^2=0\}$ is a modified Whitney's umbrella and $X_2:=\{u=0\}$. Then $N(X_1)=\{x^2-y^2=0,z=0,u=0\}\cup\{(0,0,-1,0)\}$, whereas $N(X)=\{x^2-y^2=0,z=0,u=0\}$ (see the concrete details in \cite[Ex.5.7(ii)]{abf2}).
\end{examples}

\subsubsection{Analytic setting}\label{asnx}
Let $X\subset\R^n$ be a $C$-analytic set. For each $\ell\geq0$ define $\Sing_\ell(X)$ recursively as $\Sing_\ell(X):=\Sing(\Sing_{\ell-1}(X))$, where $\Sing_0(X):=X$. For each $k\in\{0,\ldots,d\}$ let ${\mathfrak F}_k$ be the family of all the $C$-analytic sets $Z\subset\R^n$ of dimension $k$ that are a $C$-irreducible component of $\Sing_\ell(X)$ for some $\ell\geq0$. Define $Z_k:=\bigcup_{Z\in{\mathfrak F}_k}Z$ and 
$$
\textstyle
R_k:=\bigcup_{j=k+1}^dZ_{j,(j)}\quad\text{where}\quad Z_{j,(j)}:=\{z\in Z_j:\dim(Z_{j,z})=j\}=\cl(Z_j\setminus\Sing(Z_j)). 
$$

Let $\widetilde{Z}_k$ be a complexification of $Z_k$ and let $(Y_k,\pi_k)$ be the normalization of $\widetilde{Z}_k$ endowed with the anti-involution $\sigma_k:Y_k\to Y_k$ that is induced by the usual conjugation on $\widetilde{Z}_k$. Let 
$$
Y_k^{\sigma_k}:=\{y\in Y_k:\ \sigma_k(y)=y\} 
$$
be the fixed part space of $Y_k$, which is a $C$-analytic space. Define 
\begin{align*} 
&Y^{\sigma_k}_{k,(k)}:=\{y\in Y_k^{\sigma_k}:\dim_{\R}(Y_{k,y}^{\sigma_k})=k\}=\cl(Y^{\sigma_k}_k\setminus\Sing(Y^{\sigma_k}_k)),\\
&C_{k,1}:=\pi_k^{-1}(Z_k)\setminus Y_k^{\sigma_k},\quad C_{k,2}:=Y_k^{\sigma_k}\setminus Y^{\sigma_k}_{k,(k)},\\
&A_{k,i}:=\cl(C_{k,i})\cap\cl(Y^{\sigma_k}_{k,(k)}\setminus\pi_k^{-1}(R_k))\quad\text{for $i=1,2$}.
\end{align*}
\begin{thm}[{\cite[\S5.1]{abf2}}]\label{ncp0}
Suppose the $C$-analytic set $X\subset\R^n$ has dimension $d\geq2$ and let $N_k(Z_k,R_k):=\pi_k(A_{k,1})\cup\pi_k(A_{k,2})$ for $k\in\{2,\ldots,d\}$. Then, 
\begin{itemize}
\item[(i)] $N_k(Z_k,R_k)$ is a $C$-semianalytic set of dimension $\leq k-2$.
\item[(ii)] A point $x\in X$ belongs to $N_k(Z_k,R_k)$ if and only if the germ $Z_{k,x}$ has a non-coherent irreducible component $T_x$ of dimension $k$ such that $\dim(T_x\setminus R_{k,x})=k$.
\item[(iii)] $\bigcup_{k=j}^dN_k(Z_k,R_k)$ is the set of points of $X$ such that the germ $X_x$ has a non-coherent irreducible component of dimension $\geq j$.
\item[(iv)] $N(X)=\bigcup_{k=2}^dN_k(Z_k,R_k)$ is a $C$-semianalytic subset of $\R^n$.
\end{itemize}
\end{thm}

\begin{remarks}\label{remnx}
(i) Each set $N_k(Z_k,R_k):=\pi_k(A_{k,1})\cup\pi_k(A_{k,2})$ is a closed subset of $X$ (for $k=2,\ldots,d$), because each $A_{k,i}$ is a closed subset of $Y_k$, $\pi_k:Y_k\to\widetilde{Z}_k$ is a proper map and $Z_k$ is a closed subset of both $X$ and $\widetilde{Z}_k$. 

(ii) For each open set $U\subset\R^n$ we have $N(X)\cap U=N(X\cap U)$, because coherence and non-coherence are local concepts. By Theorem \ref{ncp0}(i),(iv) we deduce $X\setminus N(X)$ is dense in $X$.

(iii) If $X\subset\R^n$ be a $C$-analytic set, then $N(X)$ is a closed subset of $X$ and $X\setminus N(X)$ is a coherent $C$-analytic subset of the open set $\R^n\setminus N(X)$.$\hfill\sqbullet$
\end{remarks}

\subsubsection{Nash setting}
Let use show: {\em The set $N(X)$ of points of non-coherence of a Nash set $X\subset\R^n$ is a semialgebraic subset of $\R^n$}.
\begin{proof}
We follow the description of $N(X)$, when $X$ is a Nash set, and we observe that the involved objects are of semialgebraic nature: complexification (semialgebraic, see \S\ref{scns0}), normalization of the complexificacion (semialgebraic, see \S\ref{innsss0}), sets of regular and singular points (semialgebraic, see \S\ref{rspns}), $C$-analytic irreducible components of $X$ are finitely many and coincide with the Nash irreducible components of $X$ (see \cite[Prop.2.8(ii)]{fgr}).

The remaining operations that appear in the construction of $N(X)$ are: boolean operations, images under semialgebraic maps, topological closures, considering sets of points of a certain dimension, and all of them are of semialgebraic nature. Thus, $N(X)$ is a semialgebraic set. 
\end{proof}

\begin{remark}\label{xnx}
By Lemma \ref{ijn} a Nash subset $X$ of $\R^n$ is coherent as an analytic set if and only if it is coherent as a Nash set. This means that the semialgebraic set $X\setminus N(X)$ is a coherent Nash subset of the open semialgebraic set $\R^n\setminus N(X)$.$\hfill\sqbullet$
\end{remark}

\subsection{Pure dimensional non-coherent examples}\label{examples}
A general idea in Real Geometry is that non-coherence arises when the irreducible components of the objects are not pure dimensional. However, non-coherence appears also due to the presence of hidden `tails' inside the irreducible components of the object. We present next several pure dimensional examples that reproduce several situations concerning $\Sing(X)$, $T(X)$ and $N(X)$. A first example is $X:=\{z^3-x^2y^3=0\}$, which appears in \cite{h} and \cite[Ex.2.1]{bm5}. It holds $\Sing(X)=\{xy=0,z=0\}$, $N(X)=\{(0,0,0)\}$ and $T(X)=\Sing(X)\setminus N(X)$. Observe that $z^3-x^2y^3=(z-x^{2/3}y)(z-wx^{2/3}y)(z-w^2x^{2/3}y)$, where $w=\frac{1}{2}+{\tt i}\frac{\sqrt{3}}{2}$, so $X$ is homeomorphic to $\R^2$. Other relevant example from \cite[Ex.1.9]{gal} and \cite[\S5.3.1, p.197]{abf5} is $X:=\{z(x+y)(x^2+y^2)-x^4=0\}\subset\R^3$, which has $\Sing(X)=\{x=0,y=0\}$, $N(X)=\{(0,0,0)\}$ and $T(X)=\Sing(X)\setminus N(X)$. We leave in both cases the concrete computations to the reader.  

The previous situation suggests to distinguish in the real analytic case between regular points and smooth points, something that in the complex analytic case never happens, because in such setting both concepts are equivalent. Let $X\subset\R^n$ be a $C$-analytic set (resp. Nash set) and let $x\in X$. We say that $x$ is a \em smooth point of $X$ \em if the analytic germ $X_x$ coincides with the germ at $x$ of a real analytic manifold. Denote ${\tt Smooth}(X)$ the set of smooth points of $X$ and ${\tt Non}\text{-}{\tt Smooth}(X):=X\setminus{\tt Smooth}(X)$. The set {\em ${\tt Smooth}(X)$ is always contained in $X\setminus N(X)$}, because the germs at a point of a real analytic manifold are coherent. By the Jacobian criterion $\Reg(X)\subset{\tt Smooth}(X)$, but the inclusion can be strict (see $X:=\{x^3-z^2y^3=0\}$, which has $\Reg(X)=X\setminus\{xy=0,z=0\}$ and ${\tt Smooth}(X)=X\setminus\{x=0,z=0\}$, and Examples \ref{snr} and \ref{snr2} below). However, {\em if $X$ is coherent, the equality $\Reg(X)={\tt Smooth}(X)$ holds}. 
\begin{proof}
Suppose $X\subset\R^n$ is a coherent $C$-analytic set and pick a point $x\in{\tt Smooth}(X)$. Then, the analytic germ $X_x$ is the germ at $x$ of a real analytic manifold, so the quotient ring $\an_{\R^n,x}/\Jj_{X,x}$ is a regular local ring. As $X$ is coherent, $\Jj_{X,x}=\Ii_{X,x}$, so $\an_{\R^n,x}/\Ii_{X,x}$ is a regular local ring. Consequently, $x\in\Reg(X)$, as required.
\end{proof}
\begin{example}
If $X$ is a $C$-analytic set (resp. a Nash set), the equality $\Reg(X)={\tt Smooth}(X)$ does not guarantee that $X$ is coherent. Let us modify the example that appeared in \cite[Ex.1.9]{gal} (see also \cite[\S5.3.1, p.197]{abf5}). Consider $X:=\{z^2(x+y)^2(x^2+y^2)-x^6=0\}\subset\R^3$ (see Figure \ref{fig2}). It holds $N(X)=\{(0,0,0)\}$, $\Sing(X)=\{x=0,yz=0\}$, $\Reg(X)={\tt Smooth}(X)$ and $T(X)=\{x=0,y=0\}\setminus N(X)$. We leave the concrete details to the reader.\hfill$\sqbullet$
\end{example}

\begin{center}
\begin{figure}[ht]
\begin{minipage}{0.49 \textwidth} 
\begin{center}
\includegraphics[width=0.63\textwidth]{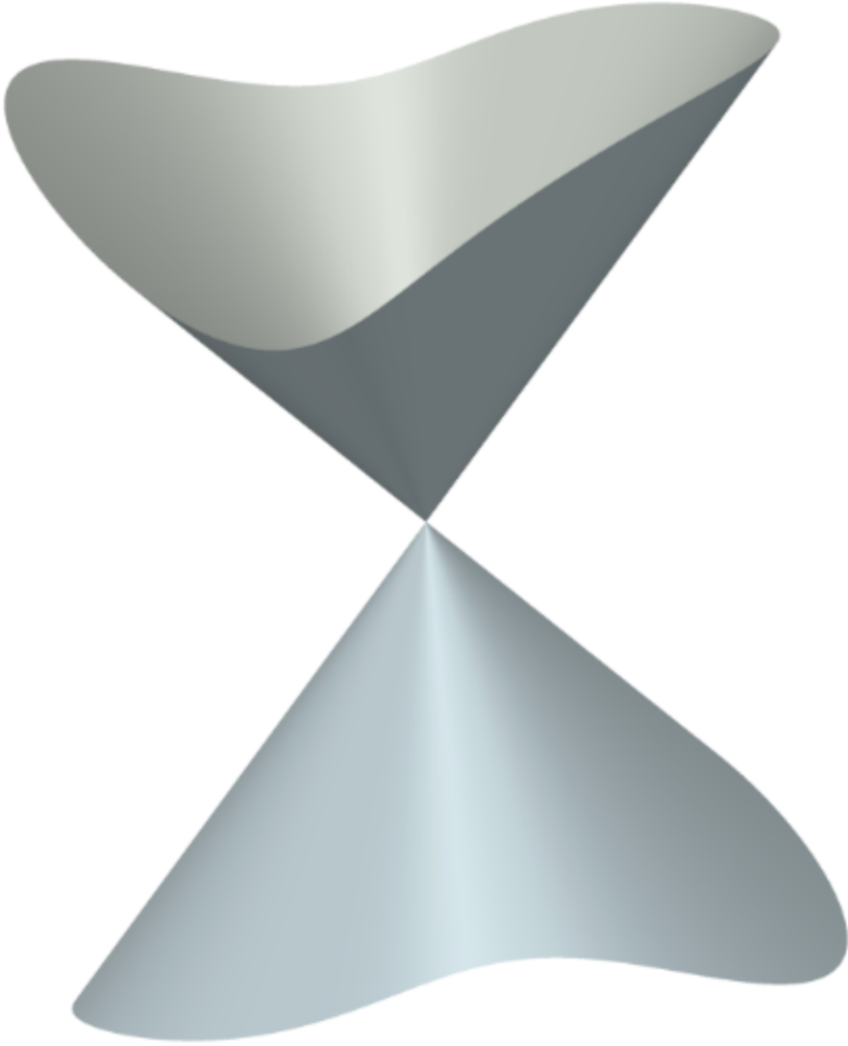}
\caption{\Small$X:z(x+y)(x^2+y^2)-x^4=0$\label{fig1}}
\end{center}
\end{minipage}
\hfil
\begin{minipage}{0.49\textwidth}
\begin{center}
\vspace*{2mm}
\includegraphics[width=0.8\textwidth]{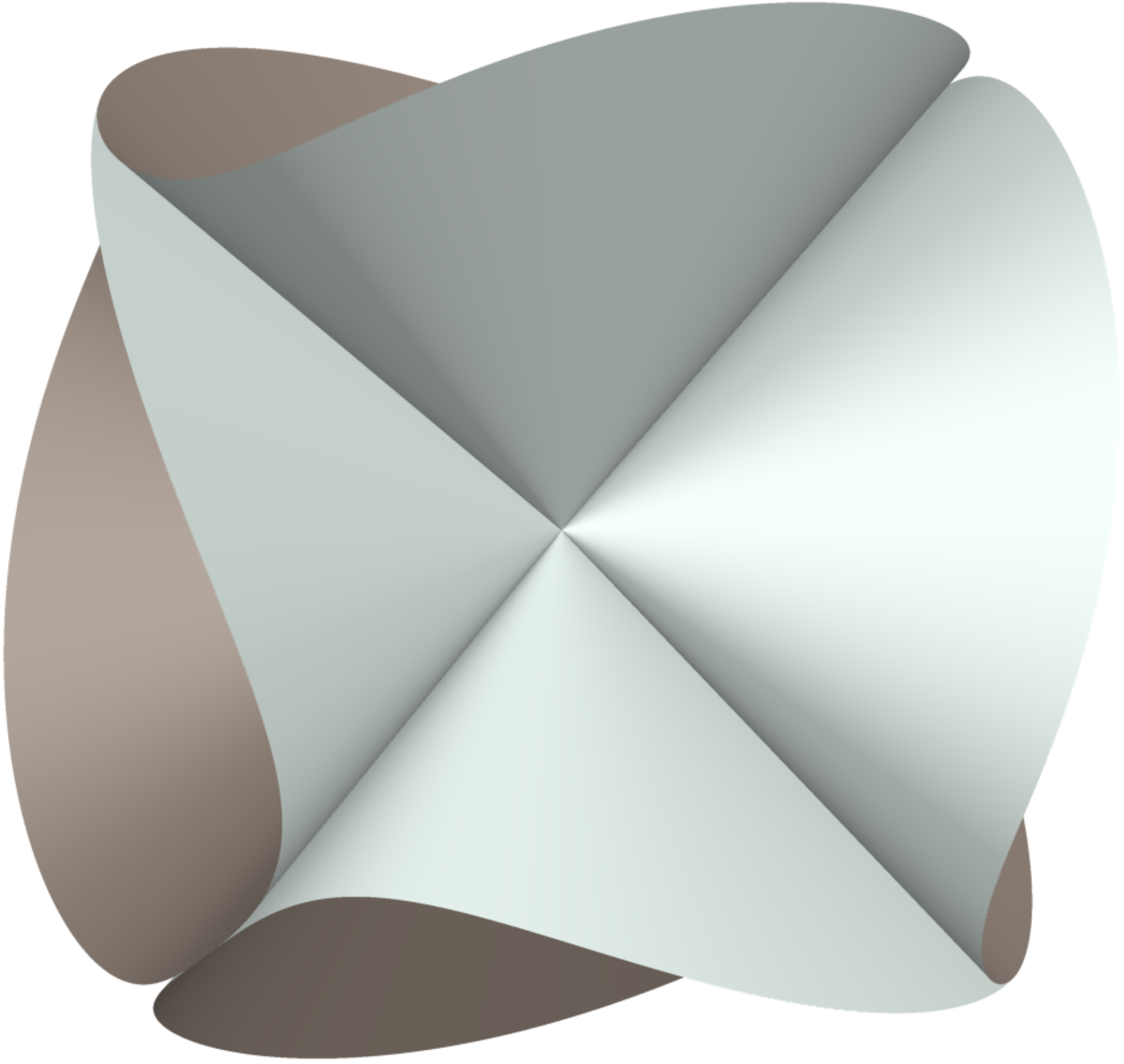}
\caption{\Small$X:z^2(x+y)^2(x^2+y^2)-x^6=0$\label{fig2}}
\end{center}
\end{minipage}
\end{figure}
\end{center}

We borrow the following pure dimensional example from \cite[Rem.7.3]{bm2}, \cite[Ex.2.2]{bm5} and \cite[Ex.2.1]{fe2}.

\begin{example}[Birdie non-coherent singularity]\label{snr}
Consider the algebraic set $X:=\{(x^2+zy^2)x-y^4=0\}\subset\R^3$. The set of regular points of $X$ is $X\setminus\{x=0,y=0\}$, whereas: {\em The set of smooth points of $X$ is $X\setminus\{x=0,y=0,z\leq0\}$} (see Figure \ref{fig3}).

To prove that the points of the open half-line $\{x=0,y=0,z<0\}$ are non-smooth we proceed by contradiction. Pick a point $p:=(0,0,-a^2)\in\{x=0,y=0,z<0\}$ and assume that it is smooth. As the line $\{x=0,y=0\}\subset X$, the vector $(0,0,1)$ would be tangent to $X$ at $p$, so the plane $z=-a^2$ would be transversal to $X$ at $p$. Thus, the intersection $X\cap\{z=-a^2\}$ should be a curve that is smooth at $p$, but this is a contradiction, because such curve $\{(x^2-(ay)^2)x-y^4=0,z=-a^2\}$ has three tangent lines at $p$, which are those lines of equations $\{x-ay=0\}$, $\{x+ay=0\}$ and $\{x=0\}$ inside the plane $\{z=-a^2\}$. The origin cannot be a smooth point of $X$, because the set of smooth points of $X$ is an open subset of $X$. Consequently, the set of non-smooth points of $X$ contains the closed half-line $\{x=0,y=0,z\leq0\}$.

Let us prove that the points of the open half-line $\{x=0,y=0,z>0\}$ are smooth. To that end, observe that the map 
$$
\varphi:\{(s,u)\in\R^2:\ u>0\}\to\R^3,\ (s,u)\mapsto((s^2+u)s^2,(s^2+u)s,u)
$$ 
is an analytic embedding whose image is $X\cap\{z>0\}$. Observe (computing the Jacobian matrix of $\varphi$) that $\varphi$ is a local analytic diffeomorphism onto its image. The analytic map $\varphi$ extends to a holomorphic map
$$
\Phi:\C^2\to\C^3,\ (v,w)\to((v^2+w)v^2,(v^2+w)v,w),
$$
which provides the normalization of the complexification $\widetilde{X}:=\{(x^2+zy^2)x-y^4=0\}$ of $X$. To that end, observe $v=\frac{x}{y}$ satisfies the integral equation $\t^3+z\t-y=0$ (and in addition $z=w$, $y=v^3+zv$ and $x=vy$). 

We prove next that $N(X)=\{(0,0,0)\}$. As $\Reg(X)\subsetneq{\tt Smooth}(X)=X\setminus\{x=0,y=0,z\leq0\}\subset X\setminus N(X)$, we have to check that the points of the open half-line $\{x=0,y=0,z<0\}$ provide coherent analytic germs. Pick a point $(0,0,t)$ with $t<0$. Then $\Phi^{-1}(\{(0,0,t)\})=\{(0,t),(\sqrt{-t},t),(-\sqrt{-t},t)\}$ and 
\begin{align*}
X_{(0,0,t)}&=\varphi(\R^2_{(0,t)})\cup\varphi(\R^2_{(\sqrt{-t},t)})\cup\varphi(\R^2_{(-\sqrt{-t},t)}),\\
\widetilde{X}_{(0,0,t)}&=\Phi(\C^2_{(0,t)})\cup\Phi(\C^2_{(\sqrt{-t},t)})\cup\Phi(\C^2_{(-\sqrt{-t},t)}).
\end{align*}
Let $g\in\Jj_{X,(0,0,t)}$ and let $G\in\an_{\C^n,(0,0,t)}$ its natural holomorphic extension. As $g$ vanishes identically on $\varphi(\R^2_{(0,t)})$, $\varphi(\R^2_{(\sqrt{-t},t)})$ and $\varphi(\R^2_{(-\sqrt{-t},t)})$, we deduce that $G$ vanishes identically on $\Phi(\C^2_{(0,t)})$, $\Phi(\C^2_{(\sqrt{-t},t)})$ and $\Phi(\C^2_{(-\sqrt{-t},t)})$, because the latters are the complexifications of the analytic germs $\varphi(\R^2_{(0,t)})$, $\varphi(\R^2_{(\sqrt{-t},t)})$ and $\varphi(\R^2_{(-\sqrt{-t},t)})$. Thus, $G\in\Jj_{\widetilde{X},(0,0,t)}=((x^2+zy^2)x-y^4)\an_{\C^n,(0,0,t)}$, so $(x^2+zy^2)x-y^4$ divides $g$ and $\Jj_{X,(0,0,t)}=((x^2+zy^2)x-y^4)\an_{\C^n,(0,0,t)}$. Consequently, $X_{(0,0,t)}$ is a coherent analytic germ for each $t<0$.  

We deduce also $T(X)=\Sing(X)\cap\{z<0\}$. The real parameter $s=\frac{x}{y}$ is a solution of the integral equation $\t^3+z\t-y=0$. Using the solution general equation of degree $3$ we find that
$$
s=f(z,y):=\sqrt[3]{\frac{y}{2}+\sqrt{\frac{z^3}{27}+\frac{y^2}{4}}}+\sqrt[3]{\frac{y}{2}-\sqrt{\frac{z^3}{27}+\frac{y^2}{4}}}
$$
As $x=ys$, we deduce that $F(x,y,z):=x-yf(y,z)$ is an analytic equation of $X\setminus\{z\leq0\}$ in $\R^3\setminus\{z\leq0\}$, which is the complement of a $C$-semianalytic set with no empty interior. The previous equation $F$ generates the ideal $\Jj_{X,p}$ for each $p\in X\setminus\{z\leq0\}$. By Lemma \ref{eq} below there exists an analytic $h\in\an(\R^n\setminus N(X))$ such that $h_x$ generates the ideal $\Jj_{X,x}$ for each $x\in\R^n\setminus N(X)$. We have not found an explicit expression of such analytic equation defined on $\R^n\setminus N(X)$ in terms of elementary analytic functions.
\hfill$\sqbullet$
\end{example}

\begin{center}
\begin{figure}[ht]
\begin{minipage}{0.49 \textwidth} 
\begin{center}
\includegraphics[width=0.8\textwidth]{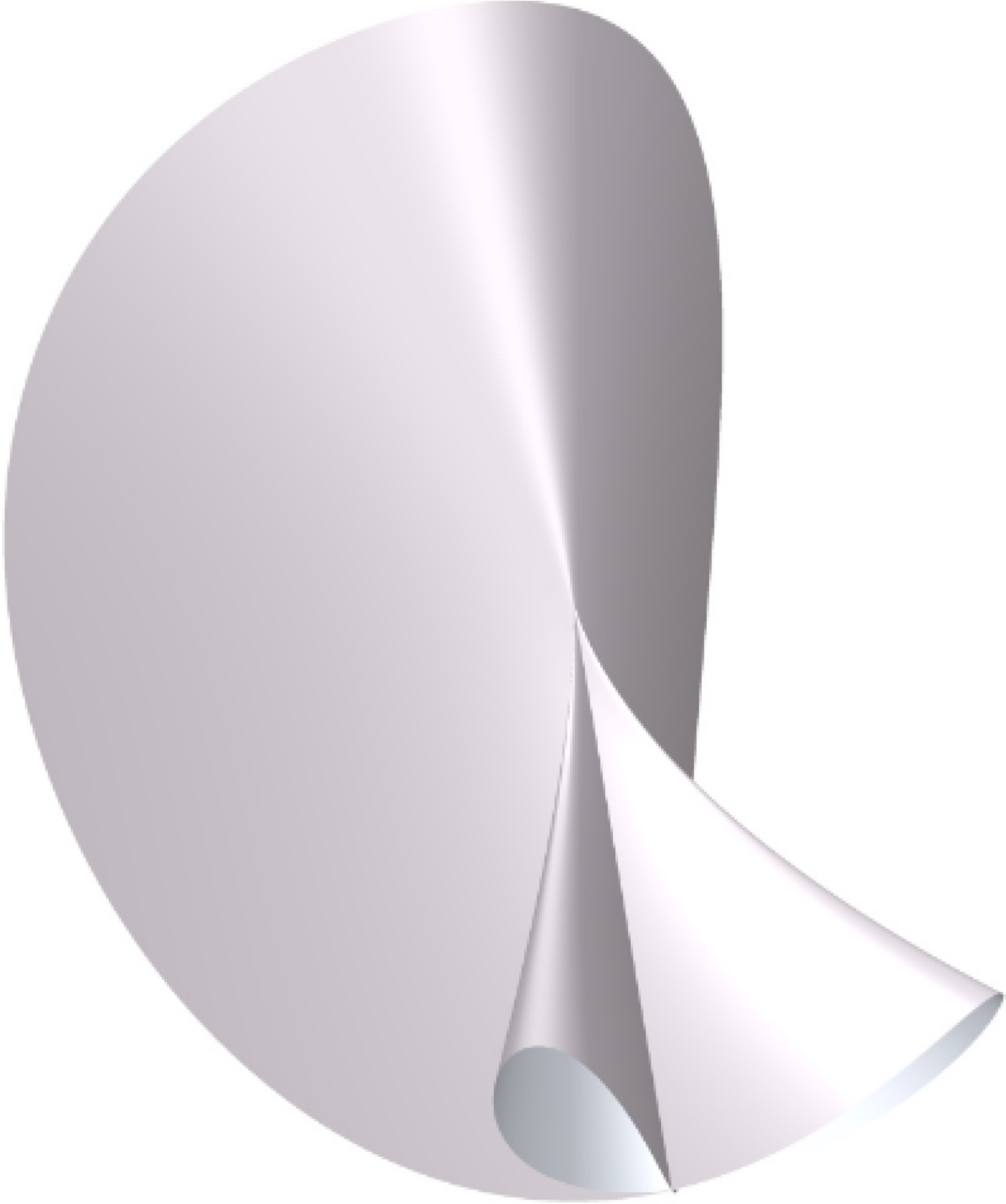}
\caption{\Small$X:(x^2+zy^2)x-y^4=0$\label{fig3}}
\end{center}
\end{minipage}
\hfil
\begin{minipage}{0.49\textwidth}
\begin{center}
\vspace*{2mm}
\includegraphics[width=0.93\textwidth]{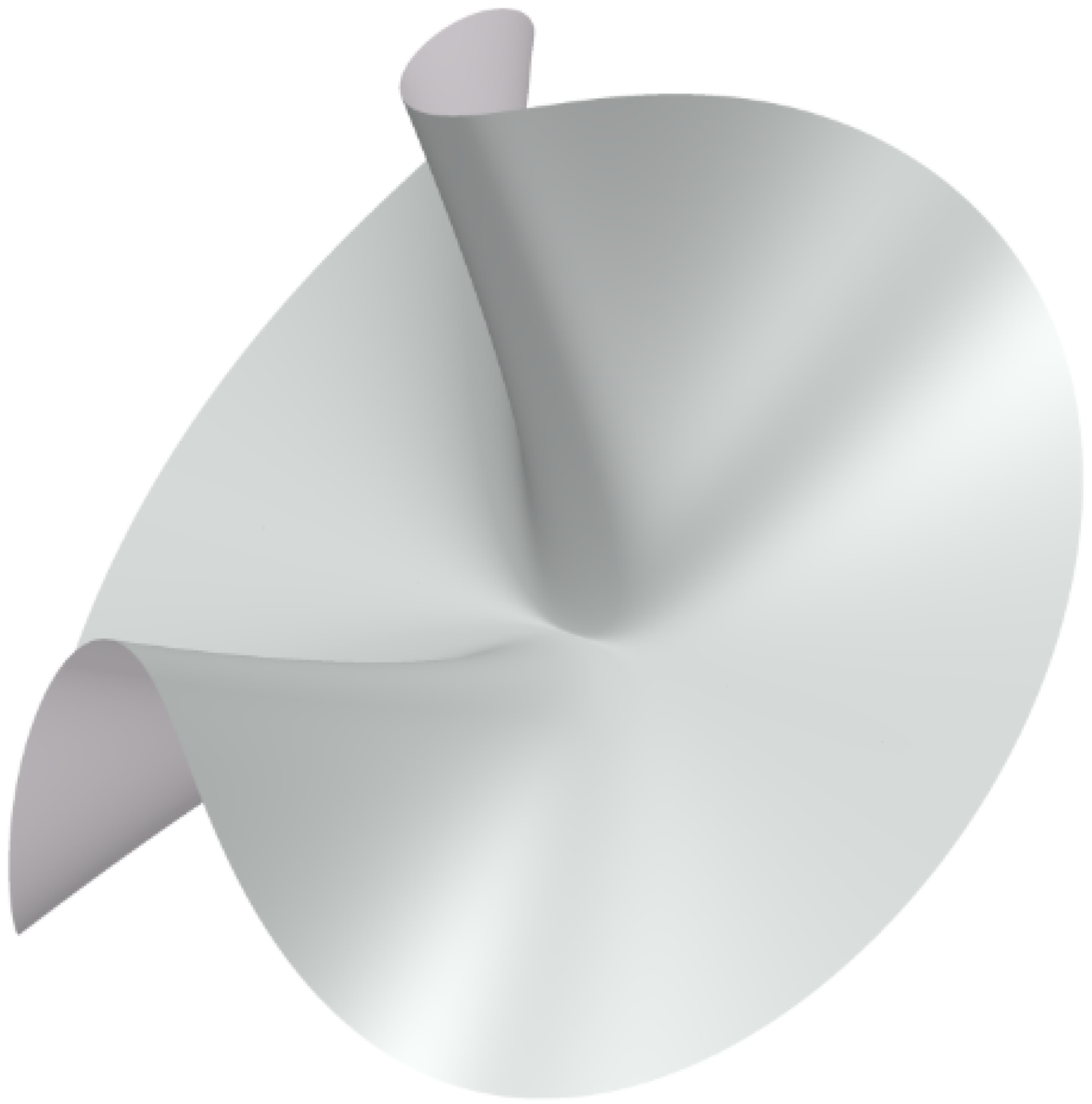}
\caption{\Small$X:(x^2+z^2y^2)x-y^4=0$\label{fig4}}
\end{center}
\end{minipage}
\end{figure}
\end{center}

\begin{example}[Fake blanket]\label{snr2}
Consider the algebraic set $X:=\{(x^2+z^2y^2)x-y^4=0\}\subset\R^3$. The set of regular points of $X$ is $X\setminus\{x=0,y=0\}$, whereas: {\em The set of smooth points of $X$ is $X\setminus\{(0,0,0)\}$} (see Figure \ref{fig4}). 

To prove that the origin is a non smooth-point we proceed by contradiction. Assume it is a smooth point. As the line $\{x=0,y=0\}\subset X$, the vector $(0,0,1)$ would be tangent to $X$ at $p$, so the plane $z=0$ would be transversal to $X$ at $p$. Thus, the intersection $X\cap\{z=0\}$ should be a curve that is smooth at $p$, but this is a contradiction, because such curve is $\{x^3-y^4=0,z=0\}$.

We prove next that the points of $X\setminus\{(0,0,0)\}$ are smooth. To that end, observe that the map 
$$
\varphi:\R^2\setminus\{(0,0)\}\to\R^3,\ (s,u)\mapsto((s^2+u^2)s^2,(s^2+u^2)s,u)
$$ 
is an analytic embedding whose image is $X\setminus\{(0,0,0)\}$. In fact, the previous map extends to a homeomorphism between $\R^2$ and $X$, whose inverse is the map 
$$
\psi:X\to\R^2,\ (x,y,z)\mapsto
\begin{cases}
(\frac{x}{y},z)&\text{if $y\neq0$},\\
(0,z)&\text{if $y=0$.}
\end{cases}
$$
As $X$ is non-coherent, because $\Reg(X)\subsetneq{\tt Smooth}(X)$, and $X\setminus\{(0,0,0)\}\subset{\tt Smooth}(X)\subset X\setminus N(X)\subsetneq X$, we deduce $N(X)=\{(0,0,0)\}$. In addition, $T(X)=\Sing(X)\setminus N(X)$.

The analytic map $\varphi$ extends to a holomorphic map
$$
\Phi:\C^2\to\C^3,\ (v,w)\to((v^2+w^2)v^2,(v^2+w^2)v,w),
$$
which provides the normalization of the complexification $\widetilde{X}:=\{(x^2+zy^2)x-y^4=0\}$ of $X$. To that end, observe $v=\frac{x}{y}$ satisfies the integral equation $\t^3+z^2\t-y=0$ (an in addition $z=w$, $y=v^3+z^2v$ and $x=vy$).

The real parameter $s=\frac{x}{y}$ is a solution of the integral equation $\t^3+z^2\t-y=0$. Using the solution general equation of degree $3$ we find that
$$
s=f(z,y):=\sqrt[3]{\frac{y}{2}+\sqrt{\frac{z^6}{27}+\frac{y^2}{4}}}+\sqrt[3]{\frac{y}{2}-\sqrt{\frac{z^6}{27}+\frac{y^2}{4}}}
$$
As $x=ys$, we deduce that $F(x,y,z):=x-yf(y,z)$ is an analytic equation of $X\setminus\{z=0\}$ in $\R^3\setminus\{z=0\}$, which is the complement in $\R^3$ of a $C$-analytic set of dimension $2$. The previous equation $F$ generates the ideal $\Jj_{X,p}$ for each $p\in X\setminus\{z=0\}$. By Lemma \ref{eq} below there exists an analytic $h\in\an(\R^n\setminus N(X))$ such that $h_x$ generates the ideal $\Jj_{X,x}$ for each $x\in\R^n\setminus N(X)$. Again, we have not found an explicit expression of such analytic equation defined on $\R^n\setminus N(X)$ in terms of elementary analytic functions.$\hfill\sqbullet$ 
\end{example}

\section{Special properties of the set of points of non-coherence}\label{s4}

We will keep the notations for the set of points of non coherence of a $C$-analytic set already introduced in Subsection \ref{dspnc}. The main result of this section is the following.

\begin{thm}\label{xy}
Let $X\subset\R^n$ be a $C$-analytic set (resp. a Nash set). Then $\cl(T(X))=\cl(T(X)\setminus N(X))=T(X)\cup N(X)$.
\end{thm}

Before proving the previous result we present some consequences (see Figures \ref{fig1a} and \ref{fig2a}).

\begin{remarks}
(i) The previous result implies $\dim(N(X)_x)<\dim(T(X)_x)$ for each $x\in N(X)$.

(ii) If $X\subset\R^n$ is a $C$-analytic set, $N(X)$ and $T(X)$ may have non-empty intersection, even if $X$ is $C$-irreducible. Consider for instance $X:=\{(x^2-y^2-x^3)^2(y-x)-z^2=0\}$. We have $T(X)=\{x^2-y^2-x^3=0,x-y>0\}\cup\{(0,0,0)\}$, whereas $N(X)=\{(0,0,0)\}$ (see Figure \ref{fig1b}).
\end{remarks}

\begin{cor}\label{txnx}
Let $X\subset\R^n$ be a $C$-analytic set (resp. a Nash set). The following conditions are equivalent:
\begin{itemize}
\item[(i)] $X$ is coherent.
\item[(ii)] $T(X)=\varnothing$.
\item[(iii)] $N(X)=\varnothing$.
\end{itemize}
\end{cor}
\begin{proof}
The implication (i) $\Longrightarrow$ (ii) follows from Corollary \ref{ij} in the analytic case and from Corollary \ref{ijn} in the Nash case. The implication (ii) $\Longrightarrow$ (iii) follows from Theorem \ref{xy}, whereas (iii) $\Longrightarrow$ (i) is a consequence of the definition of $N(X)$ and of coherent sheaf of ideals. 
\end{proof}

\begin{center}
\begin{figure}[ht]
\begin{minipage}{0.49 \textwidth} 
\begin{center}
\begin{tikzpicture}[scale=0.33]

\draw[line width=1.25pt] (-10,0) -- (10,0);
\draw[line width=1.25pt] (-9.85,6.975) -- (9.85,-6.975);
\draw[line width=1.25pt] (-10,0) arc(270:90:1cm and 3.5cm);
\draw[line width=1.25pt] (-10,7) arc(90:-90:1cm and 3.5cm);

\draw[line width=1.25pt] (10,0) arc(90:-90:1cm and 3.5cm);
\draw[line width=1.25pt,dashed] (10,-7) arc(270:90:1cm and 3.5cm);

\draw[fill=gray!70,opacity=0.3,draw=none] (0,0) -- (-10,7) arc(90:-90:1cm and 3.5cm) -- (0,0);

\draw[fill=gray!70,opacity=0.3,draw=none] (0,0) -- (10,0) arc(90:-90:1cm and 3.5cm) -- (0,0);

\draw[fill=gray!90,opacity=0.5,draw=none] (-10,0) arc(270:90:1cm and 3.5cm) -- (-10,7) arc(90:-90:1cm and 3.5cm);

\draw (0,0) node{$\bullet$};
\draw (-5,5) node{\Small $X$};
\draw (1.5,1) node{\Small$\Sing(X)$?};

\end{tikzpicture}
\caption{`Real vision' of a $C$-analytic set $X$\label{fig1a}}
\end{center}
\end{minipage}
\hfil
\begin{minipage}{0.49\textwidth}
\begin{center}
\vspace*{2mm}
\begin{tikzpicture}[scale=0.33]

\draw[line width=1.25pt] (-10,0) -- (10,0);
\draw[line width=1.25pt] (-9.85,6.975) -- (9.85,-6.975);
\draw[line width=1.25pt] (-10,0) arc(270:90:1cm and 3.5cm);
\draw[line width=1.25pt] (-10,7) arc(90:-90:1cm and 3.5cm);

\draw[line width=1.25pt] (10,0) arc(90:-90:1cm and 3.5cm);
\draw[line width=1.25pt,dashed] (10,-7) arc(270:90:1cm and 3.5cm);

\draw[fill=gray!70,opacity=0.3,draw=none] (0,0) -- (-10,7) arc(90:-90:1cm and 3.5cm) -- (0,0);

\draw[fill=gray!70,opacity=0.3,draw=none] (0,0) -- (10,0) arc(90:-90:1cm and 3.5cm) -- (0,0);

\draw[fill=gray!90,opacity=0.5,draw=none] (-10,0) arc(270:90:1cm and 3.5cm) -- (-10,7) arc(90:-90:1cm and 3.5cm);

\draw[line width=1.25pt, color=orange] (-10,0) -- (10,0);
\draw[line width=1.25pt, color=orange] (-9.85,3.975) -- (9.85,-3.975);
\draw[line width=1.25pt, color=orange] (-10,0) arc(270:90:0.6cm and 2cm);
\draw[line width=1.25pt, color=orange] (-10,4) arc(90:-90:0.6cm and 2cm);

\draw[line width=1.25pt, color=orange] (10,0) arc(90:-90:0.6cm and 2cm);
\draw[line width=1.25pt,dashed, color=orange] (10,-4) arc(270:90:0.6cm and 2cm);

\draw[fill=orange!70,opacity=0.3,draw=none] (0,0) -- (-10,4) arc(90:-90:0.6cm and 2cm) -- (0,0);

\draw[fill=orange!70,opacity=0.3,draw=none] (0,0) -- (10,0) arc(90:-90:0.6cm and 2cm) -- (0,0);

\draw[fill=orange!90,opacity=0.5,draw=none] (-10,0) arc(270:90:0.6cm and 2cm) -- (-10,4) arc(90:-90:0.6cm and 2cm);

\draw[line width=1.25pt, color=orange] (-10,0) -- (10,0);
\draw[line width=1.25pt, color=orange] (-9.85,-3.975) -- (9.85,3.975);
\draw[line width=1.25pt, color=orange] (-10,0) arc(90:270:0.6cm and 2cm);
\draw[line width=1.25pt, color=orange] (-10,-4) arc(-90:90:0.6cm and 2cm);

\draw[line width=1.25pt, color=orange] (10,0) arc(-90:90:0.6cm and 2cm);
\draw[line width=1.25pt,dashed, color=orange] (10,4) arc(90:270:0.6cm and 2cm);

\draw[fill=orange!70,opacity=0.3,draw=none] (0,0) -- (-10,-4) arc(-90:90:0.6cm and 2cm) -- (0,0);

\draw[fill=orange!70,opacity=0.3,draw=none] (0,0) -- (10,0) arc(-90:90:0.6cm and 2cm) -- (0,0);

\draw[fill=orange!90,opacity=0.5,draw=none] (-10,0) arc(90:270:0.6cm and 2cm) -- (-10,-4) arc(-90:90:0.6cm and 2cm);

\draw[line width=1.75pt, color=red] (-10,0) -- (10,0);



\draw (0,0) node{$\bullet$};
\draw (-5,5) node{\Small$X$};
\draw (0.25,1.2) node{\Small$N(X)$};
\draw (-5,0.75) node{\Small\color{red}$T(X)$};
\draw (3,4.5) node{\Small${\color{red}T(X)}\cup N(X)\subset\Sing(X)$};
\end{tikzpicture}
\caption{`Imaginary vision'\\ of $X$ inside $\widetilde{X}$\label{fig2a}}
\end{center}
\end{minipage}
\end{figure}
\end{center}

\begin{figure}[ht]
\begin{minipage}{0.49 \textwidth} 
\begin{center}
\includegraphics[width=0.6\textwidth]{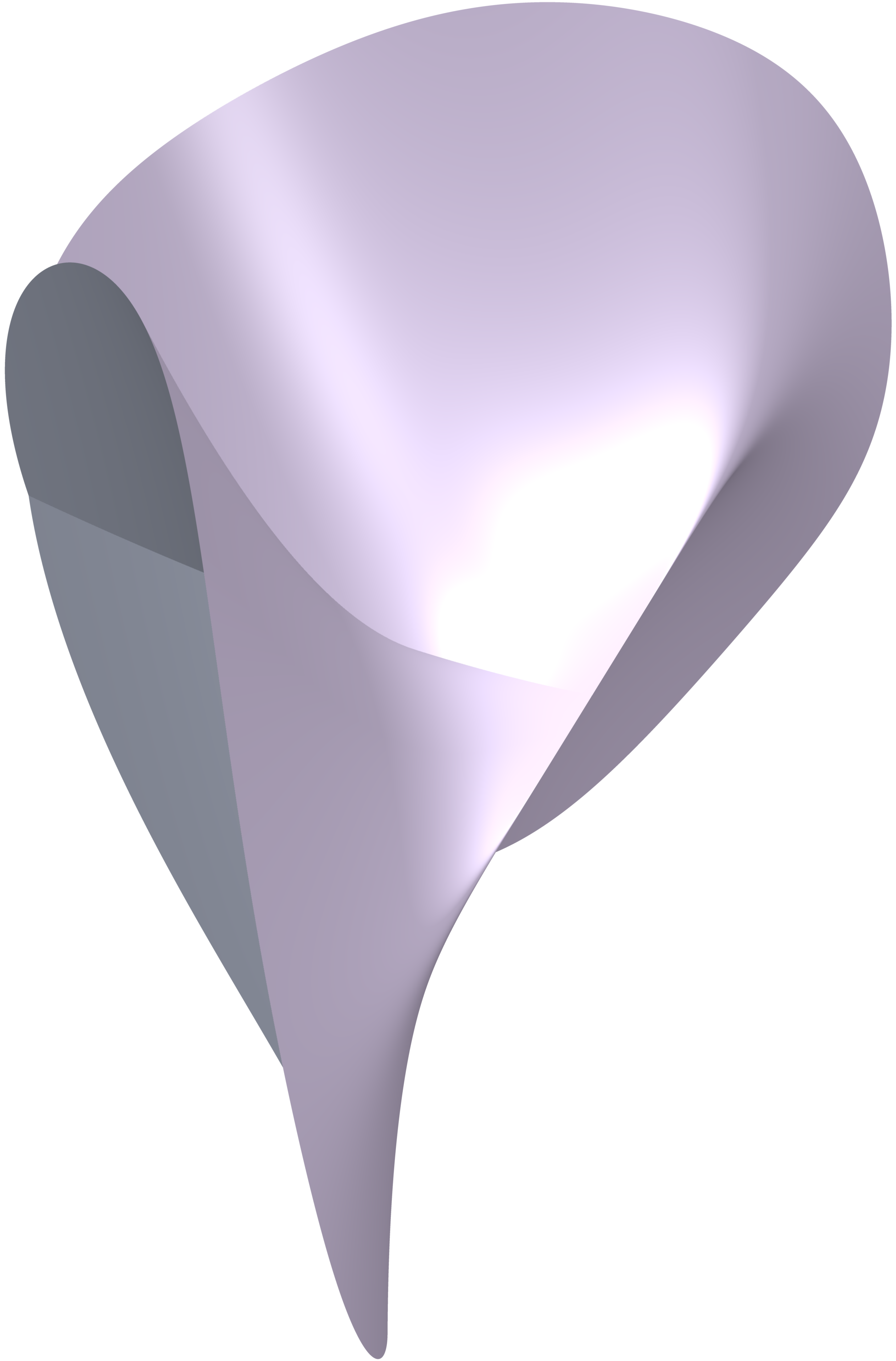}
\end{center}
\end{minipage}
\hfil
\begin{minipage}{0.49\textwidth}
\begin{center}
\begin{tikzpicture}[scale=1.2]
 \begin{axis}[
 axis equal,
 axis equal image,
 axis x line = middle,
 axis y line = middle,
 xticklabels={,,},
 yticklabels={,,},
 ]
 \addplot [domain = -1:1,samples=3000, unbounded coords=jump,black,thick,smooth]
 {sqrt(x*x-x*x*x)};
 \addplot [domain = -1:1,samples=3000, unbounded coords=jump,black,thick,smooth]
 {-sqrt(x*x-x*x*x)};
 \addplot [domain = -1.25:1.25,samples=3000,dashed, unbounded coords=jump,black,thick,smooth]
 {x};
 \addplot[soldot] coordinates{(0,0)};
 \end{axis}
 \draw[fill=gray!100,opacity=0.5,draw=none] (0,5.35) -- (5.05,5.35) -- (0,0.3);
\end{tikzpicture}
\end{center}
\end{minipage}
\caption{$X:(x^2-y^2-x^3)^2(y-x)-z^2=0$\label{fig1b}}
\end{figure}

We prove next that each `tail' of a $C$-analytic set is welded to $X$ through $N(X)$.

\begin{cor}\label{stx}
Let $X\subset\R^n$ be a $C$-analytic set and let $S$ be a connected component of $\cl(T(X))$. Then $S\cap N(X)\neq\varnothing$.
\end{cor}

To lighten the proof of Corollary \ref{stx} we extract the following result.

\begin{lem}\label{quotten}
Let $\gta,\gtb$ be ideals of $\an_{\R^n,x}$. Then $(\gta:\gtb)\otimes_\R\C=(\gta\otimes_\R\C:\gtb\otimes_\R\C)$.
\end{lem}
\begin{proof}
We have to prove $(\gta:\gtb)\an_{\C^n,x}=(\gta\an_{\C^n,x}:\gtb\an_{\C^n,x})$. If $f\in(\gta:\gtb)$, then $f\gtb\subset\gta$, so $f\gtb\an_{\C^n,x}\subset\gta\an_{\C^n,x}$. Thus, $f\in(\gta\an_{\C^n,x}:\gtb\an_{\C^n,x})$ and $(\gta:\gtb)\an_{\C^n,x}\subset(\gta\an_{\C^n,x}:\gtb\an_{\C^n,x})$.

Let $f\in(\gta\an_{\C^n,x}:\gtb\an_{\C^n,x})$ and write $f:=f_1+{\tt i}f_2$ where $f_1,f_2\in\an_{\R^n,x}$, see \S\ref{pr}. If $h\in\gta$, then $(f_1+{\tt i}f_2)h\in\gtb\an_{\C^n,x}$. As $\gtb$ is an ideal of $\an_{\R^n,x}$, we have $(f_1-{\tt i}f_2)h\in\gtb\an_{\C^n,x}$. Thus, $f_1h,f_2h\in\gtb\an_{\C^n,x}\cap\an_{\R^n,x}=\gtb$, so $f_1,f_2\in(\gta:\gtb)$ and $f=f_1+{\tt i}f_2\in(\gta:\gtb)\otimes_\R\C$, as required.
\end{proof}

We are ready to prove Corollary \ref{stx}.

\begin{proof}[Proof of Corollary \em\ref{stx}]
Suppose $S\cap N(X)=\varnothing$. Let $\Ff$ be the sheaf of ideals on $\R^n$ given by
$$
\Ff_x:=\begin{cases}
\an_{\R^n,x}&\text{if $x\in\R^n\setminus S$,}\\
(\Ii_{X,x}:\Jj_{X,x})&\text{if $x\in S$.}
\end{cases}
$$
As $T(X)$ is a $C$-semianalytic set, also $\cl(T(X))$ is a $C$-semianalytic set. We deduce by \cite[Ch.VIII.Prop.3.1]{abr} $\cl(T(X))$ is locally connected, so its connected components are both open and closed subset of $T(X)$. As $(\Ii_{X,x}:\Jj_{X,x})=\an_{\R^n,x}$ for each $x\in\R^n\setminus\cl(T(X))$ and both $\Jj_X$ and $\Ii_X$ are coherent sheaves on $S$ (because $S\cap N(X)=\varnothing$), we deduce by \cite[Ch.I.(6.12)]{g} that the transporter sheaf $(\Jj_{X}:\Ii_X)$ is coherent on an open neighborhood of $S$. Consequently, $\Ff$ is a coherent sheaf on $\R^n$ and its support $S\subset X$ is by \cite[Prop.15]{c2} a $C$-analytic subset of $\R^n$. As $S\subset\cl(T(X))\subset\Sing(X)$, no $C$-irreducible component of $X$ is contained in $S$. Let $\{X_i\}_{i\in I}$ be the $C$-irreducible components of $X$. By \cite[Prop.11]{wb} there exist an open neighborhood $\Omega\subset\C^n$ of $\R^n$ and a complexification $\widetilde{X}$ of $X$, which is an invariant complex analytic subset of $\Omega$, and its irreducible components $\{\widetilde{X}_i\}_{i\in I}$ are the complexifications of the irreducible components of $X_i$. This property can be kept after shrinking suitably $\Omega$ (if needed).
 
By \cite[Prop.2 \& 5]{c2} the coherent sheaf of $\an_{\C^n}$-ideals $\Ff\otimes_\R\C$ extends to a coherent invariant sheaf of $\an_{\C^n}$-ideals $\Gg$ on an invariant open Stein neighborhood of $\R^n$ in $\C^n$, which we may assume it is $\Omega$. By Lemma \ref{quotten} $\Gg_x=(\Ii_{X,x}\otimes_\R\C:\Jj_{X,x}\otimes_\R\C)$ for each $x\in S$. Thus, $\Gg_x\cdot(\Jj_{X,x}\otimes_\R\C)\subset\Ii_{X,x}\otimes_\R\C=\Jj_{\widetilde{X},x}^\C$ (see \S\ref{ccas0}), so $\widetilde{X}_x\subset\ZZ(\Gg_x)\cup\ZZ(\Jj_{X,x}\otimes_\R\C)$ for each $x\in S$.

The support of $\Gg$ is a complex analytic subset $Z$ of $\Omega$ such that $Z\cap\R^n=S$. Pick a point $x_0\in S$. As $\Jj_{X,x_0}\neq\Ii_{X,x_0}$, we also have $\Jj_{X,x_0}\otimes\C\neq\Ii_{X,x_0}\otimes\C=\Jj_{\widetilde{X},x_0}^\C$, so there exists an irreducible component $X_i$ of $X$ such that $x_0\in X_i$ and $\widetilde{X}_{i,x_0}\not\subset\ZZ(\Jj_{X,x_0}\otimes\C)$. As $\widetilde{X}_{x_0}\subset\ZZ(\Gg_{x_0})\cup\ZZ(\Jj_{X,x_0}\otimes_\R\C)$, there exists an irreducible component $T_{x_0}$ of $\widetilde{X}_{i,x_0}$ such that $T_{x_0}\subset\ZZ(\Gg_{x_0})=Z_{x_0}$. As $\widetilde{X}_i$ is irreducible and in particular pure dimensional, we deduce by the identity principle that $\widetilde{X}_i\subset Z$, so $X_i=\widetilde{X}_i\cap\R^n\subset Z\cap\R^n=S$, which is a contradiction. Consequently, $S\cap N(X)\neq\varnothing$, as required. 
\end{proof}

\subsection{Preliminary results}\label{pr}
Before proving Theorem \ref{xy} we need some preliminary results. Let $\Omega\subset\C^n$ be an invariant open set and let $F:\Omega\to\C$ be a holomorphic function. Recall that $\ol{F\circ\sigma}:\Omega\to\C$ is a holomorphic function on $\Omega$. Define 
$$
\Re(F):=\frac{F+\ol{F\circ\sigma}}{2}\quad\text{and}\quad\Im(F):=\frac{F-\ol{F\circ\sigma}}{2{\tt i}}
$$
the real and the imaginary part of $F$, which are invariant holomorphic functions on $\Omega$ and satisfy $F=\Re(F)+{\tt i}\Im(F)$. The same type of definitions hold for holomorphic germs at the points of $\Omega\cap\R^n$. If $Y_x\subset\C^n_x$ is a complex analytic germ, we denote $\Jj^\C_{Y,x}$ the ideal of complex analytic germs $F_x\in\an_{\C^n,x}$ such that $Y_x\subset\ZZ(F_x)$ (see \S\ref{ccas0}). 

We say that $Y_x$ is \em a complex Nash germ \em if there exists algebraic analytic function germs $F_{1,x},\ldots,F_{r,x}\in\an_{\C^n,x}$ such that $Y_x=\ZZ(F_{1,x},\ldots,F_{r,x})$. Recall that $F_x$ is an {\em algebraic analytic function germ} if there exists a non-zero polynomial $P\in\C[\x,\y]:=\C[\x_1,\ldots,\x_n,\y]$ such that $P(\x,f_x)=0$. If $x\in\R^n$, then $\Nn_{\R^n,x}\otimes_\R\C$ is the ring of complex algebraic analytic function germs at $x$. If $Y_x$ is an invariant complex Nash germ, $Y_x\cap\R^n_x$ is a Nash germ. In addition, if $\gta$ is a radical ideal of $\Nn_{\R^n,x}$, then $\ZZ(\gta\otimes_\R\C)$ defines an invariant complex Nash germ at $x$. By means of Artin's approximation theorem for the complex case one can imitate \cite[Cor.8.3.2 \& Prop.8.6.9]{bcr} almost {\em vervatim} to show that the (analytic) irreducible components $Y_{1,x},\ldots,Y_{r,x}$ of a complex Nash germs are again complex Nash germs. If $Y_x\subset\C^n_x$ is a complex Nash germ, we denote $\Jj^{\C\bullet}_{Y,x}$ the ideal of complex Nash germs $F_x\in\an_{\C^n,x}$ such that $Y_x\subset\ZZ(F_x)$ (see \S\ref{ccas0}). 

\subsubsection{Local properties of the set of `tails'}
We prove next the following result concerning the local behavior of the set of `tails'.

\begin{lem}\label{dimm}
Let $x\in\R^n$, let $\gta\subset\an_{\R^n,x}$ (resp. $\gta\subset\Nn_{\R^n,x}$) be a radical ideal and define $X_x:=\ZZ(\gta)$. Suppose there exists a union of irreducible component $Z_x$ of the invariant (complex) analytic set germ $Y_x:=\ZZ(\gta\otimes_\R\C)$ such that $\dim_\R(Z_x\cap\R^n_x)<\dim_\C(Z_x)$. Then $\gta\subsetneq\Jj_{X,x}$ (resp. $\gta\subsetneq\Jj^\bullet_{X,x}$).
\end{lem}

To lighten its proof we present before a preliminary result.

\begin{lem}\label{integers}
Let $A$ be a commutative ring that is a $K$-algebra over a field $K$ and let $\gta\subset A$ be an ideal. Let $a_1,\ldots,a_n\in A$ and assume that there exists a square matrix $M$ of order $n$ with coefficients in $K$ such that $\det(M)\neq0$ and all the entries of $M(a_1,\ldots,a_n)^t$ belong to $\gta$. Then $a_1,\ldots,a_n\in\gta$.
\end{lem}
\begin{proof}
Write $M:=(m_{ij})_{1\leq i,j\leq n}$ and define $b_i:=m_{i1}a_1+\cdots+m_{in}a_n\in\gta$, $a:=(a_1,\ldots,a_n)$ and $b:=(b_1,\ldots,b_n)$. Then $b^t=Ma^t$, so $a^t=M^{-1}b^t$ and as $b_1,\ldots,b_n\in\gta$, we conclude $a_1,\ldots,a_n\in\gta$, as required.
\end{proof}

We are ready to prove Lemma \ref{dimm}.

\begin{proof}[Proof of Lemma \em\ref{dimm}]
We prove first the {\sc Analytic case}. The ideal $\gta\otimes_\R\C$ is by \cite[Ch.V,\! \S15,\! Prop.5]{b} a radical ideal, so by R\"uckert's Nullstellensatz $\Jj_{Y_x,x}^\C=\gta\otimes_\R\C$. If $Y_{1,x},\ldots,Y_{r,x}$ are the irreducible components of $Y_x$, we have $\gta\otimes_\R\C=\bigcap_{i=1}^r\Jj^\C_{Y_i,x}$. Denote $X_{i,x}:=Y_{i,x}\cap\R^n$. As $Z_x$ is a union of irreducible components of $Y_x$ and $\dim_\R(Z_x\cap\R^n_x)<\dim_\C(Z_x)$, there exists $1\leq i_0\leq r$ such that $\dim_\R(X_{i_0,x})<\dim_\C(Y_{i_0,x})$. Let us prove: {\em There exists $h_x\in\Jj_{X,x}$ such that $\ZZ(h_x)=X_x$ and its invariant holomorphic extension $H_x$ to $\C^n_x$ does not belong to $\Jj^\C_{Y_{i_0},x}$.} Consequently, $h_x\in\Jj_{X,x}\setminus\Jj^\C_{Y_{i_0},x}\subset\Jj_{X,x}\setminus\bigcap_{i=1}^r\Jj^\C_{Y_i,x}=\Jj_{X,x}\setminus(\gta\otimes_\R\C)=\Jj_{X,x}\setminus\gta$.

Let $\gtb:=\Jj_{X_{i_0},x}\otimes_\R\C$, which is by \cite[Ch.V,\! \S15,\! Prop.5]{b} a radical ideal. Define $V_{i_0,x}:=\ZZ(\gtb)$ and as $\gtb$ is a radical ideal, we deduce by R\"uckert's Nullstellensatz $\Jj_{V_{i_0},x}^\C=\gtb$. Observe that $V_{i_0,x}$ is invariant (because $\Jj_{V_{i_0},x}^\C=\gtb$ admits a system of invariant generators). By \cite[IV.1.Prop.3]{n1} $\dim_\C(V_{i_0,x})=\dim_\R(X_{i_0,x})<\dim_\C(Y_{i_0,x})$, so $\Jj_{Y_{i_0},x}^\C\subsetneq\Jj_{V_{i_0},x}^\C=\gtb$. Let $H_{i_0,x}'\in\gtb\setminus\Jj_{Y_{i_0},x}^\C$. As $V_{i_0,x}$ is invariant, also $\ol{H_{i_0,x}'\circ\sigma}\in\gtb$, so $\Re(H_{i_0,x}'),\Im(H_{i_0,x}')\in\gtb$. As $H_{i_0,x}'=\Re(H_{i_0,x}')+{\tt i}\Im(H_{i_0,x}')$, we conclude that either $\Re(H_{i_0,x}')\not\in\Jj_{Y_{i_0},x}^\C$ or $\Im(H_{i_0,x}')\not\in\Jj_{Y_{i_0},x}^\C$, so after substituting $H_{i_0,x}'\in\gtb$ by either $\Re(H_{i_0,x}')$ or $\Im(H_{i_0,x}')$, we may assume $H_{i_0,x}'\in\gtb\setminus\Jj_{Y_{i_0},x}^\C$ is in addition invariant. Let $H_{i_0,x}''\in\gtb$ be invariant such that $\ZZ(H_{i_0,x}'')\cap\R^n_x=X_{i_0,x}$ (it is enough to consider the sum of squares of a finite system of invariant generators of $\gtb$) and pick a positive $\rho\in\R$ such that $H_{i_0,x}:=H_{i_0,x}'^2+\rho H_{i_0,x}''^2\not\in\Jj_{Y_{i_0},x}^\C$ (to prove the existence of $\rho$ use Lemma \ref{integers} applied to a $2\times2$ invertible matrix whose rows are of the type $(1,\rho_i)$ where $\rho_i>0$ is a positive real number for $i=1,2$). Thus, $h_{i_0,x}:=H_{i_0,x}|_{\R^n_x}\in\gtb\cap\an_{\R^n,x}=\Jj_{X_{i_0},x}$, $\ZZ(h_{i_0,x})=X_{i_0,x}$ and $H_{i_0,x}\not\in\Jj_{Y_{i_0},x}^\C$. 

For each $i=1,\ldots,r$ such that $X_{i,x}\subset X_{i_0,x}$ define $h_{i,x}:=h_{i_0,x}$, whereas if $X_{i,x}\not\subset X_{i_0,x}$ we claim: {\em There exists $h_{i,x}\in\Jj_{X_i,x}\setminus\Jj_{Y_{i_0},x}^\C$ such that $\ZZ(h_{i,x})=X_{i,x}$}. 

Pick a system of generators $G_{1,x},\ldots,G_{s,x}$ of $\Jj_{Y_i,x}^\C$ and observe that 
$$
X_{i,x}\subset\ZZ(\ol{G_{1,x}\circ\sigma},\ldots,\ol{G_{s,x}\circ\sigma})=\sigma(Y_{i,x}). 
$$
Consequently,
\begin{equation}\label{ysy}
\begin{split}
X_{i,x}&=Y_{i,x}\cap\R^n_x=Y_{i,x}\cap\sigma(Y_{i,x})\cap\R^n_x=\ZZ(G_{1,x},\ldots,G_{s,x},\ol{G_{1,x}\circ\sigma},\ldots,\ol{G_{s,x}\circ\sigma})\cap\R^n_x\\
&=\ZZ(\Re(G_{1,x}),\ldots,\Re(G_{s,x}),\Im(G_{1,x}),\ldots,\Im(G_{s,x}))\cap\R^n_x. 
\end{split}
\end{equation}
Choose positive real numbers $\lambda_1,\ldots,\lambda_s,\mu_1,\ldots,\mu_s\in\R$ such that $H_{i,x}:=\sum_{k=1}^s\lambda_k\Re(G_{k,x})^2+\mu_k\Im(G_{k,x})^2\not\in\Jj_{Y_{i_0},x}^\C$. Such positive real numbers exist, because otherwise by Lemma \ref{integers} (applied to a $(2s)\times(2s)$ invertible matrix whose coefficients are all real and strictly positive) $\Re(G_{k,x}),\Im(G_{k,x})\in\Jj_{Y_{i_0},x}$ for $k=1,\ldots,s$, so $\Jj_{Y_i,x}^\C\subset\Jj_{Y_{i_0},x}^\C$, which is a contradiction, because $i\neq i_0$ (recall that $X_{i,x}\not\subset X_{i_0,x}$) and both $Y_{i,x}$ and $Y_{i_0,x}$ are irreducible components of $Y_x$. 

As $X_{i,x}=\ZZ(\Re(G_{1,x}),\ldots,\Re(G_{s,x}),\Im(G_{1,x}),\ldots,\Im(G_{s,x}))\cap\R^n_x$ (see \eqref{ysy}), we deduce $X_{i,x}=\ZZ(h_{i,x})$, where $h_{i,x}:=H_{i,x}|_{\R^n_x}\in\Jj_{X_i,x}\setminus\Jj_{Y_{i_0},x}^\C$, as claimed.

By construction $H_{i,x}\not\in\Jj_{Y_{i_0},x}^\C$ for $i=1,\ldots,r$. As $\Jj_{Y_{i_0},x}^\C$ is a prime ideal, we conclude $H_x:=\prod_{i=1}^rH_{i,x}\not\in\Jj_{Y_{i_0},x}^\C$, whereas $h_x:=H_x|_{\R^n_x}\in\bigcap_{i=1}^r\Jj_{X_i,x}=\Jj_{X,x}$ and $\ZZ(h_x)=\ZZ(h_{1,x}\cdots h_{r,x})=\bigcup_{i=1}^rX_{i,x}=X_x$, as required. 

The proof in the {\sc Nash case} follows {\em vervatim} after changing $\an_{\R^n,x}$ by $\Nn_{\R^n,x}$, $\an_{\C^n,x}$ by $\Nn_{\R^n,x}\otimes_\R\C$, $\Jj^\C_{T,x}$ by $\Jj^{\C\bullet}_{T,x}$ (if $T$ is a complex Nash set) and $\Jj_{S,x}$ by $\Jj_{S,x}^\bullet$ (if $S$ is a Nash set).
\end{proof}

We provide an alternative geometric characterization of the set of `tails' of a $C$-analytic set.

\begin{cor}[Geometric characterization of the set of `tails']\label{txca}
Let $X\subset\R^n$ be a $C$-analytic set and let $\widetilde{X}$ be a complexification of $X$. Then $x\in T(X)$ if and only if there exists an irreducible component $T_x$ of $\widetilde{X}_x$ (depending on $x$) such that $\dim_\R(T_x\cap\R^n_x)<\dim_\C(T_x)$.
\end{cor}
\begin{proof}
The only if implication follows from Lemma \ref{tx0}, whereas the if implication follows from Lemma \ref{dimm} applied to $\gta=\Ii_{X,x}$.
\end{proof}

\subsubsection{Dimensional formula}
A subset $S$ of an open set $\Omega\subset\R^n$ is a \em semianalytic set \em if for each point $x\in\Omega$ there exists an open neighborhood $U^x$ such that $S\cap U^x$ is a finite union of sets of the type $\{f=0,g_1>0,\ldots,g_r>0\}\subset U^x$ where $f,g_i\in\an(U^x)$ are analytic functions on $U^x$. We recall the following result concerning dimensions of certain differences of semianalytic sets.

\begin{lem}\label{prop}
Let $\Omega\subset\R^n$ be an open set and let $A,B\subset\Omega$ be semianalytic sets. Then $\dim(\cl(A\setminus B)\cap B)<\dim(A\setminus B)$.
\end{lem}
\begin{proof}
By \cite[\S17.Prop.5, p.60]{l} it holds $\dim(\cl(A\setminus B)\setminus(A\setminus B))<\dim(A\setminus B)$. As 
$$
\cl(A\setminus B)\setminus(A\setminus B)=(\cl(A\setminus B)\cap B)\cup(\cl(A\setminus B)\setminus A), 
$$
we deduce $\dim(\cl(A\setminus B)\cap B)\leq\dim(\cl(A\setminus B)\setminus(A\setminus B))<\dim(A\setminus B)$, as required.
\end{proof}

We are ready to prove Theorem \ref{xy}.

\subsection{Proof of Theorem \ref{xy}}
We prove both the analytic and the Nash cases (pointing out the differences in the second case with respect to the first). We keep all the notations introduced in \S\ref{dspnc} for the description of the set of points on non-coherence of a $C$-analytic set. We claim: {\em It is enough to prove $N(X)\subset\cl(T(X)\setminus N(X))$.} 

As $T(X)=\{x\in X:\ \Jj_{X,x}\neq\Ii_{X,x}\}$ and $N(X)$ is the set of points $x\in X$ such that for each open neighborhood $U\subset X$ of $x$ the restriction sheaf $\Jj_{X}|_U$ is not of finite type, we deduce that $X\setminus(T(X)\cup N(X))$ is the set of points $x\in X$ such that there exists an open neighborhood $U^x\subset X$ of $x$ satisfying $\Jj_{X}|_{U^x}$ is of finite type and $\Jj_{X,x}=\Ii_{X,x}$. Thus, for each $x\in X\setminus(T(X)\cup N(X))$ there exists a perhaps smaller open neighborhood $V^x\subset U^x$ of $x$ such that $\Jj_{X}|_{V^x}=\Ii_{X}|_{V^x}$. Consequently, $X\setminus(T(X)\cup N(X))$ is an open subset of $X$, so $T(X)\cup N(X)$ is a closed subset of $X$. If $N(X)\subset\cl(T(X)\setminus N(X))$, we deduce $\cl(T(X))\subset T(X)\cup N(X)=(T(X)\setminus N(X))\cup N(X)\subset\cl(T(X)\setminus N(X))\subset\cl(T(X))$. Thus, all the equalities in the statement hold, as claimed.

We prove next $N(X)\subset\cl(T(X)\setminus N(X))$ in several steps:

\noindent{\sc Step 1.} {\em Initial preparation}. Define 
\begin{align*}
e:=&\max\{\dim(N_k(Z_k,R_k)):\ k=2,\ldots,d\}=\dim(N(X)),\\
\ell_e:=&\max\{k=2,\ldots,d:\ \dim(N_k(Z_k,R_k))=e\}.
\end{align*}
As $N(X)\neq\varnothing$, both integers exist. By Theorem \ref{ncp0} $N_{\ell_e}(Z_{\ell_e},R_{\ell_e})\setminus\bigcup_{k={\ell_e}+1}^dN_k(Z_k,R_k)$ is the set of points $z\in X$ such that $X_z$ has a non-coherent irreducible component of dimension ${\ell_e}$ and all the irreducible components of $X_z$ of higher dimension are coherent. As coherence and non-coherence are properties of local nature, $N(X\cap U)=N(X)\cap U$ for each open set $U\subset\R^n$. We have $\dim(N_k(Z_k,R_k))<\dim(N_{\ell_e}(Z_{\ell_e},R_{\ell_e}))$ for each $k={\ell_e}+1,\ldots,d$, so $S_{e,{\ell_e}}:=N_{\ell_e}(Z_{\ell_e},R_{\ell_e})\setminus\bigcup_{k={\ell_e}+1}^dN_k(Z_k,R_k)\neq\varnothing$ and has dimension $e$. Let $S_{e,{\ell_e}}^*$ be the set of points of dimension $e$ of $S_{e,{\ell_e}}$, which is a dense subset of the set of points of dimension $e$ of $N_{\ell_e}(Z_{\ell_e},R_{\ell_e})$.

\noindent{\sc Step 2.} Pick a point $x\in S_{e,{\ell_e}}^*$ and observe that the $C$-semianalytic germ $N_{\ell_e}(Z_{\ell_e},R_{\ell_e})_x$ has dimension $e$. Let $X_{1,x},\ldots,X_{s,x}$ be the irreducible components of $X_x$. We may assume:
\begin{itemize}
\item $X_{1,x},\ldots,X_{p,x}$ have dimension ${\ell_e}$.
\item $X_{p+1,x},\ldots,X_{m,x}$ have dimension $<{\ell_e}$.
\item $X_{m+1,x},\ldots,X_{s,x}$ have dimension $>{\ell_e}$.
\end{itemize}

By \cite[Ch.VIII.Prop.3.1]{abr} there exist: (1) $\veps_1,\ldots,\veps_n>0$ such that $x\in U^x:=\prod_{i=1}^n(-\veps_i,\veps_i)$ and (2) irreducible $C$-analytic sets $X_i\subset U^x$ such that $X_i$ is a representative of $X_{i,x}$ and $\dim(X_i)=\dim(X_{i,x})$ for $i=1,\ldots,r$. In case $X$ is a Nash set, the $X_i$ are also semialgebraic, so they are irreducible Nash sets (Lemma \ref{ncas}). As each $N_k(Z_k,R_k)$ is by Remark \ref{remnx}(i) closed, we may assume $N_k(Z_k,R_k)\cap U^x=\varnothing$ for $k={\ell_e}+1,\ldots,d$. This means that
$$
N(X\cap U^x)=N(X)\cap U^x=\bigcup_{k=2}^{\ell_e} N_k(Z_k,R_k)\cap U^x.
$$
We assume in addition $X_k$ is coherent in $U^x$ for $k=m+1,\ldots,r$. Let $V_0\subset\C^n$ be an invariant open (semialgebraic) neighborhood of $U^x$ such that $g_1,\ldots,g_r$ admit invariant holomorphic extensions $G_1,\ldots,G_r\in\an(V_0)$ to $V_0$.

Consider the Nash diffeomorphism
$$
\varphi:U^x=(-\veps,\veps)^n\to\R^n,\ (y_1,\ldots,y_n)\mapsto\Big(\frac{y_1}{\sqrt{\veps_1^2-y_1^2}},\ldots,\frac{y_n}{\sqrt{\veps_n^2-y_n^2}}\Big).
$$
As $\varphi$ is a local Nash diffeomorphism on an invariant open semialgebraic neighborhood of $B_n(0,1)$ in $\C^n\equiv\R^{2n}$, there exists by \cite[Lem.9.2]{bfr} invariant open semialgebraic neighborhoods $V\subset V_0$ of $B_n(0,1)$ and $W\subset\C^n\equiv\R^{2n}$ of $\R^n$ such that $\varphi$ extends to an invariant analytic (in fact invariant complex Nash) diffeomorphism $\Phi:V\to W$. Using $\Phi$ we may assume $U^x=\R^n$. After our preliminary reductions $N_k(Z_k,R_k)=\varnothing$ for $k={\ell_e}+1,\ldots,d$, so $N(X)=\bigcup_{k=2}^{\ell_e} N_k(Z_k,R_k)$.

If $z\in N_{\ell_e}(Z_{\ell_e},R_{\ell_e})$, then $X_z$ has a non-coherent irreducible component of dimension ${\ell_e}$. As $X_{m+1},\ldots,X_s$ are coherent and the irreducible $C$-analytic components of $X$ of dimension ${\ell_e}$ are $X_1,\ldots,X_p$, we deduce $N_{\ell_e}(Z_{\ell_e},R_{\ell_e})\subset\bigcup_{i=1}^pN_{\ell_e}(X_i,\varnothing)$. Consequently,
\begin{multline*}
N_{\ell_e}(Z_{\ell_e},R_{\ell_e})=\bigcup_{i=1}^p(N_{\ell_e}(Z_{\ell_e},R_{\ell_e})\cap N_{\ell_e}(X_i,\varnothing))\\ 
\leadsto\ N_{\ell_e}(Z_{\ell_e},R_{\ell_e})_x=\bigcup_{i=1}^p(N_{\ell_e}(Z_{\ell_e},R_{\ell_e})_x\cap N_{\ell_e}(X_i,\varnothing)_x).
\end{multline*}
We may assume $x\in N_{\ell_e}(Z_{\ell_e},R_{\ell_e})\cap N_{\ell_e}(X_1,\varnothing)$ and $\dim(N_{\ell_e}(Z_{\ell_e},R_{\ell_e})_x\cap N_{\ell_e}(X_1,\varnothing)_x)=e$, so $\dim(N_{\ell_e}(X_1,\varnothing)_x)\geq e$. 

\noindent{\sc Step 3.} {\em We prove $x\in\cl(T(X)\setminus N(X))$.} Let us recall (and simplify) how to construct $N_{\ell_e}(X_1,\varnothing)$. Write $Z=X_1$. As $Z$ is an irreducible $C$-analytic set of dimension ${\ell_e}$, we may assume that the complexification $\widetilde{Z}$ of $Z$ is irreducible. Let $\pi:Y\to\widetilde{Z}$ be the normalization of $\widetilde{Z}$ endowed with an anti-involution $\sigma:Y\to Y$, which is induced by the usual conjugation on $\widetilde{Z}$. Let $Y^\sigma:=\{y\in Y:\ \sigma(y)=y\}$ be the fixed part of $Y$, which is a $C$-analytic space (and in addition semialgebraic in some $\C^m$ if $X$ is a Nash set). Define
\begin{align*}
&Y^\sigma_{({\ell_e})}:=\cl(Y^\sigma\setminus\Sing(Y^\sigma)),\\
&C_1:=\pi^{-1}(Z)\setminus Y^\sigma,\quad C_2:=Y^\sigma\setminus Y^\sigma_{({\ell_e})},\\
&A_i:=\cl(C_i)\cap Y^\sigma_{({\ell_e})}
\end{align*}
It holds $N_{\ell_e}(Z,\varnothing)=\pi(A_1)\cup\pi(A_2)$. We consider here the real analytic structure of $\pi:Y\to\widetilde{Z}$ to give meaning to the structure of $C$-semianalytic set of $\pi^{-1}(X)$ and to compute (real) dimensions of the involved objects. As $\pi$ is a proper map with finite fibers, we have 
$$
0\leq e\leq\dim(N_{\ell_e}(Z,\varnothing))=\max\{\dim(\pi(A_1)),\dim(\pi(A_2))\}=\max\{\dim(A_1),\dim(A_2)\}
$$
As $Y^{\sigma}_{({\ell_e})}$ is a closed subset of $Y^\sigma$, we have 
\begin{align*}
\dim(A_1)&=\dim(\cl(C_1)\cap\cl(Y^{\sigma}_{({\ell_e})}))\leq\dim(\cl(\pi^{-1}(Z)\setminus Y^\sigma)\cap Y^\sigma),\\
\dim(A_2)&=\dim(\cl(C_2)\cap\cl(Y^\sigma_{({\ell_e})}))=\dim(\cl(Y^\sigma\setminus Y^\sigma_{({\ell_e})})\cap Y^\sigma_{({\ell_e})}).
\end{align*}
By Lemma \ref{prop} and as $\pi$ is an analytic proper map with finite fibers, we have
\begin{align}
\dim(\pi(A_1))=\dim(A_1)&\leq\dim(\cl(\pi^{-1}(Z)\setminus Y^\sigma)\cap Y^\sigma)<\dim(\pi^{-1}(Z)\setminus Y^\sigma),\label{a1}\\
\dim(\pi(A_2))=\dim(A_2)&=\dim(\cl(Y^\sigma\setminus Y^\sigma_{({\ell_e})})\cap Y^\sigma_{({\ell_e})})<\dim(Y^\sigma\setminus Y^\sigma_{({\ell_e})}).\label{a2}
\end{align}
Thus, either $\pi^{-1}(Z)\setminus Y^\sigma$ or $Y^\sigma\setminus Y^\sigma_{({\ell_e})}$ has dimension $>e$.

As $x\in N_{\ell_e}(Z,\varnothing)=\pi(A_1)\cup\pi(A_2)$ and $\dim(N_{\ell_e}(Z,\varnothing)_x)\geq e$, there exists a point $w\in A_1\cup A_2$ such that $\pi(w)=x$ and $\dim(A_{1,w}\cup A_{2,w})=\dim(N_{\ell_e}(Z,\varnothing)_x)\geq e$ (recall that $\pi$ is an analytic proper map with finite fibers). Using \eqref{a1} and \eqref{a2} we deduce that either $\dim(\pi^{-1}(Z)_w\setminus Y^\sigma_w)>e$ or $\dim(Y^\sigma_w\setminus Y^\sigma_{({\ell_e}),w})>e$. As $\dim(N(X))=e$ and $\pi$ is an analytic proper map with finite fibers, $\pi^{-1}(N(X))$ is a $C$-semianalytic set of dimension $e$. Thus, either $\dim((\pi^{-1}(Z)_w\setminus Y^\sigma_w)\setminus\pi^{-1}(N(X))_w)>e\geq0$ or $\dim((Y^\sigma_w\setminus Y^\sigma_{({\ell_e}),w})\setminus\pi^{-1}(N(X))_w)>e\geq0$. By the analytic curve selection lemma \cite[Rem.VII.4.3(b)]{abr} (in the analytic case) or the Nash curve selection lemma \cite[Prop.8.1.13]{bcr} (in the Nash case) there exists an analytic arc (resp. Nash arc) $\beta:(-1,1)\to\R^n$ such that $\beta(0)=x$ and either $\beta((0,1))\subset(\pi^{-1}(Z)\setminus Y^\sigma)\setminus\pi^{-1}(N(X))$ or $\beta((0,1))\subset(Y^\sigma\setminus Y^\sigma_{({\ell_e})})\setminus\pi^{-1}(N(X))$. Define $\alpha:(-1,1)\to X_1,\ t\mapsto (\pi\circ\beta)(t)$, which is an analytic arc (resp. Nash arc) such that $\alpha(0)=x$ and $\alpha((0,1))\subset T(X_1)\setminus N(X)$ (Lemma \ref{dot}), so $\Ii_{X_1,\alpha(t)}\subsetneq\Jj_{X_1,\alpha(t)}$ for each $t\in(0,1)$. We claim: {\em $\Ii_{X,\alpha(t)}\subsetneq\Jj_{X,\alpha(t)}$ for each $t\in(0,1)$}, so $\alpha((0,1))\subset T(X)\setminus N(X)$ and $x\in\cl(T(X)\setminus N(X))$.

As $\Ii_{X,\alpha(t)}\subset\bigcap_{i=1}^s\Ii_{X_i,\alpha(t)}$ and $\bigcap_{i=1}^s\Jj_{X_i,\alpha(t)}=\Jj_{X,\alpha(t)}$, it is enough to check $\Jj_{X,\alpha(t)}\setminus\Ii_{X_1,\alpha(t)}$ for each $t\in(0,1)$. As $\Ii_{X_1,\alpha(t)}\subsetneq\Jj_{X_1,\alpha(t)}$, there exists by Lemma \ref{txca} an irreducible component $T_{t,\alpha(t)}$ of $\widetilde{X}_{1,\alpha(t)}$ such that $\dim(T_{t,\alpha(t)}\cap\R^n_{\alpha(t)})<\dim(T_{t,\alpha(t)})$ for each $t\in(0,1)$. By Lemma \ref{dimm} we have $\Jj_{X,\alpha(t)}\setminus\Ii_{X_1,\alpha(t)}\neq\varnothing$ for each $t\in(0,1)$, as claimed. 

We have proved that $S_{e,{\ell_e}}^*\subset\cl(T(X)\setminus N(X))$, so $\cl(S_{e,{\ell_e}}^*)\subset\cl(T(X)\setminus N(X))$. 

Consider the $C$-analytic subset $X':=X\setminus\cl(S_{e,{\ell_e}}^*)$ of $\Omega:=\R^n\setminus\cl(S_{e,{\ell_e}}^*)$ and observe that $N(X')=N(X)\setminus\cl(S_{e,{\ell_e}}^*)$ and $T(X')\subset T(X)\setminus\cl(S_{e,{\ell_e}}^*)$, because $\widetilde{X'}\subset\widetilde{X}\setminus\cl(S_{e,{\ell_e}}^*)$. Denote
\begin{align*}
e':=&\max\{\dim(N_k(Z_k,R_k)):\ k=2,\ldots,d\}=\dim(N(X')),\\
\ell_{e'}:=&\max\{k=2,\ldots,d:\ \dim(N_k(Z_k,R_k))=e'\}
\end{align*}
and observe that $e'\leq e$ and $\ell_{e'}<\ell_e$ if $e'=e$, because $S_{e,{\ell_e}}^*$ is dense in the set of points of dimension $e$ of $N_{\ell_e}(Z_{\ell_e},R_{\ell_e})$. Thus, by induction hypothesis on the pair $(e,\ell_e)$ each $x\in N(X)\setminus\cl(S_{e,{\ell_e}}^*)$ belongs to $\cl(T(X')\setminus N(X'))\subset\cl(T(X)\setminus N(X))$ (because $N(X)=N(X')\cup\cl(S_{e,{\ell_e}}^*)$), as required.
\qed

\section{Special global equations outside the set of points of non coherence}\label{s5}

In this section we find global equations of a $C$-analytic set $X\subset\R^n$ outside the set $N(X)$ of points of non coherence that encodes the differential property of the points of $T(X)\setminus N(X)$.

\begin{thm}\label{h}
Let $X\subset\R^n$ be a $C$-analytic set (resp. Nash set) and let $N(X)$ be its set of points of non-coherence. Then there exists $h\in\Ii(X\setminus N(X))$ (resp. $h\in \Ii^\bullet(X\setminus N(X))$) such that $\ZZ(h)=X\setminus N(X)$ and $h_x\in\Jj_{X,x}\setminus\Ii_{X,x}$ for each $x\in T(X)\setminus N(X)$.
\end{thm}

The previous function $h$ is not unique, however we will see below (Theorem \ref{eq}) that under certain hypothesis ($X$ is either a $C$-analytic locally hypersurface or a Nash locally hypersurface) there is an $h$ satisfying the conditions of Theorem \ref{h} such that all the others are multiples of such $h$. To lighten the proof of Theorem \ref{h} we need the following preliminary result, whose proof is inspired by \cite[Lem.4.1]{abf1}.
\begin{lem}\label{series}
Let $\Omega\subset\C^n$ be an invariant open Stein set and let $\Gg$ be a coherent invariant sheaf of $\an_\Omega$-ideals on $\Omega$. Let $\{H_j\}_{j\geq1}\subset H^0(\Omega,\Gg)$ be a sequence. There exists positive real numbers $\nu_j>0$ for each $j\geq1$ such that: if $0<\mu_j\leq\nu_j$ and $p_j\geq1$ is a positive integer for each $j\geq1$, then $H:=\sum_{j\geq1}\mu_j^{2p_j}H_k^{2p_j}\in H^0(\Omega,\Gg)$.
\end{lem}
\begin{proof}
Let $\{L_\ell\}_{\ell\geq1}$ be an exhaustion of $\Omega$ by compact sets. Define $\lambda_j:=\max_{L_j}\{|H_j|^2\}+1$ and $\nu_j:=1/(\sqrt{2^j\lambda_j})$. Fix $0<\mu_j\leq\nu_j$ and a positive integer $p_j\geq1$ for each $j\geq1$. Consider the series $H:=\sum_{j\geq1}\mu_j^{2p_j}H_j^{2p_j}$. We claim: {\em $H$ converges uniformly on the compact subsets of $\Omega$}. 

Indeed, let $L\subset\Omega$ be a compact set and observe that there exists an index $\ell_0\geq1$ such that $L\subset L_\ell$ for all $\ell\geq\ell_0$. Moreover, for each $z\in L$ we have $\mu_\ell^2|H_\ell(z)|^2\leq\nu_\ell^2|H_\ell(z)|^2\leq\frac{1}{2^\ell}<1$, so $\mu_\ell^{2p_j}|H_\ell(z)|^{2p_j}\leq\mu_\ell^2|H_\ell(z)|^2\leq\frac{1}{2^\ell}$ if $\ell\geq\ell_0$. Thus,
$$
\Big|\sum_{\ell\geq\ell_0}\mu_\ell^{2p_j}H_\ell^{2p_j}(z)\Big|\leq\sum_{\ell\geq\ell_0}\mu_\ell^{2p_j}|H_\ell(z)|^{2p_j}\leq\sum_{\ell\geq\ell_0}\frac{1}{2^\ell}=\frac{1}{2^{\ell_0-1}}\leq 1
$$ 
for each $z\in L$, as claimed. 

Denote $S_m:=\sum_{j=1}^m\mu_j^{2p_j}H_j^{2p_j}\in H^0(\Omega,\Gg)$ and observe $H:=\sum_{j\geq1}\mu_j^{2p_j}H_j^{2p_j}=\lim_{m\to\infty}S_m$ in the Fr\'{e}chet topology of $H^0(\Omega,\Gg)$. As $H^0(\Omega,\Gg)$ is by \cite[VIII.Thm.11, pag.62]{c1} a closed ideal of $H^0(\Omega,\an_\Omega)$, we conclude $H\in H^0(\Omega,\Gg)$, as required.
\end{proof}

We denote $\C\{\t\}$ the ring of convergent series in one variable with coefficients in $\C$ and let $\omega:\C\{\t\}\to\N\cup\{+\infty\},\ f:=\sum_{k\geq 0}a_k\t^k\mapsto\omega(f):=\inf\{k\geq 0:\ a_k\neq0\}$ the function {\em order} in $\C\{\t\}$. We are ready to prove Theorem \ref{h}.

\begin{proof}[Proof of Theorem \em\ref{h}]
We prove both cases separately:

\noindent{\sc Analytic Case}. Consider the coherent sheaf of $\an_{\C^n}$-ideals $\Ff:=\Jj_X|_{\R^n\setminus N(X)}\otimes_R\C$ on the open set $\R^n\setminus N(X)$. By \cite[Prop.2 \& 5]{c2} the coherent sheaf of $\an_{\C^n}$-ideals $\Ff$ extends to a coherent invariant sheaf of $\an_{\C^n}$-ideals $\Gg$ on an invariant open Stein neighborhood $\Omega$ of $U$ in $\C^n$. Let $\{L_\ell\}_{\ell\geq1}$ be an exhaustion of $\Omega$ by compact sets. As $\Gg$ is invariant and coherent, Cartan's Theorem A implies that there exists a countable collection of invariant holomorphic sections $\{H_j\}_{j\geq1}\subset H^0(\Omega,\Gg)$ such that for each $\ell\geq1$ there exists $j(\ell)$ so that the analytic germs $H_{1,z},\ldots,H_{j(\ell),z}$ generate the ideal $\Gg_z$ for each $z\in L_\ell$. 

Let $D:=\{x_k\}_{k\geq1}\subset T(X)$ be a countable dense subset of $T(X)$. For each $k\geq1$ we have $\Jj_{X,x_k}\setminus\Ii_{X,x_k}\neq\varnothing$, so $(\Jj_{X,x_k}\otimes_\R\C)\setminus\Jj^\C_{\widetilde{X},x_k}\neq\varnothing$ (because $\Ii_{X,x_k}\otimes_\R\C=\Jj^\C_{\widetilde{X},x_k}$). Thus, $\widetilde{X}_{x_k}\setminus\ZZ(\{H_{j}:\ j\geq1\})\neq\varnothing$ for each $k\geq1$. By the curve selection lemma \cite[Rem.VII.4.3(b)]{abr} (using the underlying real structures of all the involved complex analytic sets) there exists an analytic curve $\gamma_k:(-1,1)\to\widetilde{X}$ such that $\gamma_k(0)=x_k$ and $\gamma_k((0,1))\subset\widetilde{X}\setminus\ZZ(\{H_{j}:\ j\geq1\})$ for each $k\geq1$.

For each $k\geq1$ let $j_k\geq1$ be such that the analytic series $H_{j_k}\circ\gamma_k\in\C\{\t\}\setminus\{0\}$. Denote $A_k:=H_{j_k}$ for each $k\geq1$. We claim: {\em There exists $H_0\in H^0(\Omega,\Gg)$ such that $H_0\circ\gamma_k\in\C\{\t\}\setminus\{0\}$ for each $k\geq1$, $h_0:=H_0|_{\R^n\setminus N(X)}\geq0$ and $X\setminus N(X)\subset\ZZ(h_0)$.}

Denote $p_1:=1$ and for each $k\geq 2$ let $p_k\geq1$ be such that $\omega(A_k^{2p_k}\circ\gamma_j)>\omega(A_j^{2p_j}\circ\gamma_j)$ for $1\leq j\leq k-1$. By Lemma \ref{series} there exists positive real numbers $\nu_k>0$ for each $k\geq1$ such that if $0<\mu_k\leq\nu_k$ for each $k\geq1$, then $H_0:=\sum_{k\geq1}\mu_k^{2p_k}A_k^{2p_k}\in H^0(\Omega,\Gg)$. We choose inductively each $0<\mu_k\leq\nu_k$ such that $\omega(\sum_{k=1}^m\mu_k^{2p_k}(A_k^{2p_k}\circ\gamma_j))\leq\omega(A_j^{2p_j}\circ\gamma_j)<+\infty$ for each $1\leq j\leq m$ and each $m\geq1$. We conclude $\omega(H_0\circ\gamma_j)\leq\omega(A_j^{2p_j}\circ\gamma_j)<+\infty$ for each $j\geq1$. Thus, $H_0\in H^0(\Omega,\Gg)$, $H_0\circ\gamma_k\in\C\{\t\}\setminus\{0\}$ for each $k\geq1$, $h_0:=H_0|_{\R^n\setminus N(X)}\geq0$ (because it is a sum of squares of invariant analytic functions) and $X\setminus N(X)\subset\ZZ(h_0)$ (because $H_0\in H^0(\Omega,\Gg)$), as claimed.

By Lemma \ref{series} there exist positive real numbers $\lambda_j>0$ such that $H':=\sum_{j\geq1}\lambda_j^2H_j^2\in H^0(\Omega,\Gg)$. For each $j\geq1$ denote $h_j:=\lambda_jH_j|_{\R^n\setminus N(X)}$ and write $h':=H'|_{\R^n\setminus N(X)}=\sum_{j\geq1}h_j^2$. It holds: $\ceros(h')=X\setminus N(X)$. 

Indeed,
$$
\ceros(h')=\bigcap_{j\geq1}\ceros(h_j)=\bigcap_{j\geq1}(\ceros(H_j)\cap\R^n)=\Big(\bigcap_{j\geq1}\ceros(H_j)\Big)\cap\R^n=\supp(\Gg)\cap\R^n=X\setminus N(X).
$$

For each $k\geq1$ let $m_k\geq1$ be an integer such that the tuple 
$$
(H_0\circ\gamma_k,H'\circ\gamma_k)\in\{\t^{m_k}\}\C\{\t\}^2\setminus\{\t^{m_k+1}\}\C\{\t\}^2
$$ 
and define $w_k:=(\frac{H_1\circ\gamma_k}{\t^{m_k}},\frac{H'\circ\gamma_k}{\t^{m_k}})\in\C\{\t\}^2\setminus\{\t\}\C\{\t\}^2$, which satisfies $w_k(0):=(w_{1k},w_{2k})\in\C^2\setminus\{0\}$. By Baire's Theorem there exists positive real numbers $\lambda_0,\lambda'>0$ such that $\lambda_0w_{1k}+\lambda' w_{2k}\neq0$ for each $k\geq1$. Consequently, $H:=\lambda_0H_0+\lambda'H'$ satisfies that
$$
\frac{H\circ\gamma_k}{\t^{m_k}}=\lambda_0\frac{H_1\circ\gamma_k}{\t^{m_k}}+\lambda'\frac{H'\circ\gamma_k}{\t^{m_k}}\in\C\{\t\}\setminus\{\t\}\C\{\t\}
$$
for each $k\geq1$. Thus, $H\circ\gamma_k\in\C\{\t\}\setminus\{0\}$ for each $k\geq1$ and we define $h:=H|_{\R^n\setminus N(X)}$, which is a positive semidefinite analytic function on $X\setminus N(X)$. We conclude $h_{x_k}\not\in\Jj_{X,x_k}\setminus\Ii_{X,x_k}$ for each $k\geq1$ and $\ZZ(h)=\ZZ(h_0)\cap\ZZ(h')=X\setminus N(X)$, so $h\in\Ii(X\setminus N(X))$. Suppose there exists $x\in T(X)$ such that $h_x\in\Ii_{X,x_k}$. Then there exists an open neighborhood $U^x\subset\R^n$ of $x$ such that $h|_{U^x}\in\Ii(X)\an(U^x)$. As $D$ is dense in $T(X)$, there exists some $k\geq1$ such that $x_k\in U^x$, so $h_{x_k}\in\Ii_{X,x_k}$, which is a contradiction. Consequently, $h_x\in\Jj_{X,x}\setminus\Ii_{X,x}$ for each $x\in T(X)$.

\noindent{\sc Nash Case}. As $X\setminus N(X)$ is a coherent Nash subset of $V:=\R^n\setminus N(X)$, Theorem \ref{summer}(A) assures that the global sections of the sheaf of $\Nn_{\R^n}$-ideals $\Jj_X^\bullet|_{V}$ generate $\Jj^\bullet_{X,x}$ for each $x\in\R^n\setminus N(X)$. As the ring of Nash functions $\Nn(V)$ is Noetherian, the ideal $\Ii^\bullet(X\setminus N(X))$ is generated by finitely many global sections $h_1,\ldots,h_s$. Thus, $\Jj^\bullet_{X,x}=\{h_1,\ldots,h_s\}\Nn_{\R^n,x}$ for each $x\in V$. 

Let $\Omega\subset\C^n$ be an invariant open neighborhood of $V$ such that $h_1,\ldots,h_s$ have invariant holomorphic extensions $H_1,\ldots,H_s\in\an(\Omega)$. Let $D:=\{x_k\}_{k\geq1}\subset T(X)$ be a countable dense subset of $T(X)$. For each $k\geq1$ we have $\Jj_{X,x_k}\setminus\Ii_{X,x_k}\neq\varnothing$ (because $\Jj_{X,x_k}^\bullet\setminus\Ii_{X,x_k}^\bullet\neq\varnothing$ and the inclusion $\Nn_{\R^n,\x_k}\hookrightarrow\an_{\R^n,\x_k}$ is by Corollary \ref{ffna} faithfully flat for each $k\geq1$), so $(\Jj_{X,x_k}\otimes_\R\C)\setminus\Jj^\C_{\widetilde{X},x_k}\neq\varnothing$ (because $\Ii_{X,x_k}\otimes_\R\C=\Jj^\C_{\widetilde{X},x_k}$). Thus, $\widetilde{X}_{x_k}\setminus\ZZ(\{H_1,\ldots,H_s\})\neq\varnothing$ for each $k\geq1$. By the curve selection lemma \cite[Rem.VII.4.3(b)]{abr} (using the underlying real structures of all the involved complex analytic sets) there exists an analytic curve $\gamma_k:(-1,1)\to\widetilde{X}$ such that $\gamma_k(0)=x_k$ and $\gamma_k((0,1))\subset\widetilde{X}\setminus\ZZ(\{H_1,\ldots,H_s\})$ for each $k\geq1$.

For each $k\geq1$ let $j_k\in\{1,\ldots,s\}$ be such that the analytic series $H_{j_k}\circ\gamma_k\in\C\{\t\}\setminus\{0\}$. For each $k\geq1$ let $m_k\geq1$ be an integer such that the tuple $(H_1^2\circ\gamma_k,\ldots,H_s^2\circ\gamma_k)\in\{\t^{m_k}\}\C\{\t\}^s\setminus\{\t^{m_k+1}\}\C\{\t\}^s$ and define $w_k:=(\frac{H_1^2\circ\gamma_k}{\t^{m_k}},\ldots,\frac{H_s^2\circ\gamma_k}{\t^{m_k}})\in\C\{\t\}^s\setminus\{\t\}\C\{\t\}^s$, which satisfies $w_k(0):=(w_{1k},\ldots,w_{sk})\neq0$. By Baire's Theorem there exists positive real numbers $\lambda_1,\ldots,\lambda_s>0$ such that $\lambda_1w_{1k}+\cdots+\lambda_sw_{sk}\neq0$ for each $k\geq1$. Consequently, $H:=\lambda_1H_1^2+\cdots+\lambda_sH_s^2$ satisfies that
$$
\frac{H\circ\gamma_k}{\t^{2m_k}}=\lambda_1\frac{H_1^2\circ\gamma_k}{\t^{2m_k}}+\cdots+\lambda_r\frac{H_s^2\circ\gamma_k}{\t^{2m_k}}\in\C\{\t\}\setminus\{\t\}\C\{\t\}
$$
for each $k\geq1$. Thus, $H\circ\gamma_k\in\C\{\t\}\setminus\{0\}$ for each $k\geq1$ and we define $h:=H|_V$. We conclude $h_{y_k}\not\in\Ii_{X,x_k}^\bullet$ for each $k\geq1$ and $\ZZ(h)=\ZZ(h_1,\ldots,h_s)=X\setminus N(X)$, so $h\in\Ii^\bullet(X\setminus N(X))$. Suppose there exists $x\in T(X)$ such that $h_x\in\Ii_{X,x_k}^\bullet$. Then there exists an open (semialgebraic) neighborhood $U^x\subset\R^n$ of $x$ such that $h|_{U^x}\in\Ii^\bullet(X)\Nn(U^x)$. As $D$ is dense in $T(X)$, there exists some $k\geq1$ such that $x_k\in U^x$, so $h_{x_k}\in\Ii_{X,x_k}^\bullet$, which is a contradiction. Consequently, $h_x\in\Jj_{X,x}^\bullet\setminus\Ii_{X,x}^\bullet$ for each $x\in T(X)$, as required.
\end{proof}

\subsection{\texorpdfstring{$C$}{C}-analytic and Nash locally hypersurfaces}
We say that a $C$-analytic set (resp. Nash set) $X\subset\R^n$ is a \em $C$-analytic locally hypersurface \em (resp. \em Nash locally hypersurface\em) if for each $x\in\R^n$ the ideal $\Jj_{X,x}$ of $\an_{\R^n,x}$ (resp. $\Jj_{X,x}^\bullet$ of $\Nn_{\R^n,x}$) is principal. In particular, a $C$-analytic locally hypersurface is pure dimensional. Thus, not every $C$-analytic hypersurface is a $C$-analytic locally hypersurface. For instance, Whitney's umbrella $W:=\{y^2-zx^2=0\}$ or Cartan's umbrella $C:=\{(x^2+y^2)z-y^3=0\}$ are $C$-analytic hypersurfaces that are not $C$-analytic locally hypersurfaces. However, the examples in \S\ref{examples} are $C$-analytic locally hypersurfaces. Let us show the existence of `good' global equations for $C$-analytic locally hypersurfaces (resp. Nash locally hypersurfaces) outside the set of points on non-coherence. 

\begin{lem}\label{eq}
Let $X\subset\R^n$ be a $C$-analytic locally hypersurface (resp. Nash locally hypersurface). Then there exists $h\in\an(\R^n\setminus N(X))$ (resp. $h\in\Nn(\R^n\setminus N(X))$) such that $h_x$ generates the ideal $\Jj_{X,x}$ (resp. $\Jj_{X,x}^\bullet$) of $\an_{\R^n,x}$ (resp. $\Nn_{\R^n,x}$) for each $x\in\R^n\setminus N(X)$.
\end{lem}
\begin{proof}
Assume first $X$ is a $C$-analytic locally hypersurface. Consider the exact sequence 
$$
0\to\an_{\R^n}^*\to\an_{\R^n}\setminus\{0\}\overset{\rho}{\to}\an_{\R^n}\setminus\{0\}/\an_{\R^n}^*\to0
$$
where $\an_{\R^n}^*$ is the sheaf of invertible analytic germs, whereas $\an_{\R^n}\setminus\{0\}/\an_{\R^n}^*$ is the sheaf of germs of positive divisors and $\rho$ is the canonical projection. The locally principal coherent sheaf of ideals $\Jj_X$ on $\R^n\setminus N(X)$ defines naturally an element $\sigma_X\in H^0(\R^n\setminus N(X),\an_{\R^n}\setminus\{0\}/\an_{\R^n}^*)$. Consider the long exact sequence
\begin{multline*}
0\to H^0(\R^n\setminus N(X),\an_{\R^n}^*)\to H^0(\R^n\setminus N(X),\an_{\R^n}\setminus\{0\})\\
\to H^0(\R^n\setminus N(X),\an_{\R^n}\setminus\{0\}/\an_{\R^n}^*)\to H^1(\R^n\setminus N(X),\an_{\R^n}^*).
\end{multline*}
We claim: $H^1(\R^n\setminus N(X),\an_{\R^n}^*)=0$. To that end, consider the exact sequence of sheaves
$$
0\to\an_{\R^n}\overset{\varphi}{\to}\an_{\R^n}^*\overset{\pi}{\to}\an_{\R^n}^*/\varphi(\an_{\R^n})\to0
$$
where $\varphi(f_x)=\exp(f_x)$ and $\pi(g_x)=g_x\varphi(\an_{\R^n,x})$. Observe that $\varphi(\an_{\R^n,x})$ is the set of all strictly positive analytic germs and $\an_{\R^n}^*=\varphi(\an_{\R^n,x})\sqcup-\varphi(\an_{\R^n,x})$, so $\pi(g_x)$ is determined by the sign of $g_x$. Thus, $\an_{\R^n}^*/\varphi(\an_{\R^n})$ is isomorphic to the constant sheaf $\Z_2$ and after identifying both sheaves, we may assume $\pi(g_x)=\frac{g_x}{|g_x|}$ for each $g_x\in\an_{\R^n}^*$. Consider the long exact sequence
\begin{multline*}
0\to H^0(\R^n\setminus N(X),\an_{\R^n})\to H^0(\R^n\setminus N(X),\an_{\R^n}^*)\to H^0(\R^n\setminus N(X),\Z_2)\\
\to H^1(\R^n\setminus N(X),\an_{\R^n})\to H^1(\R^n\setminus N(X),\an_{\R^n}^*)\to H^1(\R^n\setminus N(X),\Z_2)\to H^2(\R^n\setminus N(X),\an_{\R^n}).
\end{multline*}
By Cartan's Theorem B we have $H^1(\R^n\setminus N(X),\an_{\R^n})=0$ and $H^2(\R^n\setminus N(X),\an_{\R^n})=0$, so
$$
H^1(\R^n\setminus N(X),\an_{\R^n}^*)\cong H^1(\R^n\setminus N(X),\Z_2).
$$
To compute the last group we use universal coefficient theorem for cohomology \cite[Thm.3.2]{ha} and obtain the following exact sequence
$$
0\to{\rm Ext}_\Z^1(H_0(\R^n\setminus N(X),\Z),\Z_2)\to H^1(\R^n\setminus N(X),\Z_2)\to{\rm Hom}(H_1(\R^n\setminus N(X),\Z),\Z_2)\to0.
$$
As $N(X)\subset X$ has codimension $\geq2$, the $C$-semialnalytic set $N(X)\subset\R^n$ has codimension $\geq 3$ in $\R^n$, so $\R^n\setminus N(X)$ is connected and $H_0(\R^n\setminus N(X),\Z)=\Z$. As $\Z$ is a free group, 
$$
{\rm Ext}_\Z^1(H_0(\R^n\setminus N(X),\Z),\Z_2)={\rm Ext}_\Z^1(\Z,\Z_2)=0
$$
(see \cite[p.195]{ha}), so $H^1(\R^n\setminus N(X),\Z_2)$ is isomorphic to ${\rm Hom}(H_1(\R^n\setminus N(X),\Z),\Z_2)$. In addition, as $N(X)$ has codimension $\geq 3$ in $\R^n$, the first homotopy group $\pi_1(\R^n\setminus N(X))=0$, so its abelianization $H_1(\R^n\setminus N(X),\Z)=0$. Thus, ${\rm Hom}(H_1(\R^n\setminus N(X),\Z),\Z_2)=0$, so also $H^1(\R^n\setminus N(X),\Z_2)=0$. This means that the homomorphism
$$
H^0(\R^n\setminus N(X),\an_{\R^n}\setminus\{0\})\overset{\rho}{\to} H^0(\R^n\setminus N(X),\an_{\R^n}\setminus\{0\}/\an_{\R^n}^*)\to0
$$
is surjective. Consequently, there exists $h\in H^0(\R^n\setminus N(X),\an_{\R^n}\setminus\{0\})$ such that $\rho(h)=\sigma_X$. Thus, $h$ generates the ideal $\Jj_{X,x}$ for each $x\in\R^n\setminus N(X)$.

Assume next $X$ is a Nash subset of $\R^n$. By the previous case there exist $h\in\Ii(X\setminus N(X))$ that generates the ideal $\Jj_{X,x}$ for each $x\in\R^n\setminus N(X)$. Let $g\in\Ii^\bullet(X\setminus N(X))\setminus\{0\}$. Then there exists $b\in\an(\R^n\setminus N(X))$ such that $g=bh$. By \cite[Cor.2]{cs} there exists Nash functions $b_1,h_1\in\Nn(\R^n\setminus N(X))$ and analytic functions $u,v\in\an(\R^n\setminus N(X))$ such that $uv=1$, $b=ub_1$ and $h=vh_1$. Thus, $g=b_1h_1$. As $v$ is a unit of $\an(\R^n\setminus N(X))$, the Nash function $h_1$ generates the ideal $\Jj_{X,x}$ for each $x\in\R^n\setminus N(X)$. As $(h_1)\Nn_{\R^n,x}\an_{\R^n,x}=(h_1)\an_{\R^n,x}=\Jj_{X,x}=\Jj_{X,x}^\bullet\an_{\R^n,x}$ and the inclusion $\Nn_{\R^n,x}\hookrightarrow\an_{\R^n,x}$ is by Corollary \ref{ffna} faithfully flat, we deduce $\Jj_{X,x}^\bullet=(h_1)\Nn_{\R^n,x}$ for each $x\in X\setminus N(X)$, as required.
\end{proof}
\begin{remarks}\label{epoint}
Let $X\subset\R^n$ be a $C$-analytic locally hypersurface (resp. Nash locally hypersurface).

(i) The irreducible components of each analytic germ $X_x$ have dimension $n-1$. Thus, if $X_x$ is not coherent, it has a non-coherent irreducible component $X_{1,x}$ of dimension $n-1$. Following the notation in \S\ref{dspnc} we deduce by Theorem \ref{ncp0} that $N(X)=N_{n-1}(X,\varnothing)$. The proof of Theorem \ref{xy} in this particular case can be substantially simplified, because $\ell_e=n-1$ in all cases.

(ii) By Lemma \ref{eq} there exists $h\in\an(\R^n\setminus N(X))$ (resp. $h\in\Nn(\R^n\setminus N(X))$) such that $h_x$ generates the ideal $\Jj_{X,x}$ (resp. $\Jj_{X,x}^\bullet$) for each $x\in\R^n\setminus N(X)$. Thus, $h$ satisfies the requirements of the statement of Theorem \ref{h} for each $x\in T(X)\setminus N(X)$. In addition, any other function satisfying such requirements is a multiple of $h$, because $J_{X,x}=h_x\an_{\R^n,x}$ for each $x\in X\setminus N(X)$.\hfill$\sqbullet$
\end{remarks}

\section{Proofs of the main results}\label{s6}

In this section we prove Theorems \ref{main1} and \ref{main2}. As an application, we also prove Corollaries \ref{cons1} and \ref{cons2}. We begin with some preliminary results.

\subsection{Obstruction set of a meromorphic function}
In order to manage better meromorphic functions on a $C$-analytic set, we will use freely the following result.
\begin{lem}\label{trof}
Let $R$ be a reduced ring and let $Q(R):=\{\frac{f}{g}:\ a,b\in R,\ \text{$b$ is not a zero divisor}\}$ be its total ring of fractions. Two fractions $\frac{f_1}{g_1},\frac{f_2}{g_2}\in Q(R)$ are equal if and only if $f_1g_2-f_2g_1=0$.
\end{lem}
\begin{proof}
Let $S\subset R$ be the multiplicative set of the non-zero divisor of $R$. Then $Q(R)=S^{-1}R$ and $R\setminus S$ is the union of the minimal prime ideals of $R$. As $R$ is reduced, the zero ideal is the intersection of the minimal prime ideals of $R$. If $\frac{f_1}{g_1}=\frac{f_2}{g_2}$, there exists $g\in S$ such that $g(f_1g_2-f_2g_1)=0$. Let $\gtp$ be a minimal prime ideal of $R$. As $g(f_1g_2-f_2g_1)\in\gtp$ and $g\in R\setminus S\subset R\setminus\gtp$, we have $f_1g_2-f_2g_1\in\gtp$. As this happens for each minimal prime ideal $\gtp$ of $R$, we deduce that $f_1g_2-f_2g_1=0$, because it belongs to the intersection of the minimal prime ideals of $R$, which is the zero ideal. The converse is straightforward. 
\end{proof}

\subsubsection{}
Let $X$ be a $C$-analytic subset of $\R^n$, let $\zeta:X\to\R$ be a function and let $x\in X$. Suppose there exists analytic germs $f_x,g_x\in\an_{\R^n,x}$ (depending on $x$) such that $\zeta_x=\frac{f_x}{g_x}$ and $g_x$ does not belong to any minimal prime of $\Jj_{X,x}$. Observe that $\zeta$ is an analytic function at $x\in X$ if there exists $a_x\in\an_{\R^n,x}$ such that $\frac{f_x}{g_x}=-a_x$ on $X_x$ or, equivalently, if $f_x+a_xg_x\in\Jj_{X,x}$. This is equivalent to $f_x\in g_x\an_{\R^n,x}+\Jj_{X,x}$. 

\subsubsection{}
Let $\{X_i\}_{i\in I}$ be the $C$-analytic irreducible components of the $C$-analytic set $X$ and let $\zeta:=\frac{f}{g}$ be a meromorphic function on $X$, that is, $f,g\in\an(\R^n)$ and $g|_{X_i}\neq0$ for each $i\in I$. If $\zeta$ coincides with the restriction to $X$ of an analytic function on $\R^n$, there exists $a\in\an(\R^n)$ such that $\frac{f}{g}=-a$ on $X$ or, equivalently, $f\in g\an(\R^n)+\Ii(X)$. We have $f_x\in g_x\an_{\R^n,x}+\Ii(X)\an_{\R^n,x}=g_x\an_{\R^n,x}+\Ii_{X,x}$ for each $x\in X$. Consequently, if there exists a point $x\in X$ such that $f_x\not\in g_x\an_{\R^n,x}+\Ii_{X,x}$, the meromorphic function $\zeta$ admits no analytic extension to $\R^n$. We define the {\em obstructing set} of the meromorphic function $\zeta=\frac{f}{g}$ as ${\tt O}(\zeta):=\{x\in X:\ f_x\not\in g_x\an_{\R^n,x}+\Ii_{X,x}\}$. By Lemma \ref{trof} the obstructing set ${\tt O}(\zeta)$ does not depend on the meromorphic representation of $\zeta$. As $\Ii_X$ is a coherent sheaf of ideals, ${\tt O}(\zeta)$ is in addition a closed subset of $X$. 

In case $X\subset\R^n$ is a Nash set and $\zeta=\frac{f}{g}$, where $f,g$ are Nash functions and $g$ is not identically zero on any of the irreducible components of $X$, we define ${\tt O}^\bullet(\zeta):=\{x\in X:\ f_x\not\in g_x\an_{\R^n,x}^\bullet+\Ii_{X,x}^\bullet\}$. As the inclusion homomorphism $\Nn_{\R^n,x}\hookrightarrow\an_{\R^n,x}$ is by Corollary \ref{ffna} faithfully flat, ${\tt O}^\bullet(\zeta)={\tt O}(\zeta)$.

\begin{remark}
{\em A meromorphic function (resp. meromorphic Nash function) $\zeta$ has an analytic (resp. Nash) extension to $\R^n$ if and only if ${\tt O}(\zeta)=\varnothing$}. Consequently, ${\tt O}(\zeta)$ determines whether $\zeta$ admits or not an analytic (resp. Nash) extension to $\R^n$ and how far is $\zeta$ from having such analytic (resp. Nash) extension.

The `only if' implication is clear, so let us proof the `if' implication. If ${\tt O}(\zeta)=\varnothing$, then $\zeta$ defines an element of $H^0(\R^n,\an_{\R^n}/\Ii_X)$ (resp. of $H^0(\R^n,\Nn_{\R^n}/\Ii_X^\bullet)$). By Cartan's Theorem B (resp. Theorem \ref{summer}(B)) there exists an analytic function $a\in\an(\R^n)$ (resp. Nash function $a\in\Nn(\R^n)$) such that $a|_X=\zeta$.\hfill$\sqbullet$
\end{remark}

\subsection{Winning family of denominators}\label{wf}
Let $Y\subset\R^n$ be a $C$-analytic subset of $\R^n$ and let $\widetilde{Y}$ be an invariant Stein complexification of $Y$, which is a closed subset of an invariant open Stein neighborhood $\Omega\subset\C^n$ of $Y$. By \cite[Proof of Prop.15]{c2} there exist finitely many invariant $P_1,\ldots,P_m\in\an(\Omega)$ such that $\ZZ(P_1,\ldots,P_m)=\widetilde{Y}$. Consider the family of holomorphic functions 
\begin{equation}\label{pkl}
P_{\lambda}:=\lambda_1P_1^2+\cdots+\lambda_mP_m^2\in\an(\Omega)
\end{equation}
where $\lambda:=(\lambda_1,\ldots,\lambda_m)\in\C^m$. We are particularly interested in the case when $\lambda$ belongs to the open orthant $\Qq_m:=\{\lambda_1>0,\ldots,\lambda_m>0\}\subset\R^m$. In that case the holomorphic functions $P_{\lambda}$ are in particular invariant. By \cite[\S2.2]{abf0} and Oka's coherence theorems \cite[Ch.IV]{n1} if $\Omega\subset\C^n$ is a contractible open Stein neighborhood of $\R^n$, we can find a square-free invariant $P^*_{\lambda}\in\an(\Omega)$ such that $P^*_{\lambda}|P_{\lambda}$ and $\ZZ(P^*_{\lambda})=\ZZ(P_{\lambda})$.

\begin{lem}\label{trivial}
Let $Z_y$ be a complex analytic germ at $y\in\C^n$ and let $Z_{1,y},\ldots,Z_{r,y}$ be the irreducible components of $Z_y$. Suppose $\widetilde{Y}_y$ does not contain $Z_{i,y}$ for $i=1,\ldots,r$. Then there exists a proper algebraic subset $S\subsetneq\R^m$ such that ($\Qq_m\setminus S$ is dense in $\Qq_m$ and)
$\dim(Z_y\cap\ZZ(P_{\lambda,y}))<\dim(Z_y)$ for each $\lambda\in\Qq_m\setminus S$.
\end{lem}
\begin{proof}
Let $W\subset\C^n$ be an open neighborhood of $y$ and let $Z_1,\ldots,Z_r$ be irreducible complex analytic subsets of $W$ such that the irreducible components of $Z_y$ are $Z_{1,y},\ldots,Z_{r,y}$. By the curve selection lemma \cite[Rem.VII.4.3(b)]{abr} (using the underlying real structures of all the involved complex analytic sets) there exists analytic curves $\alpha_i:(-1,1)\to Z_i$ such that $\alpha_i(0)=y$ and $\alpha_i((0,1))\subset Z_i\setminus\widetilde{Y}$ for $i=1,\ldots,r$. As $\widetilde{Y}=\ZZ(P_1,\ldots,P_m)$ there exists $q_i\geq1$ such that $(P_1^2\circ\alpha_i,\ldots,P_m^2\circ\alpha_i)\in((\t^{q_i})\C[[\t]])^m\setminus((\t^{q_i+1})\C[[\t]])^m$, so $\zeta_i:=(\frac{(P_1^2\circ\alpha_i)}{\t^{q_i}},\ldots,\frac{(P_m^2\circ\alpha_i)}{\t^{q_i}})\in\C[[\t]]^m\setminus((\t)\C[[\t]]^m)$ for $i=1,\ldots,r$. Thus, the vectors $w_i:=(w_{i1},\ldots,w_{im})=\zeta_i(0)$ are not identically zero for $i=1,\ldots,r$. Consider $S:=\bigcup_{i=1}^r\{\lambda\in\R^m:\ \lambda_1w_{i1}+\cdots+\lambda_mw_{im}=0\}$, which is a proper algebraic subset of $\R^n$, so $\Qq_m\setminus S$ is a dense open subset of $\Qq_m$. If $\lambda\in\Qq_m\setminus S$, then $\prod_{i=1}^rP_{\lambda}\circ\alpha_i\neq0$. Consequently, $\dim(Z_i\cap \ZZ(P_{\lambda}))<\dim(Z_i)$ for each $i\in\{1,\ldots,r\}$, so $\dim(Z_y\cap \ZZ(P_{\lambda,y}))<\dim(Z_y)$, as required.
\end{proof}

\begin{examples}
(i) An example of application of the previous results is $Y:=\{q:=(q_1,\ldots,q_n)\}$, $\widetilde{Y}:=\{q\}$ and $P_i:=\x_i-q_i$ for $i=1,\ldots,n$.

(ii) Another example of application of the previous result (in the analytic case) is a discrete subset $Y:=\{y_k\}_{k\geq1}$ of $T(X)$. In this case $\widetilde{Y}=Y$. In order to provide finitely many equations to describe $\widetilde{Y}$ we consideran analytic diffeomorphism $\varphi:\R^n\to\R^n$ that transforms $Y$ onto the set $\Z\times\{(0,\overset{(n-1)}{\ldots},0)\}$, see \cite[Cor.3]{bks}. Take $P_1(x):=\sin(2\pi x_1)$, $P_2(x)=x_2$, \ldots, $P_n(x)=x_n$.\hfill$\sqbullet$
\end{examples}

\subsection{Winning equations around a \texorpdfstring{$C$}{C}-analytic subset (almost numerators)}\label{sh}
Let $X\subset\R^n$ be a $C$-analytic set and let $Y\subset X$ be a $C$-analytic subset. 

\subsubsection{}\label{sh1}
We have seen in Lemma \ref{h} that there exists an analytic function $h\in\an(\R^n\setminus N(X))$ (resp. Nash function $h\in\Nn(\R^n\setminus N(X))$) such that $\ZZ(h)=X\setminus N(X)$ and $h_x\in\Jj_{X,x}\setminus\Ii_{X,x}$ (resp. $h_x\in\Jj_{X,x}^\bullet\setminus\Ii_{X,x}^\bullet$) for each $x\in T(X)\setminus N(X)$. Thus, if $Y\subset X\setminus N(X)$, then $U_0:=\R^n\setminus N(X)$ is an open neighborhood (resp. open semialgebraic neighborhood) of $Y$ in $\R^n$ and $h\in\an(U_0)$ (resp. $h\in\Nn(U_0)$) satisfies $\ZZ(h)=X\cap U_0$ and $h_y\in\Jj_{X,y}\setminus\Ii_{X,y}$ (resp. $h_y\in\Jj_{X,y}^\bullet\setminus\Ii_{X,y}^\bullet$) for each $y\in Y\cap T(X)$. If $X$ is non-coherent, then $T(X)\setminus N(X)\neq\varnothing$ (Theorem \ref{xy} and Corollary \ref{txnx}).

\subsubsection{}
If $Y:=\{y_k\}_{k\geq1}$ is a discrete set, for each $y_k\in Y\cap T(X)$ there exists $h_{k,y_k}\in\Jj_{X,y_k}\setminus\Ii_{X,y_k}$ such that $\ZZ(h_{k,y_k})=X_{y_k}$ and $h'_{k,y_k}\in\Ii_{X,y_k}$ such that $\ZZ(h'_{k,y_k})=X_{y_k}$. Substituting $h_{k,y_k}$ by $h_{k,y_k}^2+h_{k,y_k}'^2$ we may assume $h_{k,y_k}\in\Jj_{X,y_k}\setminus\Ii_{X,y_k}$ and $\ZZ(h_{k,y_k})=X_{y_k}$ for each $y_k\in Y\cap T(X)$. For the points $y_k\in Y\setminus T(X)$ let $h_{k,y_k}\in\Ii_{X,y_k}=\Jj_{X,y_k}$ be such that $\ZZ(h_{k,y_k})=X_{y_k}$. Let $\{V_k\}_{k\geq1}$ be pairwise disjoint open neighborhoods in $\R^n$ of the point $y_k$ such that $h_{k,y_k}$ has an analytic representative $h_k$ in $V_k$ for each $k\geq1$. Define $U_0:=\bigsqcup_{k\geq1}V_k$ and 
$$
h:U_0\to\R,\ x\mapsto h_k(x)\ \text{if $x\in V_k$}.
$$
The function $h\in\an(U_0)$ satisfies $\ZZ(h)=X\cap U_0$ and $h_y\in\Jj_{X,x}\setminus\Ii_{X,x}$ for each $x\in Y\cap T(X)$. In this case we have no restriction with respect to the set $N(X)$. The previous construction holds analogously in the Nash case if we consider a finite set instead of a discrete set.

\subsection{Proof of Theorems \ref{main1} and \ref{main2}}
We recall first the $1$-cocycle and $1$-coboundary relations for a sheaf on $\R^n$ corresponding to an open covering of $\R^n$ constituted only by two open sets. 

\begin{remark}\label{2os}
Let ${\mathfrak U}:=\{U_0,U_1\}$ be an open covering of $\R^n$. Let $\Ff$ be a sheaf on $\R^n$. Pick $i,j,k\in\{0,1\}$ and sections $f_{ij}\in\Ff(U_i\cap U_j),f_{ik}\in\Ff(U_i\cap U_k),f_{jk}\in\Ff(U_j\cap U_k)$. All possible intersections of three elements of ${\mathfrak U}$ and the corresponding $1$-cocycle relations are the following:
\begin{align*}
&U_0\cap U_0\cap U_0\ \leadsto\ 0=f_{00}-f_{00}+f_{00}=f_{00},\\
&U_0\cap U_0\cap U_1\ \leadsto\ 0=f_{01}-f_{01}+f_{00}=f_{00},\\
&U_0\cap U_1\cap U_0\ \leadsto\ 0=f_{10}-f_{00}+f_{01}=f_{10}+f_{01},\\
&U_0\cap U_1\cap U_1\ \leadsto\ 0=f_{11}-f_{01}+f_{01}=f_{11},\\
&U_1\cap U_0\cap U_0\ \leadsto\ 0=f_{00}-f_{10}+f_{10}=f_{00},\\
&U_1\cap U_0\cap U_1\ \leadsto\ 0=f_{01}-f_{11}+f_{10}=f_{10}+f_{01},\\
&U_1\cap U_1\cap U_0\ \leadsto\ 0=f_{10}-f_{10}+f_{11}=f_{11},\\
&U_1\cap U_1\cap U_1\ \leadsto\ 0=f_{11}-f_{11}+f_{11}=f_{11}.
\end{align*}
Thus, we obtain the following relations $f_{00}=0$ on $U_0$, $f_{11}=0$ on $U_1$ and $f_{10}=-f_{01}$ on $U_0\cap U_1$. Given $f_0\in\Ff(U_0)$ and $f_1\in\Ff(U_1)$, the $1$-coboundary relations are the following:
\begin{align*}
&U_0\cap U_0\ \leadsto\ f_{00}=f_0-f_0=0,\\
&U_0\cap U_1\ \leadsto\ f_{01}=f_1-f_0,\\
&U_1\cap U_0\ \leadsto\ f_{10}=f_0-f_1,\\
&U_1\cap U_1\ \leadsto\ f_{11}=f_1-f_1=0.
\end{align*}
In this situation a $1$-cocycle 
$$
(f_{00},f_{01},f_{10},f_{11})=(0,f_{01},-f_{01},0)\in\Ff(U_0\cap U_0)\times\Ff(U_0\cap U_1)\times\Ff(U_1\cap U_0)\times\Ff(U_1\cap U_1)
$$ 
is a $1$-coboundary if and only if there exists $f_i\in\Ff(U_i)$ for $i\in\{0,1\}$ such that $f_{01}=f_1-f_0$.$\hfill\sqbullet$
\end{remark}

As we have pointed out in the Introduction, cohomology will not help in the Nash setting, because even $H^1(\R,\Nn_\R)\neq0$, see \cite{hb}. Thus, we will use cohomology to prove Theorem \ref{main1} and find the mentioned meromorphic function that is local analytic on the $C$-analytic set $X\subset\R^n$, but does not admit an analytic extension to $\R^n$. Then we will adapt the last part of the latter construction to prove Theorem \ref{main2} finding a meromorphic Nash function on $\R^n$ that is local Nash on the Nash set $X\subset\R^n$, but does not admit an analytic extension to $\R^n$, skipping however the cohomological procedure. We are ready to prove the main results of this article.

We will prove simultaneously stronger results than Theorems \ref{main1} and \ref{main2} and we provide next the precise statements. When reading the following results the reader should have in mind Theorem \ref{xy}, Theorem \ref{h} and \S\ref{sh}, because they guarantee the existence of the $C$-analytic sets (resp. Nash sets) $Y$ in the statementss. 

\begin{thm}\label{main11}
Let $X\subset\R^n$ be a $C$-analytic set with $N(X)\neq\varnothing$. Let $Y\subset X$ be a $C$-analytic subset that contains no $C$-irreducible component of $X$ and meets $T(X)$, let $U_0\subset\R^n$ be an open neighborhood of $Y$ and let $h\in H^0(U_0,\Jj_X)$ be such that $h_y\in\Jj_{X,y}\setminus\Ii_{X,y}$ for each $y\in Y\cap T(X)$. Then there exist a meromorphic function $\zeta:\R^n\dashrightarrow\R$ with set of poles $Y$ such that $\zeta|_X$ is an analytic function on $X$ and ${\tt O}(\zeta)=Y\cap T(X)$.
\end{thm}

\begin{thm}\label{main21}
Let $X\subset\R^n$ be a Nash set with $N(X)\neq\varnothing$. Let $Y\subset X$ be a Nash subset that contains no irreducible component of $X$ and meets $T(X)$, let $U_0\subset\R^n$ be an open semialgebraic neighborhood of $Y$ and let $h\in H^0(U_0,\Jj_X^\bullet)$ be such that $h_y\in\Jj_{X,y}^\bullet\setminus\Ii_{X,y}^\bullet$ for each $y\in Y\cap T(X)$. Then there exist a meromorphic Nash function $\zeta:\R^n\dashrightarrow\R$ with set of poles $Y$ such that $\zeta|_X$ is a local Nash function on $X$ and ${\tt O}(\zeta)=Y\cap T(X)$.
\end{thm}
\begin{remark}\label{many}
As there are `many' different possible intersections $Y\cap T(X)$ for $C$-analytic subsets $Y$ of $X$, we deduce by Theorem \ref{main11} (resp. Theorem \ref{main21}) that there exists `many' different meromorphic functions (resp. Nash meromorphic functions) on a non-coherent $C$-analytic set (resp. Nash set) $X$ that are analytic (resp. local Nash) on $X$, but have no analytic extensions (resp. Nash extensions) to $\R^n$, see the statements of Theorems \ref{main1} and \ref{main2}. This also provides enlightening information in the analytic case about the first cohomology group $H^1(X,\Jj_X)$.
\hfill$\sqbullet$
\end{remark}

\begin{proof}[Proofs of Theorems {\em\ref{main11}} and {\em\ref{main21}}]
The proof of both results is conducted in several steps:

\noindent{\sc Step 1.} {\em Initial preparation.} Consider the exact sequence of sheaves
$$
0\to\Jj_X\to\an_{\R^n}\to\an_{\R^n}/\Jj_X\to0
$$
and the exact long corresponding sequence of global sections
\begin{multline*}
0\to H^0(\R^n,\Jj_X)\to H^0(\R^n,\an_{\R^n})\to H^0(\R^n,\an_{\R^n}/\Jj_X)\to H^1(\R^n,\Jj_X)\to H^1(\R^n,\an_{\R^n})\\
\to\cdots\to H^p(\R^n,\an_{\R^n})\to H^p(\R^n,\an_{\R^n}/\Jj_X)\to H^{p+1}(\R^n,\Jj_X)\to H^{p+1}(\R^n,\an_{\R^n}).
\end{multline*}
By Cartan's Theorem B, we have $H^p(\R^n,\an_{\R^n})=0$ for each $p\geq1$. Consequently, 
$$
H^p(\R^n,\an_{\R^n}/\Jj_X)\cong H^{p+1}(\R^n,\Jj_X)
$$ 
for $p\geq1$ and we have a (initial) exact sequence
\begin{equation}\label{delta}
0\to H^0(\R^n,\Jj_X)\to H^0(\R^n,\an_{\R^n})\to H^0(\R^n,\an_{\R^n}/\Jj_X)\overset{\delta}{\to} H^1(\R^n,\Jj_X)\to0.
\end{equation}
The homomorphism $H^0(\R^n,\an_{\R^n})\to H^0(\R^n,\an_{\R^n}/\Jj_X)$ is surjective if and only if $H^1(\R^n,\Jj_X)=0$. If $X$ is coherent, the sheaf of $\an_{\R^n}$-ideals $\Jj_X$ is coherent and by Cartan's Theorem B we have $H^1(\R^n,\Jj_X)=0$, so we will prove: \em if $X$ is non-coherent, then $H^1(\R^n,\Jj_X)\neq0$\em. As $H^1(\R^n,\Jj_X)=\displaystyle\lim_{\longrightarrow}H^1({\mathfrak V},\Jj_X)$, where ${\mathfrak V}$ runs over the open coverings of $\R^n$ and $H^1({\mathfrak V},\Jj_X)$ is the \u{C}ech cohomology group of $\Jj_X$ associated to ${\mathfrak V}$, it suffices to show that $H^1({\mathfrak U},\Jj_X)\neq0$ for coverings ${\mathfrak U}$ constituted by two open sets (see Remark \ref{2os}). 

\noindent{\sc Step 2.} {\em Construction of a $1$-cocycle that it not a $1$-coboundary.} Let $U_0$ be the open set provided in the statement, define $U_1:=\R^n\setminus Y$ and observe that ${\mathfrak U}:=\{U_0,U_1\}$ is an open covering of $\R^n$. We are going to construct a $1$-cocycle $f_{01}\in H^0(U_0\cap U_1,\Jj_X)$ such that for each pair $(f_0,f_1)\in H^0(U_0,\Jj_X)\times H^0(U_1,\Jj_X)$ it holds $f_{01}\neq f_1-f_0$. 

Let $\widetilde{Y}$ be a complexification of $Y$ that is a closed subset of an invariant contractible open Stein neighborhood $\Omega\subset\C^n$ of $\R^n$. Let $P_1,\ldots,P_m\in\an(\Omega)$ be invariant complex analytic functions such that $\widetilde{Y}=\ZZ(P_1,\ldots,P_m)$, use \cite[Proof of Prop.15]{c2}. In the Nash setting we pick a finite system of generators $p_1,\ldots,p_m$ of the ideal $\Ii(Y)$, consider analytic extensions $P_1,\ldots,P_m$ of $p_1,\ldots,p_m$ to an open semialgebraic neighborhood $\Omega$ of $\R^n$ in $\C^n$ and define $\widetilde{Y}:=\ZZ(P_1,\ldots,P_m)$, which is a complexification of $Y$. As $\Nn_{\R^n,y}\hookrightarrow\an_{\R^n,y}$ is by Corollary \ref{ffna} faithfully flat for each $y\in\R^n$, in the Nash case we have $h_y\in\Jj_{X,y}\setminus\Ii_{X,y}=(\Jj_{X,y}^\bullet\an_{\R^n,x})\setminus(\Ii_{X,y}^\bullet\an_{\R^n,x})$ for each $y\in Y\cap T(X)$ as well. For each $\lambda\in\Qq_m:=\{\lambda_1>0,\ldots,\lambda_m>0\}$ define
$$
P_{\lambda}(\z):=\lambda_1P_1^2+\cdots+\lambda_mP_m^2\quad\text{and}\quad f_{01,\lambda}(\z):=\frac{h(\z)}{P_{\lambda}(\z)}\in H^0(U_0\cap U_1,\Jj_X),
$$
where $U_0\cap U_1=U_0\setminus Y$. Substituting $P_{\lambda}$ by a suitable divisor, we may assume $P_{\lambda}$ has no multiple factors (see \S\ref{wf}).

Pick $f_0\in H^0(U_0,\Jj_X)$ and $f_1\in H^0(U_1,\Jj_X)$ and assume that $f_{01,\lambda}=f_1-f_0$ on $U_0\cap U_1$. Thus, $h=f_1P_{\lambda}-f_0P_{\lambda}$ and define $g_{\lambda}:=f_1P_{\lambda}=h+f_0P_{\lambda}$. The latter function is analytic on $\R^n=U_0\cup U_1$, because $P_{\lambda}$ is analytic on $\R^n$, $f_1$ is analytic on $U_1=\R^n\setminus Y$ and both $h,f_0$ are analytic on $U_0$ (which contains $Y$). As $f_1|_{X\setminus Y}=0$, $h|_{X\cap U_0}=0$, $f_0|_{X\cap U_0}=0$, $P_{\lambda}|_{\R^n}\in\an(\R^n)$ and $Y\subset X\cap U_0$, we deduce that $g_{\lambda}|_X=0$. Let $G_{\lambda}$ be an invariant analytic extension of $g_{\lambda}$ to an invariant open neighborhood $\Omega'\subset\Omega$ of $\R^n$. As $g_{\lambda}|_X=0$, we may assume by Lemma \ref{fx} that $G_{\lambda}|_{\widetilde{X}\cap\Omega'}=0$. Let $V_0\subset\C^n$ be an invariant open neighborhood of $U_0$ such that $h,f_0$ extend to invariant holomorphic functions $H,F_0$ on $V_0$. 

Let $D:=\{y_j\}_{j\geq1}$ be a dense subset of $Y\cap T(X)$. For each $j\geq1$ we have $h_{y_j}\in\Jj_{X,y_j}\setminus\Ii_{X,y_j}$, so $H_{y_j}\not\in\Ii_{X,y_j}\otimes_\R\C=\Jj_{\widetilde{X},y_j}^\C$. Thus, for each $j\geq1$ there exists an irreducible component $Z_{j,y_j}$ of $\widetilde{X}_{y_j}$ such that $Z_{j,y_j}\not\subset\ZZ(H_{y_j})$. If $X_1$ is a $C$-irreducible component of $X$ of dimension $1$, then $\Jj_{X_1}=\Ii_{X_1}$, so $\dim_\C(Z_{j,y_j})\geq2$ for each $j\geq1$. By \cite[Cor.10.9]{ei} we deduce $\dim_\C(Z_{j,y_j}\cap\ZZ(H_{y_j}))=\dim_\C(Z_{j,y_j})-1\geq1$. By Lemma \ref{trivial}, \cite[Cor.10.9]{ei} and Baire's Theorem there exists $\lambda_0\in\Qq_n$ such that $\dim_\C(Z_{j,y_j}\cap\ZZ(H_{y_j})\cap\ZZ(P_{\lambda_0,y_j}))=\dim_\C(Z_{j,y_j})-2\geq0$ and $\dim_\C(Z_{j,y_j}\cap\ZZ(P_{\lambda_0,y_j}))=\dim_\C(Z_{j,y_j})-1\geq1$ for each $j\geq1$. Thus, $\dim_\C((\widetilde{X}_{y_j}\cap\ZZ(P_{\lambda_0,y_j}))\setminus\ZZ(H_{y_j}))\geq1$ for each $j\geq1$. The set of possible values of $\lambda_0$ is a countable intersection of open dense subsets of $\Qq_m$ (and it is itself a dense subset of $\Qq_m$), but we fix one $\lambda_0$ for the following.

Pick a point $y\in Y\cap T(X)$ and suppose $\widetilde{X}_y\cap\ZZ(P_{\lambda_0,y})\subset\ZZ(H_y)$. There exists an open neighborhood $W^y$ of $y$ such that $\widetilde{X}\cap\ZZ(P_{\lambda_0})\cap W^y\subset\ZZ(H)\cap W^y$. As $D$ is dense in $Y$, we find $y_j\in D\cap W^y$ such that $\widetilde{X}_{y_j}\cap\ZZ(P_{\lambda_0,y_j})\subset\ZZ(H_{y_j})$, whereas $\dim_\C((\widetilde{X}_{y_j}\cap\ZZ(H_{y_j}))\setminus\ZZ(H_{y_j}))\geq1$ for each $j\geq1$, which is a contradiction. Thus, $\dim_\C((\widetilde{X}_{y}\cap\ZZ(P_{\lambda_0,y}))\setminus\ZZ(H_{y}))\geq1$ for each $y\in Y\cap T(X)$.

Fix a point $y\in Y\cap T(X)$. By the curve selection lemma \cite[Rem.VII.4.3(b)]{abr} (using the underlying real structures of all the involved complex analytic sets) there exists an analytic curve $\beta^y:(-1,1)\to\widetilde{X}$ such that $\beta^y(0)=y$ and $\beta^y((0,1))\subset(\widetilde{X}\cap\ZZ(P_{\lambda_0}))\setminus\ZZ(H)$. We have $G_{\lambda_0}\circ\beta^y=0$, $P_{\lambda_0}\circ\beta^y=0$, whereas $H\circ\beta^y\neq0$. Consequently, 
$$
0=G_{\lambda_0}\circ\beta^y=H\circ\beta^y+(F_0\circ\beta^y)(P_{\lambda_0}\circ\beta^y)=H\circ\beta^y,
$$
which is a contradiction. Thus, $f_{01,\lambda_0}\in H^0(U_0\cap U_1,\Jj_X)$ provides a $1$-cocycle that it is not a $1$-coboundary and the obstruction to be a $1$-coboundary concentrates at the points of $Y\cap T(X)$.

\noindent{\sc Step 3.} {\em Construction of an analytic function (resp. semialgebraic local Nash function) on $X$ that is not the restriction of a analytic function (resp. Nash function) on $\R^n$.}
Let us use the previous $1$-cocycle to find an element $\xi\in H^0(\R^n,\an_{\R^n}/\Jj_X)$ that is not the image of an element of $H^0(\R^n,\an_{\R^n})$. This provides a local analytic function on $X$ that is not the restriction of a analytic function on $\R^n$. As in the Nash case the involved data is Nash, we obtain a semialgebraic local Nash function on $X$ that is not the restriction of a analytic function on $\R^n$, so it is not the restriction of a Nash function on $\R^n$.

Consider the diagram
{\scriptsize$$
\xymatrix{
0\ar[r]&H^0(U_0,\Jj_X)\times H^0(U_1,\Jj_X)\ar[r]\ar[d]&\an_{\R^n}(U_0)\times\an_{\R^n}(U_1)\ar[r]\ar[d]&H^0(U_0,\an_{\R^n}/\Jj_X)\times H^0(U_1,\an_{\R^n}/\Jj_X)\ar[r]\ar[d]&0\\
0\ar[r]&H^0(U_0\cap U_1,\Jj_X)\ar[r]&\an_{\R^n}(U_0\cap U_1)\ar[r]&H^0(U_0\cap U_1,\an_{\R^n}/\Jj_X)\ar[r]&0\\
&(f_0,f_1)\ar@{|->}[r]\ar@{|->}[d]&(f_0,f_1)\ar@{|->}[d]\\
&f_1-f_0\ar@{|->}[r]&f_1-f_0\\
&&(q_0,q_1)\ar@{|->}[r]\ar@{|->}[d]&([q_0],[q_1])\ar@{|->}[d]&\\
&&q_1-q_0\ar@{|->}[r]&[q_1-q_0]&
}
$$}
For $f_{01,\lambda_0}$ we have to find $(q_0,q_1)\in\an_{\R^n}(U_0)\times\an_{\R^n}(U_1)$ such that 
$$
\frac{h}{P_{\lambda_0}}=f_{01,\lambda_0}=q_1-q_0=\delta((q_0,q_1))\ \text{ on $U_0\cap U_1$}.
$$
Thus, $h=q_1P_{\lambda_0}-q_0P_{\lambda_0}$ and $h_{\lambda_0}:=q_1P_{\lambda_0}=h+q_0P_{\lambda_0}$. The latter function is analytic on $\R^n$, because $P_{\lambda_0}$ is analytic in $\R^n$, $q_1$ is analytic on $U_1=\R^n\setminus Y$ and both $h,q_0$ are analytic on $U_0$ (which contains $Y$). Let us show how one constructs $h_{\lambda_0}\in\an(\R^n)$ such that $q_0:=\frac{h_{\lambda_0}-h}{P_{\lambda_0}}\in\an(U_0)$. We distinguish both cases:

\noindent{\sc Analytic Case.} We consider the coherent sheaf of $\an_{\R^n}$-ideals $\Ff:=P_{\lambda_0}\an_{\R^n}$ on $\R^n$ and the exact sequence of coherent $0\to\Ff\to\an_{\R^n}\to\an_{\R^n}/\Ff\to0$. As the support of the sheaf $\Ff$ is $Y$ and $Y\subset U_0$, the analytic function $h$ defines an element of $H^0(\R^n,\an_{\R^n}/\Ff)$. By Cartan's Theorem B the sequence
$$
0\to H^0(\R^n,\Ff)\to H^0(\R^n,\an_{\R^n})\to H^0(\R^n,\an_{\R^n}/\Ff)\to0
$$
is exact, so there exist an analytic function $h_{\lambda_0}\in\an(\R^n)$ such that $h_{\lambda_0}-h$ is a multiple of $P_{\lambda_0}$ in an open neighborhood of $Y$. 

\noindent{\sc Nash Case.} We consider the finite sheaf of $\Nn_{\R^n}$-ideals $\Ff:=P_{\lambda_0}\Nn_{\R^n}$. As the support of $\Ff$ is $Y$ and $Y\subset U_0$, the Nash function $h$ defines an element of $H^0(\R^n,\Nn_{\R^n}/\Ff)$. By Proposition \ref{factfinitesheaf}(B) there exist a Nash function $h_{\lambda_0}\in\Nn(\R^n)$ such that $h_{\lambda_0}-h$ is a multiple of $P_{\lambda_0}$ in an open (semialgebraic) neighborhood of $Y$. 

Thus, we have constructed $h_{\lambda_0}$ both in the analytic and in the Nash case. Observe in both cases that $q_0:=\frac{h_{\lambda_0}-h}{P_{\lambda_0}}\in\an(U_0)$, because $h_{\lambda_0}-h\in\an(U_0)$, $\ZZ(P_{\lambda_0})\cap\R^n=Y$ and $P_{\lambda_0}$ divides $h_{\lambda_0}-h$ in some neighborhood of $Y$.

As $Y\subset U_0$ and $\ZZ(h)=X\cap U_0$, we deduce $h_{\lambda_0}|_Y=0$. As $\ZZ(P_{\lambda_0})=Y$, we have
$$
\begin{cases}
q_0:=\frac{h_{\lambda_0}-h}{P_{\lambda_0}}\in\an(U_0),\\
q_1:=\frac{h_{\lambda_0}}{P_{\lambda_0}}\in\an(U_1).
\end{cases}
$$
Observe that $\delta((q_0,q_1))=q_1-q_0=\frac{h}{P_{\lambda_0}}=f_{01,\lambda_0}$ (see \eqref{delta}). 

\noindent{\sc Sought function.}
Consider the analytic (resp. semialgebraic local Nash) function on $X$ defined by
$$
q_{\lambda_0}:X\to\R,\ z\mapsto\begin{cases}
q_0(z)&\text{if $z\in X\cap U_0=X\setminus N(X)$},\\
q_1(z)&\text{if $z\in X\cap U_1=X\setminus Y$}.
\end{cases}
$$
The previous analytic (resp. semialgebraic local Nash) function is well defined on $X$, because $h$ is an analytic equation of $X\cap U_0$ inside $U_0$. In fact, $\xi_{\lambda_0}:=\frac{h_{\lambda_0}}{P_{\lambda_0}}$ is a meromorphic (resp. meromorphic Nash function) function on $\R^n$ such that $Y=\ZZ(P_{\lambda_0})\subset\ZZ(h_{\lambda_0})$, $\xi_{\lambda_0}|_X=q$ (observe that $h|_{X\cap U_0}=0$) is an analytic (resp. semialgebraic local Nash) function on $X$ and it has no analytic extension to $\R^n$ (and consequently in the Nash case it has no Nash extension to $\R^n$), because it defines a $1$-cocyle that is not a $1$-coboundary. However, let us conclude alternatively: {\em $\xi_{\lambda_0}=\frac{h_{\lambda_0}}{P_{\lambda_0}}$ has no analytic extension to $\R^n$.} 

Suppose there exists an analytic function $a\in\an(\R^n)$ such that $a|_X=\frac{h_{\lambda_0}}{P_{\lambda_0}}|_X$. Let $\Omega\subset\C^n$ be an open neighborhood of $\R^n$ to which both $h_{\lambda_0}$ and $a$ admit holomorphic extensions that we denote $H_{\lambda_0}$ and $A$. As $(aP_{\lambda_0}-h_{\lambda_0})|_X=0$, we have by Lemma \ref{fx} that $(AP_{\lambda_0}-H_{\lambda_0})|_{\widetilde{X}}=0$. Pick a point $y\in Y\cap T(X)$ and let $\beta^y:(-1,1)\to\widetilde{X}$ be the analytic curve described in {\sc Step 2} such that $\beta^y(0)=y$ and $\beta^y((0,1))\subset(\widetilde{X}\cap\ZZ(P_{\lambda_0}))\setminus\ZZ(H)$. Let $V\subset V_0$ be an invariant open neighborhood of $Y$ and an invariant analytic function $B\in\an(V)$ such that $H|_V-H_{\lambda_0}|_V=BP_{\lambda_0}|_V$. We have
\begin{align*}
&(A\circ\beta^y)(P_{\lambda_0}\circ\beta^y)-(H_{\lambda_0}\circ\beta^y)=0,\\
&(H\circ\beta^y)-(H_{\lambda_0}\circ\beta^y)-(B\circ\beta^y)(P_{\lambda_0}\circ\beta^y)=0.
\end{align*}
As $P_{\lambda_0}\circ\beta^y=0$, we deduce $H_{\lambda_0}\circ\beta^y=0$ and $(H\circ\beta^y)-(H_{\lambda_0}\circ\beta^y)=0$, which is a contradiction, because $H\circ\beta^y\neq0$. 

\noindent{\sc Step 4.} {\em Determining the obstructing set.}
We finally prove: ${\tt O}(\xi_{\lambda_0})=Y\cap T(X)$. As $\xi_{\lambda_0}$ is analytic on $X$, $\Ii_{X,x}=\Jj_{X,x}$ for each $x\in\R^n\setminus\cl(T(X))$ and the set of poles of $\xi_{\lambda_0}$ is contained in $Y$, we deduce that ${\tt O}(\xi_{\lambda_0})\subset Y\cap\cl(T(X))$. If $y\in N(X)\setminus T(X)$, we have $\Jj_{X,y}=\Ii_{X,y}$. Thus, $h_y\in\Ii_{X,y}$ and $h_{\lambda_0,y}-h_y\in P_{\lambda_0,y}\an_{\R^n,y}$, so $h_{\lambda_0,y}\in P_{\lambda_0,y}\an_{\R^n,y}+\Ii_{X,y}$. Consequently, $y\not\in{\tt O}(\xi_{\lambda_0})$. As $\cl(T(X))=T(X)\cup N(X)$ (Theorem \ref{xy}), we conclude ${\tt O}(\xi_{\lambda_0})\subset Y\cap T(X)$.

Pick a point $y\in Y\cap T(X)$ and suppose there exists $a_y'\in\an_{\R^n,y}$ such that $h_{\lambda_0,y}-a_y'P_{{\lambda_0},y}\in\Ii_{X,y}$. We keep the notations introduced above. Let $A_y'$ be the invariant analytic extension of $a_y$ to $\C^n_y$ and let $A'$ be an analytic representative of $A_y'$ defined on an open neighborhood of $y$ in $\C^n$. We may assume (after shrinking the domain of $\beta^y$ and reparameterizing) that $\im(\beta^y)$ is contained in such open neighborhood. We have
\begin{align*}
&(A'\circ\beta^y)(P_{\lambda_0}\circ\beta^y)-(H_{\lambda_0}\circ\beta^y)=0,\\
&(H\circ\beta^y)-(H_{\lambda_0}\circ\beta^y)-(B\circ\beta^y)(P_{\lambda_0}\circ\beta^y)=0.
\end{align*}
As $P_{\lambda_0}\circ\beta^y=0$, we deduce $H_{\lambda_0}\circ\beta^y=0$ and $(H\circ\beta^y)-(H_{\lambda_0}\circ\beta^y)=0$, which is a contradiction, because $H\circ\beta^y\neq0$. Thus, $y\in{\tt O}(\xi_{\lambda_0})$ and we deduce $Y\cap T(X)\subset{\tt O}(\xi_{\lambda_0})$, so ${\tt O}(\xi_{\lambda_0})=Y\cap T(X)$. 

This finishes the proof of both Theorem \ref{main1} and Theorem \ref{main2}.
\end{proof}
\begin{remark}\label{summ}
If $X\subset\R^n$ is a non-coherent analytic set, we have proceeded as follows to find $q\in\Cc^\omega(X)\setminus\an(X)$. We choose a $C$-analytic subset $Y\subset X\subset\R^n$ that meets $T(X)$. We choose an analytic function $h$ on a neighborhood $U_0$ of $Y$ such that $h_y\in\Jj_{X,y}\setminus\Ii_{X,y}$ for each $y\in Y$. We can guarantee the existence of such an analytic function if $Y$ is either a discrete set (resp. finite in the Nash case) or if $Y\cap N(X)=\varnothing$ (see \S\ref{sh}). We choose a finite family of analytic functions (resp. Nash functions) $P_1,\ldots,P_m$ on an open neighborhood (resp. open semialgebraic neighborhood) $\Omega\subset\C^n$ of $\R^n$ such that $\widetilde{Y}:=\ZZ(P_1,\ldots,P_m)$ is a complexification of $Y$. We consider the family of invariant analytic functions (resp. Nash functions) $P_{\lambda}=\sum_{i=1}^m\lambda_iP_i^2$ for $\lambda\in\Qq_m$. We choose suitable $\lambda_0\in\Qq_m$ such that $f_{01,\lambda_0}:=\frac{h}{P_{\lambda_0}}$ defines a $1$-cocycle of $H^1(\R^n,\Jj_X)$ that is not a $1$-coboundary. From this $1$-cocycle we construct $q$ finding $h_{\lambda_0}\in\an(\R^n)$ such that $h_{\lambda_0}-h$ is divisible by $P_{\lambda_0}$ in a neighborhood of $Y$. The chosen $q\in\Cc^\omega(X)\setminus\an(X)$ is the restriction to $X$ of the meromorphic function $\xi_{\lambda_0}:=\frac{h_{\lambda_0}}{P_{\lambda_0}}$. In case $X\subset\R^n$ is a Nash set we may choose $h\in\Ii^\bullet(X\setminus N(X))$ and $h_{\lambda_0}\in\Nn(\R^n)$. Thus, $\xi_{\lambda_0}:=\frac{h_{\lambda_0}}{P_{\lambda_0}}$ defines an element $q\in\Cc^{{\Nn}}(X)\setminus\Nn(X)$.
\hfill$\sqbullet$
\end{remark}

\subsection{An application: `non-global' models for non-coherent objects}
From the previous results let us deduce the existence of `non-global' models for each non-coherent subsets of $\R^n$ both in the analytic and in the Nash settings.

\begin{cor}\label{cons1}
Let $X\subset\R^n$ be a non-coherent $C$-analytic set. There exist many real analytic sets $X^*\subset\R^{n+1}$ that are not a $C$-analytic sets, but $X^*$ is analytically diffeomorphic to $X$.
\end{cor}

\begin{cor}\label{cons2}
Let $X\subset\R^n$ be a non-coherent Nash set. There exist many semialgebraic sets $X^*\subset\R^{n+1}$ that are locally Nash sets, but not Nash subsets of $\R^{n+1}$, such that $X^*$ is local Nash diffeomorphic to $X$.
\end{cor}

\begin{remark}
As coherence of real analytic sets is a local property, it is preserved by (local) analytic diffeomorphims. Thus, a real analytic set that is analytically diffeomorphic to a coherent real analytic set, it is a $C$-analytic set. Analogously, as coherence of Nash sets is a local property, it is preserved by local Nash diffeomorphims. Consequently, a semialgebraic set that is locally a Nash set and local Nash diffeomorphic to a coherent Nash set, it is a Nash set.
\hfill$\sqbullet$ 
\end{remark}

We proceed with the proof of Corollaries \ref{cons1} and \ref{cons2} keeping the notation already introduced along the proof of Theorems \ref{main11} and \ref{main21}. 

\begin{proof}[Proof of Corollaries {\em\ref{cons1}} and {\em\ref{cons2}}]
We continue from the end of {\sc Step 4} and keep all the notations already introduced in the proof of Theorems \ref{main1} and \ref{main2}.

\noindent{\sc Step 5.} {\em Construction of an analytic (resp. semialgebraic local Nash) set $X^*\subset\R^{n+1}$ that is analytically diffeomorphic (Nash diffeomorphic) to $X\subset\R^n$.} 
Define $X^*:={\rm graph}(\xi_{\lambda_0}|_X)$, which is an analytic (resp. semialgebraic local Nash) set analytically diffeomorphic (resp. local Nash diffeomorphic) to $X$ via the mutually inverse analytic maps (resp. local Nash maps):
\begin{align*}
\Gamma_{\lambda_0}:&X\to X^*,\ x\mapsto(x,\xi_{\lambda_0}(x)),\\
\pi:&X^*\to X,\ (x,y)\mapsto x.
\end{align*}

Let us take in mind along the proof that ${\tt O}(\xi_{\lambda_0})=Y\cap T(X)$. By Lemma \ref{ncas} the proof of both Corollaries \ref{cons1} and \ref{cons2} reduces to check: {\em $X^*$ is not a $C$-analytic set}. 

Recall that $Y=\ZZ(P_{\lambda_0})$ and consider the $C$-analytic set (resp. Nash set)
$$
Z:=(X\times\R)\cap\ZZ(P_{\lambda_0}\x_{n+1}-h_{\lambda_0})
$$
of $\R^{n+1}$. If $L:=Y\times\R$, we claim: $Z=X^*\cup L$. 

The inclusion right to left is clear, because $X^*={\rm graph}(\xi_{\lambda_0}|_X)$ and $Y=\ZZ(P_{\lambda_0})\subset\ZZ(h_{\lambda_0})$. Pick a point $(x,t)\in Z$. If $x\not\in Y$, then $P_{\lambda_0}(x)\neq0$ and $(x,t)\in X^*$. If $x\in Y$, then $(x,t)\in L$.

\noindent{\sc Step 6.} {\em The analytic set $X^*$ is not a $C$-analytic subset of $\R^{n+1}$ and the obstruction concentrates at the points of $((Y\cap T(X))\times\R)\cap X^*$.} Assume that $X^*$ is a $C$-analytic set, let $\widetilde{X^*}$ be an invariant complexification of $X^*$ and let $\widetilde{X}$ be an invariant complexification of $X$. Let $\Omega\subset\C^{n+1}$ be an invariant Stein open neighborhood of $\R^{n+1}$ in $\C^{n+1}$ such that $\widetilde{X^*}$ is a complex analytic subset of $\Omega$ (maybe after shrinking $\widetilde{X^*}$). Let $\Theta\subset\C^n$ be the image of $\Omega$ under the projection $\Pi:\C^{n+1}\to\C^n,\ (z,z_{n+1})\mapsto z$, which is an open subset of $\C^n$. Shrinking $\Omega$ and $\Theta$ if necessary, we may assume $\widetilde{X}$ is a complex analytic subset of $\Theta$ (maybe after shrinking $\widetilde{X}$). We may assume $h_{\lambda_0},P_{\lambda_0}$ have invariant holomorphic extensions $H_{\lambda_0},P_{\lambda_0}$ to $\Theta$. Denote by $\Xi_{\lambda_0}:=\frac{H_{\lambda_0}}{P_{\lambda_0}}$ the meromorphic extension of $\xi_{\lambda_0}$ to $\Theta$. By the identity principle
$$
\widetilde{X^*}\setminus\ZZ(P_{\lambda_0}\circ\Pi)=\{(z,\Xi_{\lambda_0}(z))\in\Omega:\ z\in\widetilde{X}\setminus\ZZ(P_{\lambda_0})\}.
$$

As $h-h_{\lambda_0}=P_{\lambda_0}b$ on an open neighborhood of $Y$ in $\R^n$ (where $b$ is an analytic function on such open neighborhood of $Y$), we may assume that there exists an invariant holomorphic extension $B$ of $b$ to $V_0$ such that $H-H_{\lambda_0}=P_{\lambda_0}B$. This means that $H$ and $H_{\lambda_0}$ take the same values on $\ZZ(P_{\lambda_0})\cap V_0$. Recall that $D:=\{y_j\}_{j\geq1}$ is a dense countable subset of $Y\cap T(X)$. Write $\rho_j:=-B(y_j)$ and observe that if $\mu\neq\rho_j,0$ for each $j\geq1$, then
\begin{align}
\ZZ(\mu P_{\lambda_0}-H_{\lambda_0},P_{\lambda_0})_{y_j}&=\ZZ(P_{\lambda_0},H_{\lambda_0})_{y_j}=\ZZ(P_{\lambda_0},H-P_{\lambda_0}B)_{y_j}=\ZZ(P_{\lambda_0},H)_{y_j},\label{ai0}\\
\ZZ(\mu P_{\lambda_0}-H_{\lambda_0},H)_{y_j}&=\ZZ(\mu P_{\lambda_0}-H_{\lambda_0}+H,H)_{y_j}
=\ZZ(\mu P_{\lambda_0}-P_{\lambda_0}B,H)_{y_j}\nonumber\\
&=\ZZ((\mu+B)P_{\lambda_0},H)_{y_j}=\ZZ(P_{\lambda_0},H)_{y_j}.\label{ai1}
\end{align}

In addition, $y_j\in\ZZ(P_{\lambda_0},H_{\lambda_0})$, so $y_j\in\ZZ(\mu P_{\lambda_0}-H_{\lambda_0})$ for each $\mu\in\R$. As $
\widetilde{X}_{y_j}\setminus\ZZ(H_{y_j})\neq\varnothing$, there exists for each $j\geq1$ an irreducible component $T_{j,y_j}$ of $\widetilde{X}_{y_j}$ such that $T_{j,y_j}\not\subset\ZZ(H_{y_j})$, so by \cite[Cor.10.9]{ei} 
\begin{equation}\label{si0}
\dim(T_{j,y_j}\cap\ZZ(H_{y_j}))=\dim(T_{j,y_j})-1. 
\end{equation}
By Lemma \ref{trivial} and Baire's Theorem we may choose ${\lambda_0}\in\Qq_m$ (from the beginning) such that 
\begin{equation}\label{si}
\dim(T_{j,y_j}\cap\ZZ(H,P_{\lambda_0})_{y_j})<\dim(T_{j,y_j}\cap\ZZ(H)_{y_j})
\end{equation}
for each $j\geq1$. We claim: \em $\widetilde{X}_{y_j}\cap\ZZ(\mu P_{\lambda_0}-H_{\lambda_0})_{y_j}\not\subset\ZZ(HP_{\lambda_0})_{y_j}$ for each $j\geq1$ and each $\mu\in\R\setminus(\{\rho_j\}_{j\geq1}\cup\{0\})$\em.

Otherwise, there exist $j_0\geq1$ and $\mu_0\in\R\setminus(\{\rho_j\}_{j\geq1}\cup\{0\})$ such that $\widetilde{X}_{y_{j_0}}\cap\ZZ(\mu_0 P_{\lambda_0}-H_{\lambda_0})_{y_{j_0}}\subset\ZZ(HP_{\lambda_0})_{y_{j_0}}$. By \eqref{ai0} and \eqref{ai1}
\begin{equation*}
\begin{split}
\widetilde{X}_{y_{j_0}}\cap\ZZ(&\mu_0 P_{\lambda_0}-H_{\lambda_0})_{y_{j_0}}\\
&=\widetilde{X}_{y_{j_0}}\cap\ZZ(\mu_0 P_{\lambda_0}-H_{\lambda_0})_{y_{j_0}}\cap\ZZ(HP_{\lambda_0})_{y_{j_0}}=\widetilde{X}_{y_{j_0}}\cap\ZZ(\mu_0 P_{\lambda_0}-H_{\lambda_0},HP_{\lambda_0})_{y_{j_0}}\\
&=(\widetilde{X}_{y_{j_0}}\cap\ZZ(\mu_0 P_{\lambda_0}-H_{\lambda_0},P_{\lambda_0})_{y_{j_0}})\cup(\widetilde{X}_{y_{j_0}}\cap\ZZ(\mu_0 P_{\lambda_0}-H_{\lambda_0},H)_{y_{j_0}})\\
&=\widetilde{X}_{y_{j_0}}\cap\ZZ(P_{\lambda_0},H)_{y_{j_0}}.
\end{split}
\end{equation*}

On the other hand, for each irreducible component $T_{y_{j_0}}$ of $\widetilde{X}_{y_{j_0}}$, we have by \cite[Cor.10.9]{ei}
\begin{equation}\label{ai}
\dim(T_{y_{j_0}}\cap\ZZ(\mu_0P_{\lambda_0}-H_{\lambda_0})_{y_{j_0}})\geq\dim(T_{y_{j_0}})-1.
\end{equation}
In addition, 
\begin{multline*}
T_{y_{j_0}}\cap\ZZ(\mu_0P_{\lambda_0}-H_{\lambda_0})_{y_{j_0}}=T_{y_{j_0}}\cap\widetilde{X}_{y_{j_0}}\cap\ZZ(\mu_0P_{\lambda_0}-H_{\lambda_0})_{y_{j_0}}\\
=T_{y_{j_0}}\cap\widetilde{X}_{y_{j_0}}\cap\ZZ(P_{\lambda_0},H)_{y_{j_0}}=T_{y_{j_0}}\cap\ZZ(P_{\lambda_0},H)_{y_{j_0}}
\end{multline*}
for each irreducible component $T_{y_{j_0}}$ of $\widetilde{X}_{y_{j_0}}$. By \eqref{si0}, \eqref{si} and \eqref{ai} there exists an irreducible component $T_{j_0,y_{j_0}}$ of $\widetilde{X}_{y_{j_0}}$ such that
\begin{multline*}
\dim(T_{j_0,y_{j_0}}\cap\ZZ(H,P_{\lambda_0})_{y_{j_0}})<\dim(T_{j_0,y_{j_0}}\cap\ZZ(H)_{y_{j_0}})\\
=\dim(T_{j_0,y_{j_0}})-1\leq\dim(T_{j_0,y_{j_0}}\cap\ZZ(\mu_0P_{\lambda_0}-H_{\lambda_0})_{y_{j_0}}),
\end{multline*}
which is a contradiction. Thus, $\widetilde{X}_{y_j}\cap\ZZ(\mu P_{\lambda_0}-H_{\lambda_0})_{y_j}\not\subset\ZZ(HP_{\lambda_0})_{y_j}$ for each $j\geq1$ and each $\mu\in\R\setminus(\{\rho_j\}_{j\geq1}\cup\{0\})$, as claimed.

Pick $y\in Y\cap T(X)$ and let us check: $\widetilde{X}_y\cap\ZZ(\mu P_{\lambda_0}-H_{\lambda_0})_y\not\subset\ZZ(HP_{\lambda_0})_y$. 

Otherwise, $\widetilde{X}_y\cap\ZZ(\mu P_{\lambda_0}-H_{\lambda_0})_y\subset\ZZ(HP_{\lambda_0})_y$, so there exists an open neighborhood $W^y\subset\C^n$ such that $\widetilde{X}\cap\ZZ(\mu P_{\lambda_0}-H_{\lambda_0})\cap W^y\subset\ZZ(HP_{\lambda_0})\cap W^y$. As $D$ is dense in $Y\cap T(X)$, there exists $y_j\in W^y$, so $\widetilde{X}_{y_j}\cap\ZZ(\mu P_{\lambda_0}-H_{\lambda_0})_{y_j}\subset\ZZ(HP_{\lambda_0})_{y_j}$, which is a contradiction, because $\widetilde{X}_{y_j}\cap\ZZ(\mu P_{\lambda_0}-H_{\lambda_0})_{y_j}\not\subset\ZZ(HP_{\lambda_0})_{y_j}$ for each $j\geq1$. 

Fix $y\in Y\cap T(X)$. By the curve selection lemma \cite[Rem.VII.4.3(b)]{abr} (using the underlying real structures of all the involved complex analytic sets) there exists an analytic arc 
$$
\alpha^y:(-1,1)\to((\widetilde{X}\cap\ZZ(\mu P_{\lambda_0}-H_{\lambda_0}))\setminus\ZZ(HP_{\lambda_0}))\cup\{y\}
$$ 
such that $(\alpha^y)^{-1}(y)=\{0\}$. Thus, $\Xi_{\lambda_0}\circ\alpha_\mu=\frac{H_{\lambda_0}\circ\alpha_\mu}{P_{\lambda_0}\circ\alpha_\mu}=\mu$ for each $\mu\in\R\setminus\{0,\rho_j:\ j\geq1\}$, so $L\setminus\{(y,0),(y,\rho_j):\ j\geq1\}\subset\widetilde{X^*}\cap\R^n=X^*$. As $\R\setminus\{0,\rho_j:\ j\geq1\}$ is (by Baire's Theorem) dense in $\R$, we conclude $L\subset X^*$, so $Z=X^*\cup L=X^*$, which is a contradiction. We conclude $X^*$ is not a $C$-analytic subset of $\R^{n+1}$ and the obstruction to be a $C$-analytic set concentrates at the points of $((Y\cap T(X))\times\R)\cap X^*$, as required. 
\end{proof}

\section{Smooth semialgebraic functions versus Nash functions}\label{s7}

The intrinsic definition of semialgebraic differentiable function is a classical problem that goes back to Whitney's extension theorem \S\ref{wet}. Inspired by the latter and based on jets of semialgebraic functions, we recall a notion of differentiable semialgebraic function of intrinsic nature proposed in \cite{at}, which is motivated by the proofs of Whitney's extension theorem collected in \cite[\S1]{m} and \cite{kp2,th}. 

\subsection{Differentiable semialgebraic functions}
Recall that we denote $\N:=\{0,1,2,\ldots\}$ the set of natural numbers including $0$.

\begin{defns}\label{def:jets}
Let $S\subset\R^n$ be a semialgebraic set and denote $\x:=(\x_1,\ldots,\x_n)$. A \em semialgebraic jet on $S$ of order $p\geq0$ \em is a collection of semialgebraic functions $F:=(f_{\alpha})_{|\alpha|\leq p}$ on $S$. Here we denote $\alpha:=(\alpha_1,\ldots,\alpha_n)\in\N^n$ and $|\alpha|:=\alpha_1+\cdots+\alpha_n$. For each $a\in S$ write
$$
T_a^pF:=\sum_{|\alpha|\leq p}\frac{f_\alpha(a)}{\alpha!}(\x-a)^\alpha\quad\text{and}\quad R_a^pF:=f_0-T_a^pF.
$$
Denote the set of all semialgebraic jets on $S$ of order $p$ with ${\mathcal J}^p(S)$, which has a natural structure of $\R$-vector space. For every $\beta\in\N^n$ such that $|\beta|\leq p$ we denote 
$$
D^\beta:{\mathcal J}^p(S)\to{\mathcal J}^{p-|\beta|}(S),\ F:=((f_{\alpha})_{|\alpha|\leq p})\mapsto F_\beta:=(f_{\gamma+\beta})_{|\gamma|\leq p-|\beta|},
$$
that is a linear map.
\end{defns}

\begin{defn}[Def.1]\label{defsr}
A continuous semialgebraic function $f:S\to\R$ is a \em $\Cc^p$ semialgebraic function \em if there exists a semialgebraic jet $F:=(f_\alpha)_{|\alpha|\leq p}$ on $S$ of order $p$ such that: 

\paragraph{}\label{diff}\em $f_0=f$ and for each $\beta$ with $|\beta|\leq p$ and every point $a\in S$ it holds $|R_x^{p-|\beta|}F_\beta(y)|=o(\|x-y\|^{p-|\beta|})$ for $x,y\in S$ when $x,y\to a$,\em

\noindent that is, the function $\tau:(0,+\infty)\to\R$ defined by
$
\tau(t):=\sup_{\substack{x,y\in S,\ x\neq y\\ \|x-a\|\leq t,\ \|y-a\|\leq t}}\frac{|R_x^{p-|\beta|}F_\beta(y)|}{\|x-y\|^{p-|\beta|}}
$
is an increasing function that can be continuously extended to $t=0$ by $\tau(0)=0$.

We denote by $\Ss^p(S)$ the collection of all the $\Cc^p$ semialgebraic functions on $S$ (shortly, denoted as $\Ss^p$-functions).
\end{defn}

\begin{remark}\label{Rest}
Condition \ref{diff} is equivalent by \cite[I.Thm.2.2]{m} to the following one:

\paragraph{}\label{diff2}\em $f_0=f$ and for each point $a\in S$ there exists an increasing, continuous and concave function $\alpha_a:[0,+\infty)\to[0,+\infty)$ such that $\alpha_a(0)=0$ and for each $z\in\R^n$ it holds
$$
|T_x^pF(z)-T_y^pF(z)|\leq\alpha_a(\|x-y\|)\cdot(\|z-x\|^p+\|z-y\|^p)
$$ 
for $x,y\in S$ when $x,y\to a$.\em\hfill$\sqbullet$
\end{remark}

There are other alternatives to define $\Ss^p$-functions on a general semialgebraic set $S\subset\R^n$. We recall some of them:
\begin{itemize}
\item[(Def.2)] There exist an open semialgebraic neighborhood $U\subset\R^n$ of $S$ and a $\Ss^p$-function $F:U\to\R$ such that $F|_S=f$. 
\item[(Def.3)] For each point $x\in S$ there exist an open semialgebraic neighborhood $U^x\subset\R^n$ of $x$ and an $\Ss^p$-function $F_x:U^x\to\R$ such that $F_x|_{U^x\cap S}=f|_{U^x\cap S}$.
\item[(Def.4)] For each point $x\in S$ there exist an open neighborhood $U^x\subset\R^n$ of $x$ and an $\Cc^p$ differentiable function $F_x:U^x\to\R$ (non-necessarily semialgebraic) such that $F_x|_{U^x\cap S}=f|_{U^x\cap S}$.
\item[(Def.5)] There exist an open neighborhood $U\subset\R^n$ of $S$ and a $\Cc^p$ differentiable function $F:U\to\R$ such that $F|_S=f$.
\end{itemize}

Observe that (Def.2) $\Longrightarrow$ (Def.3) $\Longrightarrow$ (Def.4) $\iff$ (Def.5) and (Def.2) $\Longrightarrow$ Definition \ref{defsr}. The implication (Def.4) $\Longrightarrow$ (Def.5) follows considering a $\Cc^p$-partition of unity. Analogously, if $S\subset\R^n$ is a compact semialgebraic set, (Def.3) $\Longrightarrow$ (Def.2) using an $\Ss^p$-partition of unity. If $S\subset\R^n$ is a closed semialgebraic set, Definition \ref{defsr} implies (Def.2), see \cite{kp2,th}. If in addition $n=2$, the authors prove in \cite{fl} that (Def.5) implies (Def.2), whereas in \cite{at} the authors show that (Def.5) implies (Def.2) if $p=1$ and $n\geq1$. In \cite[Cor.1.7]{bcm} the authors present a weak version of `(Def.5) implies (Def.2)' via an intermediate function $t:\N\to\N$, which encodes {\em a certain loss of differentiability}. Namely, if there exist an open semialgebraic neighborhood $U\subset\R^n$ of a closed semialgebraic set $S\subset\R^n$ and a $\Cc^{t(p)}$ differentiable function $G:U\to\R$ such that $G|_S=f$, then there exists an $\Ss^p$-function $F:U\to\R$ such that $F|_S=f$. In the general case (that is, $S\subset\R^n$ is a general semialgebraic set), Definition \ref{defsr} implies the existence of a $\Ss^{p-1}(U)$ function $F$ on an open semialgebraic neighborhood $U\subset\R^n$ of $S$, see \cite[Lem.2.6]{bfg}.

\subsection{Smooth semialgebraic functions and Nash functions}\label{ssf}
A classical result in real analytic geometry \cite[Prop.8.1.8]{bcr} states that a function on a Nash manifold $S\subset\R^n$ is Nash (smooth and semialgebraic) if and only if it is an analytic algebraic function on $S$. One can wonder if a similar result holds for general semialgebraic subsets of $\R^n$. We begin by defining
\begin{equation}\label{sinfty}
\Ss^{(\infty)}(S):=\bigcap_{p\geq0}\Ss^p(S)
\end{equation}
the {\em ring of smooth semialgebraic functions on $S$}, where $\Ss^p(S)$ is the ring of $\Ss^p$ functions on $S$ following Definition \ref{defsr}. We analyze first the local case and for the sake of completeness we borrow the following result from \cite{bfg}.

\begin{lem}\label{local}
Let $S\subset\R^n$ be a semialgebraic set and let $f\in\Ss^{(\infty)}(S)$. For each $x\in S$ there exists a Nash function germ $F_x\in\Nn_{\R^n,x}$ such that $F_x|_{S_x}=f_x$.
\end{lem}
\begin{proof}
Let $f\in\Ss^{(\infty)}(S)$ and assume $x=0\in S$. For each $p\geq0$ let $F_p:=(f^p_\alpha)_{|\alpha|\leq p}$ be a semialgebraic jet of order $p$ associated to $f$. We may assume: \em $f^p_\alpha(0)=f_\alpha^k(0)$ for each pair of integers $p\leq k$ and each $\alpha\in \N^n$ with $|\alpha|\leq p$\em, so the definition $f_\alpha(0):=f_{\alpha}^p(0)$ if $|\alpha|=p$ is consistent.

To that end, recall first that $f_0^p=f$ for each $p\geq0$. Next, replace recursively the semialgebraic jet $F_{p+1}$ by $\widetilde{F}_{p+1}:=(\widetilde{f}^{p+1}_\alpha)$ where $\widetilde{f}_\alpha^{p+1}(x):=f_\alpha^{p+1}(x)-f^{p+1}_\alpha(0)+\tilde{f}^{p}_\alpha(0)$ whenever $|\alpha|\leq p$. Using property \ref{diff2} one shows that $\widetilde{F}_p$ is a semialgebraic jet associated to $f$ as a $\Ss^p$-function and $\tilde{f}_\alpha^k(0)=\tilde{f}_\alpha^p(0)$ if $k\geq p$ and $|\alpha|\leq p$. 

Consider the formal series
\begin{equation}\label{h0}
h_0:=\sum_{\alpha\in\N^n}\frac{1}{\alpha!}f_\alpha(0)\x^\alpha\in\R[[\x]]
\end{equation} 
and let $S_0$ be the germ at the origin of $S$. For simplicity we identify $\Nn_{\R^n,0}$ with $\R[[\x]]_{\rm alg}$. Let us prove: \em $f|_{S_0}$ is the germ of a Nash function on $S_0$\em.

Let $\Delta_1,\ldots,\Delta_s$ be a Nash stratification of $S_0$ (apply \cite[\S9.1]{bcr} to a suitable representative of $S_0$). Let $Z_i$ be the Nash closure of the germ $\Delta_i$, which is an irreducible Nash set germ (because $\Delta_i$ is a connected germ at point of the closure of a Nash stratum), and let $X_i$ be the Nash closure of the graph $\Gamma_i:=\Gamma(f|_{\Delta_i})$. It holds $\dim(X_i)=\dim(\Gamma_i)=\dim(\Delta_i)=\dim(Z_i)$ and $X_i$ is an irreducible Nash germ (because it is the Nash closure of the graph of a Nash function germ $f|_{\Delta_i}$ on a connected germ $\Delta_i$ at point of the closure of a Nash manifold).
 
To ease notation write for a while $\Delta:=\Delta_i$, $Z:=Z_i$, $\Gamma:=\Gamma_i$ and $X:=X_i$. Let $g(\x,\y)\in\R[[\x,\y]]_{\rm alg}$ be such that $X=\ZZ(g)$. We claim: $g(\x,h_0(\x))\in \Ii(Z)\R[[\x]]$.

Let ${\mathfrak G}$ be the collection of the germs $\gamma_0$ at the origin of (continuous) semialgebraic curves $\gamma:[0,1)\to\Delta\subset\R^n$ such that $\gamma(0)=0$. We identify each germ $\gamma_0$ with a Puiseux tuple $\R[[\t^*]]^n$. For each $\gamma\in{\mathfrak G}$ define the homomorphism $\gamma^*:\R[[\x]]\to\R[[\t^*]],\ \zeta\mapsto\zeta\circ\gamma$. By the curve selection lemma \cite[Thm.2.5.5]{bcr} we have $\Ii(Z)=\bigcap_{\gamma\in{\mathfrak G}}\ker(\gamma^*)\cap\R[[\x]]_{\rm alg}$. The completion of the local noetherian ring $\R[[\x]]_{\rm alg}$ is $\R[[\x]]$. We have
$$
\Ii(Z)\R[[\x]]=\Big(\bigcap_{\gamma\in{\mathfrak G}}\ker(\gamma^*)\cap\R[[\x]]_{\rm alg}\Big)\R[[\x]]=\bigcap_{\gamma\in{\mathfrak G}}(\ker(\gamma^*)\cap\R[[\x]]_{\rm alg})\R[[\x]]=\bigcap_{\gamma\in{\mathfrak G}}\ker(\gamma^*).
$$ 
Assume $\xi:=g(\x,h_0(\x))\not\in\Ii(Z)\R[[\x]]$. Then there exists $\gamma\in{\mathfrak G}\subset\R[[\t^*]]^n_{\rm alg}$ such that $\xi\notin\ker(\gamma^*)$. After reparameterizing the variable $\t$ we may assume $\gamma\in\R[[\t]]^n_{\rm alg}$. Let $k\geq1$ be the order of the series $\xi(\gamma)$. Write $h_0=a+b$ where $a:=\sum_{|\alpha|\leq k}\frac{1}{\alpha!}f_\alpha(0)\x^\alpha\in\R[\x]$ has degree $\leq k$ and $b$ has order $>k$. There exists a power series $s\in\R[[\x,\z_1,\z_2]]$ such that 
\begin{equation}\label{s}
g(\x,\z_1+\z_2)=g(\x,\z_1)+\z_2s(\x,\z_1,\z_2). 
\end{equation}
Consequently, $g(\x,h_0)=g(\x,a+b)=g(\x,a)+bs(\x,a,b)$ for some $s\in\R[[\x,\z_1,\z_2]]$. Thus, $\xi(\gamma)=g(\gamma,a(\gamma))+bs(\gamma,a(\gamma),b(\gamma))$. As $\omega(\xi(\gamma))=k$ and $\omega(bs(\x,a,b))\geq k+1$, we deduce $\omega(g(\x,a)(\gamma))=k$. Let us analyze next the Puiseux series $f(\gamma)\in\R[[\t^*]]$. We have
\begin{equation}\label{fpola}
|f(y)-a(y)|=\Big|f(y)-\sum_{|\alpha|\leq k}\frac{1}{\alpha!}f_{\alpha}(0)y^\alpha\Big|=o(\|y\|^k)
\end{equation}
for $y\in S$ when $y\to 0$. Thus,
$$
\lim_{t\to 0^+}\Big(\frac{f(\gamma(t))-a(\gamma(t))}{t^k}\Big)=0,
$$
so $f(\gamma)-a(\gamma)\in\R[[\t^*]]$ has order strictly greater than $k$. By \eqref{s} we deduce
$$
\omega(g(\gamma,f(\gamma))-g(\gamma,a(\gamma)))>k,
$$ 
so $\omega(g((\gamma),f(\gamma)))=k$, which is a contradiction, because $g(\gamma(t),f(\gamma(t)))=0$ for each $t\in [0,1)$ (recall that $\gamma\in{\mathfrak G}$ and $X=\ZZ(g)$ is the Zariski closure of $\Gamma(f|_{\Delta})$). Consequently, $g(\x,h_0(\x))\in \Ii(Z)\R[[\x]]$, as claimed.

Let $f_{i1},\ldots,f_{im}\in\R[[\x]]_{\rm alg}$ be a system of generators of $I(Z_i)$ and let $g_i\in\R[[\x,\y]]_{\rm alg}$ be such that $\ZZ(g_i)=X_i$ for $i=1,\ldots,s$. We know that $g_i(\x,h_0(\x))\in I(Z_i)\R[[x]]$. Thus, there exist $a_{i1},\ldots,a_{im}\in\R[[\x]]$ such that $g_i(\x,h_0)=a_{i1}f_{i1}+\cdots+a_{im}f_{im}$ for $i=1,\ldots,s$. By Artin's approximation theorem \cite[Thm.8.3.1]{bcr} there exist Nash functions germs $F,\widetilde{a}_{i1},\ldots,\widetilde{a}_{im}\in\R[[\x]]_{\rm alg}$ such that $g_i(\x,F)=\widetilde{a}_{i1}f_{i1}+\cdots+\widetilde{a}_{im}f_{im}$ for each $i=1,\ldots,s$. Consequently, $\Gamma(F|_{Z_i})\subset X_i$ for each $i=1,\ldots,s$. As $Z_i$ is an irreducible Nash set and $F$ is a Nash function germ, $\Gamma(F|_{Z_i})$ is an irreducible Nash set of dimension $\dim(Z_i)=\dim(X_i)$. As $X_i$ is irreducible, $X_i=\Gamma(F|_{Z_i})$ for $i=1,\ldots,s$. Thus, $F|_{\Delta_i}=X_i\cap(\Delta_i\times\R)=\Gamma(f|_{\Delta_i})$ for $i=1,\ldots,s$, because each $X_i$ is the graph of $F$ over $Z_i$ and simultaneously it is the Nash closure of $\Gamma(f|_{\Delta_i})$ for $i=1,\ldots,s$. We conclude $F|_{\Delta_i}=f|_{\Delta_i}$ for $i=1,\ldots,s$, so $F|_{S_0}=f_0$, as required. 
\end{proof}
\begin{remark}
{\em All the definitions above of $\Ss^p$-function on a semialgebraic point $S$ provide the same definition of $\Ss^{(\infty)}$ function germ at each point $x\in S$.} 

In the proof of Lemma \ref{local} we have used that $f\in\Ss^{(\infty)}(S)$ only to show that there exists a formal power series $h:=\sum_{\alpha\in\N^n}b_\alpha\x^\alpha\in\R[[\x]]$ such that $\Big|f(y)-\sum_{|\alpha|\leq k}b_\alpha y^\alpha\Big|=o(\|y\|^k)$ for each $k\geq1$
and each $y\in S$ when $y\to0$.

It is enough show: {\em Each alternative definition of $\Ss^p$ function provides a formal power series $h$ satisfying the previous condition.}

As (Def.2) $\Longrightarrow$ (Def.3) $\Longrightarrow$ (Def.4) $\iff$ (Def.5), it is enough to prove the existence of $h$ for (Def.4). Then, for each $p\geq0$ there exists an open neighborhood $U^x_p$ and a $\Cc^p$ differentiable function $F_{p,x}$ on $U^x_p$ such that $F_{p,x}|_{S\cap U^x_p}=f|_{S\cap U^x_p}$. We change recursively $F_{p+1,x}$ by 
$$
\widetilde{F}_{p+1,x}:=F_{p,x}-\sum_{|\alpha|\leq p-1}\frac{1}{\alpha!}\frac{\partial^{|\alpha|}\widetilde{F}_{p,x}}{\partial\x^\alpha}(x)(\x-x)^\alpha
$$
and $\widetilde{F}_{0,x}=F_{0,x}$. Then
$$
h:=\sum_{\alpha\in\N^n}\frac{1}{\alpha!}\frac{\widetilde{F}_{|\alpha|,x}}{\partial\x^\alpha}(0)\x^\alpha\in\R[[\x]]
$$
satisfies the required conditions.\hfill$\sqbullet$
\end{remark}

\begin{defn}\label{nashs}
We define the ring $\Nn(S)$ of {\em Nash functions on the semialgebraic set $S\subset\R^n$} as the restrictions to $S$ of the elements of the ring $H^0(S,(\Nn_{\R^n})|_S)$ of global cross-sections on $S$ of the coherent sheaf $\Nn_{\R^n}$. 
\end{defn}
\begin{remark}\label{nashsr}
By \cite[Thm.1.3]{fgr} we have $H^0(S,(\Nn_{\R^n})|_S)=\displaystyle\lim_{\longrightarrow}\Nn(V)$ where $V\subset\R^n$ covers the open semialgebraic neighborhoods of $S$.\hfill$\sqbullet$ 
\end{remark}

It is clear that $\Nn(S)\subset\Ss^{(\infty)}(S)$ and at this point we wonder: 

\begin{prob}
Under which conditions the inclusion of rings $\Nn(S)\subset\Ss^{(\infty)}(S)$ is an equality?
\end{prob}

\subsection{Nash sets}
We begin considering the case when $S=X\subset\R^n$ is a Nash set. Let us prove the equality between the rings $\Ss^{(\infty)}(X)$ with $\Cc^{\Nn}(X)$.

\begin{lem}\label{ens}
Let $X\subset\R^n$ be a Nash set. Then $\Ss^{(\infty)}(X)=\Cc^{\Nn}(X)$.
\end{lem}
\begin{proof}
The inclusion $\Ss^{(\infty)}(X)\subset\Cc^{\Nn}(X)$ follows from Lemma \ref{local}. To prove the converse inclusion pick $f\in\Cc^{\Nn}(X)=H^0(X,\Nn_{\R^n}/\Jj_{X}^\bullet)$ and let us check: {\em $f$ is a semialgebraic function}. 

Observe that $f|_{X\setminus N(X)}\in H^0(X\setminus N(X),\Nn_{\R^n}/\Jj_{X}^\bullet)=\Nn(X\setminus N(X))$ (see Remark \ref{xnx}). By Remark \ref{remnx}(ii) we know that $X\setminus N(X)$ is dense in $X$. As $f:X\to\R$ is a continuous function and $f|_{X\setminus N(X)}$ is semialgebraic, we conclude that $f$ itself is semialgebraic (because the graph of $f$ is the closure of the graph of $f|_{X\setminus N(X)}$, so it is a semialgebraic set). 

As $f\in\Cc^{\Nn}(X)$, we deduce that $f$ is the restriction to $X$ of a $\Cc^p$ differentiable function on $\R^n$ for each $p\geq1$. By \cite[Cor.1.7]{bcm} we conclude that $f$ is the restriction to $X$ of a $\Cc^p$ semialgebraic function on $\R^n$ for each $p\geq1$, so $f\in\Ss^{(\infty)}(X)$, as required. 
\end{proof}

We have already proven in Corollary \ref{iso} that $\Nn(X)=H^0(X,\Nn_{\R^n}/\Ii_{X}^\bullet)$ for each Nash set $X\subset\R^n$. By Theorem \ref{main2} we have: {\em $\Ss^{(\infty)}(X)=\Cc^{\Nn}(X)=\Nn(X)$ if and only if $X$ is coherent (that is, if $\Jj_{X,x}^\bullet=\Ii_{X,x}^\bullet$ for each $x\in X$)}.

\subsection{Semialgebraic sets}\label{ss}
Let $S\subset\R^n$ be a semialgebraic set and let $\Jj_S^\bullet$ be the sheaf of $\Nn_{\R^n}$-ideals given by 
\begin{equation}
\Jj_{S,x}^\bullet:=\{f_x\in\Nn_{\R^n,x}:\ S_x\subset\ZZ(f_x)\},
\end{equation}
the germs of Nash functions on $\R^n_x$ vanishing identically on $S_x$ for each $x\in U$. 

\subsubsection{Ring of smooth semialgebraic functions}
We prove next (taking again advantage of \cite[Cor.1.7]{bcm}) the following characterization of the ring $\Ss^{(\infty)}(S)$ when $S$ is a locally compact semialgebraic set: 

\begin{thm}\label{sis}
Let $S\subset\R^n$ be a locally compact semialgebraic set and let $f:S\to\R$ be a function. The following conditions are equivalent:
\begin{itemize}
\item[(i)] $f\in\Ss^{(\infty)}(S)$.
\item[(ii)] $f\in\Cc^{\Nn}(S):=H^0(S,\Nn_{\R^n}/\Jj_{S}^\bullet)$.
\item[(iii)] $f$ is semialgebraic and for each $x\in S$ there exists a Nash function germ $F_x\in\Nn_{\R^n,x}$ such that $F_x|_{S_x}=f_x$.
\end{itemize}
\end{thm}
\begin{proof}
(i) $\Longrightarrow$ (ii) The inclusion $\Ss^{(\infty)}(S)\subset\Cc^{\Nn}(S)$ follows from Lemma \ref{local}. 

(ii) $\Longrightarrow$ (iii) Let us prove that if $f\in\Cc^{\Nn}(S)$, then $f$ is semialgebraic. The other condition follows from the definition of $\Cc^{\Nn}(S)$. Let $U$ be the open semialgebraic set and $X\subset U$ the Nash set described above. Let $T$ be the (semialgebraic) set of points of $S\setminus\Sing(X)$ of dimension $d:=\dim(S)$ and $M_d$ the interior of $T$ in $X\setminus\Sing(X)$. Thus, $M_d\subset X\setminus N(X)$ is a Nash manifold of dimension $d$ and it is dense in the set $S_{(d)}$ of points of $S$ of dimension $d$. Observe that $\Jj^\bullet_S|_{M_d}=\Jj^\bullet_{M_d}=\Ii^\bullet_{M_d}$. Thus, $f|_{M_d}$ is a Nash function on $M_d$. As the graph of $f|_{S_{(d)}}$ is the closure in $S_{(d)}\times\R$ of $f|_{M_d}$, we deduce that $f|_{S_{(d)}}$ is semialgebraic. Observe that $S':=S\setminus S_{(d)}$ is an open semialgebraic subset of $S$ of strictly smaller dimension. In addition, $\Jj_S|_{S'}=\Jj_{S'}$ and by induction hypothesis (on the dimension), we deduce $f|_{S'}$ is semialgebraic (in the case of dimension $0$ all the maps are semialgebraic). Consequently, $f$ is a semialgebraic map.

(iii) $\Longrightarrow$ (i) As $S$ is locally compact, $C:=\cl(S)\setminus S$ is a closed semialgebraic subset of $\R^n$. As for each $x\in S$ there exists a Nash function germ $F_x\in\Nn_{\R^n,x}$ such that $F_x|_{S_x}=f_x$, there exist an open neighborhood $U_0\subset\R^n$ of $S$ and a $\Cc^\infty$ function $F_0:U_0\to\R$ such that $F_0|_S=f$. Let $\sigma_0:\R^n\setminus C\to[0,1]$ be a $\Cc^\infty$ bump function such that $\sigma_0|_S=1$ and $\sigma_0|_{(\R^n\setminus C)\setminus U_0}=0$ and define $F:=\sigma_0\cdot F_0:\R^n\setminus C\to[0,1]$, which is a $\Cc^\infty$ function on $\R^n\setminus C$ such that $F|_S=f$.

By Mostowski's Lemma \cite[Lem.6]{m2} there exists a continuous semialgebraic function $h:\R^n\to\R$ such that $\ZZ(h)=C$ and $h|_{\R^n\setminus C}$ is a Nash function. Consider the Nash diffeomorphism 
$$
\varphi:\R^n\setminus C\to M:=\{(x,t)\in\R^{n+1}:\ h(x)t-1=0\}\subset\R^{n+1}. 
$$
By \cite[Cor.8.9.5]{bcr} there exist an open semialgebraic neighborhood $U$ of $M$ in $\R^{n+1}$ and a Nash retraction $\rho:U\to M$. Let $\sigma:\R^{n+1}\to[0,1]$ be a $\Cc^\infty$ bump function such that $\sigma|_{M}=1$ and $\sigma|_{\R^{n+1}\setminus U}=0$. Thus, $g:=\sigma\cdot(F\circ\varphi^{-1}\circ\rho):\R^{n+1}\to\R$ is a $\Cc^{\infty}$ differentiable function on $\R^{n+1}$ (if we extend it by $0$ outside $U$). If $T:=\varphi(S)$, we have that $g|_{T}=f\circ\varphi^{-1}|_T$ is semialgebraic. By \cite[Cor.1.7]{bcm} for each $p\geq1$ there exists an $\Ss^p$ function $G_p:\R^{n+1}\to\R$ such that $G_p|_T=f\circ\varphi^{-1}|_T$. Define $F_p:=G_p\circ\varphi:\R^n\setminus C\to\R$, which is a $\Ss^p$ function such that $F_p|_S=f$ for each $p\geq1$. Thus, $f\in\bigcap_{p\geq0}\Ss^p(S)=\Ss^{(\infty)}(S)$, as required.
\end{proof}
\begin{remark}
The previous result (and in particular the implication (iii) $\Longrightarrow$ (i)) implies that if $S$ is a locally compact semialgebraic set, then one can consider any of the possible definitions of $\Ss^p$ function on $S$ for each $p\geq1$ proposed above and we obtain the same ring $\Ss^{(\infty)}(S)$.\hfill$\sqbullet$
\end{remark}

Theorem \ref{sis} is still true (as a consequence of a more technical use of \cite[Cor.1.7]{bcm}) when $S$ is a general semialgebraic set. As the proof is quite technical, we postpone it until Appendix \ref{B}. Again all possible definitions of $\Ss^p$ function on a general semialgebraic set $S$ proposed above for each $p\geq1$, provide the same ring $\Ss^{(\infty)}(S)$ of smooth semialgebraic functions on $S$.

\subsubsection{Ring of Nash functions}\label{sss}
Let $S\subset\R^n$ be a semialgebraic set. By \cite[\S(4.10)]{fg3} there exist an open semialgebraic $U\subset\R^n$ and a Nash set $X\subset U$ such that if $V\subset U$ is an open semialgebraic neighborhood of $S$ and $Y\subset V$ is the smallest Nash subset of $V$ that contains $S$, there exists an open semialgebraic neighborhood $W\subset V$ of $S$ such that $Y\cap W=X\cap W$. 

As $\Nn(X\cap V)=\Nn(V)/\Ii^\bullet(X\cap V)$ for each open semialgebraic neighborhood $V$ of $S$, we deduce (using again \cite[Thm.1.3]{fgr}) that $\Nn(S)=\displaystyle\lim_{\longrightarrow}\Nn(X\cap V)$ where $V\subset\R^n$ covers the open semialgebraic neighborhoods of $S$. By Remark \ref{generalnash}(i) we have $\Ii^\bullet(X\cap V)=\Ii^\bullet(X)\Nn(V)$, so $\Nn(X\cap V)=\Nn(V)/\Ii^\bullet(X)\Nn(V)$ for each open semialgebraic neighborhood $V$ of $S$ and $\Nn(S)=\displaystyle\lim_{\longrightarrow}\Nn(V)/\Ii^\bullet(X)\Nn(V)$ where $V\subset\R^n$ covers the open semialgebraic neighborhoods of $S$. Thus, 
\begin{equation}\label{ssseq}
\Nn(S)=H^0(S,(\Nn_{\R^n}/\Ii_X^\bullet)|_S). 
\end{equation}

\subsubsection{Smooth semialgebraic functions versus Nash functions}
Let $U$ be the open semialgebraic set and $X\subset U$ the Nash set described at the beginning of \S\ref{sss}. We have $\Ii_{X,x}^\bullet\subset\Jj_{S,x}^\bullet$ for each $x\in S$. Define 
\begin{equation}\label{as}
A(S):=\{x\in S:\ \Ii_{X,x}^\bullet\neq\Jj_{S,x}^\bullet\}. 
\end{equation}
If $A(S)=\varnothing$, then $\Ii_{X,x}^\bullet=\Jj_{S,x}^\bullet$ for each $x\in S$ and 
$$
\Ss^{(\infty)}(S)=\Cc^{\Nn}(S)=H^0(S,\Nn_{\R^n}/\Jj_{S}^\bullet)=H^0(S,(\Nn_{\R^n}/\Ii_X^\bullet)|_{S})=\Nn(S).
$$
An enlightening example of the previous situation is the semialgebraic set $S:=\{x^2-zy^2=0,z\geq0\}$ (Whitney's umbrella with the handle erased), which has $A(S)=\varnothing$. It is natural to wonder what happens if $A(S)\neq\varnothing$. The following result solves this situation.

\begin{thm}
If $S\subset\R^n$ be a semialgebraic set, $\Nn(S)=\Cc^{\Nn}(S)$ if and only if 
$A(S)=\varnothing$.
\end{thm}
\begin{proof}
It only remains to prove: {\em If $A(S)\neq\varnothing$, then $\Nn(S)\subsetneq\Cc^{\Nn}(S)$}. We distinguish two cases:

\noindent{\sc Case 1.} We prove first: \em if $\Jj_{X,y}^\bullet\neq\Jj_{S,y}^\bullet$ for some $y\in S$, then $\Nn(S)\subsetneq\Cc^{\Nn}(S)$\em. 

Pick a point $y:=(y_1,\ldots,y_n)\in S$ such that $\Jj_{X,y}^\bullet\neq\Jj_{S,y}^\bullet$. Then $Y_y=\ZZ(\Jj_{S,y}^\bullet)\subsetneq\ZZ(\Jj_{X,y}^\bullet)=X_y$. Let $U^y\subset\R^n$ be an open semialgebraic neighborhood of $y$ and let $Y$ be a Nash subset of $U^y$ that is a representative of $Y_y$. Let $h\in\an(U^y)$ be a Nash equation of $Y$ in $U^y$. By the Nash curve selection lemma \cite[Prop.8.1.13]{bcr} there exists a Nash arc $\alpha:(-1,1)\to X$ such that $\alpha(0)=y$ and $\alpha((-1,1)\setminus\{0\})\subset (X\setminus Y)\cap U^y$. Let $k:=\omega(h\circ\alpha)$ be the order of $h\circ\alpha$ at $0$ and define $P_k:=(\x_1-y_1)^{2k}+\cdots+(\x_n-y_n)^{2k}$. Observe that the order $\omega(P_k\circ\alpha)$ of $P_k\circ\alpha$ at $0$ is greater than or equal to $2k$. Observe that $P_k$ has no multiple factors, because the complex analytic set $\ZZ(\frac{\partial P_k}{\partial\x_1},\cdots,\frac{\partial P_k}{\partial\x_n})=\{(0,\ldots,0)\}$. Consider the finite sheaf of $\Nn_{\R^n}$-ideals $\Ff:=P_k\Nn_{\R^n}$ on $\R^n$. As the support of $\Ff$ is $\{y\}$, the Nash function $h$ defines an element of $H^0(\R^n,\Nn_{\R^n}/\Ff)$. By Theorem \ref{factfinitesheaf} (B) there exists a Nash function $h_k\in\Nn(\R^n)$ that represents $h$ as an element of $H^0(\R^n,\Nn_{\R^n}/\Ff)$. This means that there exists a Nash function $a$ on $U^y$ such that $h_k-h=aP_k$ on $U^y$. Consider the meromorphic Nash function $\xi_k:=\frac{h_k}{P_k}$. 

Observe that $\xi_k$ is Nash on $\R^n\setminus\{y\}$. As $h$ is an Nash equation of $Y$ on $U^y$ and $h_k-h=aP_k$ on $U^y$, we deduce $\xi_k=a$ on $Y$, so $\xi_k$ is Nash on $Y$. This means that $\xi_k$ is local Nash on $S$, that is, an element of $\Cc^{\Nn}(S)$. However, {\em $\xi_k$ is not an element of $\Nn(S)=H^0(S,\Nn_{\R^n}/\Ii_X^\bullet)$}.

The restriction $\xi_k|_X$ is not a Nash function at $y$, because $\omega(P_k\circ\alpha)\geq 2k$, whereas $\omega(h_k\circ\alpha)=\omega(h\circ\alpha-(a\circ\alpha)(P_k\circ\alpha))=\omega(h\circ\alpha)=k$.

\noindent{\sc Case 2.} We assume in the following $\Jj_S^\bullet=\Jj_X^\bullet|_S$. Let $y\in S$ be such that $(\Jj_X^\bullet|_S)_y=\Jj_{S,y}^\bullet\neq\Ii^\bullet_{X,y}$ and pick $h_y\in\Jj_{X,y}^\bullet\setminus\Ii^\bullet_{X,y}$. Let $U$ be an open semialgebraic neighborhood $U$ of $S$ in $\R^n$ such that $X\cap U$ is the Nash closure of $S$ in $U$. By Theorem \ref{main21} taking $Y=\{y\}$ there exist an open semialgebraic neighborhood $U$ of $S$ in $\R^n$ an a meromorphic Nash function $\xi:U\dashrightarrow\R$ on $U$ such that $\xi|_X$ is a local Nash and the obstruction of $\xi$ to have a Nash extension to $U$ concentrates at $y\in S$, Consequently, $\xi|_S\in\Cc^{\Nn}(S)\setminus\Nn(S)$, as required.
\end{proof}

\appendix
\section{Semialgebraicity of complexification\\ and normalization in the Nash setting}\label{A}

We prove next for the sake of completeness \S\ref{scns0} and \S\ref{innsss0}.

\begin{proof}[Semialgebraicity of the complexification {\em(\S\ref{scns0})}]
Let $f_1,\ldots,f_p$ be a system of generators of $\Ii^{\bullet}(X)$ and consider the Nash map $f:=(f_1,\ldots,f_p):\R^n\to\R^p$. Denote $\x:=(\x_1,\ldots,\x_n)$ and $\y:=(\y_1,\ldots,\y_p)$. By Artin-Mazur description of Nash manifolds and Nash maps \cite[Thm.8.4.4]{bcr} there exists a non-singular algebraic set $V\subset\R^{n+q}$ of dimension $n$, a connected component $M$ of $V$, a Nash diffeomorphism $\psi:\R^n\to M$ and a polynomial map $g:=(g_1,\ldots,g_p):V\to\R^p$ such that $f=g\circ\psi$. In addition, if $\pi:\R^{n+q}=\R^n\times\R^q\to\R^n$ is the projection onto the first factor, then $\pi|_{V}=\psi^{-1}$.

Let $h_1,\ldots,h_k$ be a system of generators of the zero ideal of $V$ in $\R[\x,\y]$ and let $\widetilde{V}$ be the zero set in $\C^{n+q}$ of the invariant polynomial extensions $H_1,\ldots,H_k$ of $h_1,\ldots,h_k$ to $\C^{n+q}$. Observe that $\widetilde{V}$ is the Zariski closure of $V$ in $\C^{n+q}$. We consider the real underlying algebraic structures of $\C^n\equiv\R^{2n}$ and $\C^{n+q}\equiv\R^{2(n+q)}$ when referring to its semialgebraic subsets. The Nash diffeomorphism $\pi|_V:V\to\R^n$ extends to an invariant complex polynomial map $\Pi:\widetilde{V}\to\C^n$, which is a local diffeomorphism at the points of $V$. By \cite[Lem.9.2]{bfr} there exist an invariant open semialgebraic neighborhood $W_1\subset\C^{n+q}$ of $V$ and an invariant open semialgebraic neighborhood $W_2\subset\C^n$ of $\R^n$ such that $\Pi|_{\widetilde{V}\cap W_1}:\widetilde{V}\cap W_1\to W_2$ is a Nash diffeomorphism. Denote $\Psi:=(\Pi|_{\widetilde{V}\cap W_1})^{-1}:W_2\to\widetilde{V}\cap W_1$, which is a Nash extension of $\psi:\R^n\to M$ to $W_2$.

Consider the invariant polynomial functions $G_1,\ldots,G_p$ that extend $g_1,\ldots,g_p$ to $\C^{n+q}$. Consider the Nash functions on $W_2$ given by $F_i:=G_i\circ\Psi$ and observe that $F_i|_{\R^n}=f_i$ for $i=1,\ldots,p$. Denote $\widetilde{X}:=\ZZ(F_1,\ldots,F_p)\subset W_2$, which satisfies $\widetilde{X}\cap\R^n=X$. We have 
$$
\Ii_{X}=\Ii(X)\an_{\R^n}=\Ii^{\bullet}(X)\an_{\R^n}=(f_1,\ldots,f_p)\an_{\R^n}.
$$
In addition, $\Ii_{X}\otimes_\R\C=(F_1,\ldots,F_p)\an_{\C^n}$, so $\widetilde{X}$ is the support of $\Ii_{X}\otimes_\R\C$ around $X$. We conclude that $\widetilde{X}$ is an invariant complexification of $X$ and it is a semialgebraic subset of $\C^n$.
\end{proof}

We will use Lemma \ref{genpol} to guarantee the semialgebraicity of the normalization of a suitable semialgebraic complexification of a Nash set already constructed above (\S\ref{innsss0}). Let $V\subset\R^n$ be an algebraic set and let $\Ii^*(V)$ be the ideal of all polynomials of $\R[\x]:=\R[\x_1,\ldots,\x_n]$ that vanish identically on $V$. For each $x\in V$ denote $\gtm_x$ the maximal ideal of $\R[\x]$ constituted by all polynomials that vanish at $x$. We say that $V$ is {\em non-singular at $x\in V$} if the local ring $(\R[\x]/\Ii^*(V))_{\gtm_x}$ is regular (compare with \S\ref{rspas} and \S\ref{rspns}) and $V$ is {\em non-singular} if $V$ is non-singular at all its points. In the vein of \cite[Prop.2.8(i)]{fgr} we prove the following result concerning Nash zero ideals generated by polynomials in the non-singular algebraic case.

\begin{lem}\label{genpol}
Let $V\subset\R^n$ be a non-singular algebraic set and let $M$ be an open semialgebraic subset of $V$. Consider the open semialgebraic set $\Omega:=\R^n\setminus(V\setminus M)$ and let $W\subset\Omega$ be an open semialgebraic neighborhood of $M$. Then $\Ii^\bullet(M)=\Ii^*(V)\Nn(W)$ in $\Nn(W)$ and $\Ii(M)=\Ii^*(V)\an(W)$ in $\an(W)$. In addition, $\Ii^\bullet_M=\Ii^*(V)\Nn_W$ and $\Ii^\bullet_M=\Ii^*(V)\an_W$.
\end{lem}
\begin{proof}
Observe that $M=\Omega\cap\cl(M)$ is a Nash submanifold of $\Omega$ and a closed subset of $\Omega$, so it is a closed subset of $W$. We claim: {\em There exists a finite open semialgebraic covering $\{W_k\}_k$ of $W$ such that $\Ii^\bullet_{M}|_{W_k}=\Ii^*(V)\Nn_W|_{W_k}$ for each $k$}. Thus, $\Ii^*(V)$ generates $\Ii^\bullet_{M}$ as a finite sheaf of $\Nn_W$-ideals.

Let $P_1,\ldots,P_m$ be a system of generators of $\Ii^*(V)$ in $\R[\x]$. Observe that $P_1^2+\cdots+P_m^2\in\Ii^*(V)$ generates $\Ii^\bullet_{M,x}$ at the points $x\in W\setminus M$, so $\Ii^\bullet_{M,x}=\Ii^*(V)\Nn_{W}$ for each $x\in W\setminus M$. Thus, we only have to cover each connected component of $M$ by finitely many open semialgebraic subsets of $W$. We assume, to lighten notations, that $M$ is connected to prove the claim. 

Denote $d:=\dim(M)$ and for each set $I\subset\{1\leq i\leq m\}$ of cardinality $n-d$ define $W_I:=\{x\in W:\ {\rm rk}(\frac{\partial P_i}{\partial\x_j}(x))_{i\in I, 1\leq j\leq n}=n-d\}$, which is an open semialgebraic subset of $W$, and $M_I=W_I\cap M$, which is an open semialgebraic subset of $M$ (and consequently a Nash manifold of dimension $d$ if it is non-empty). Observe that $M_I$ is a closed subset of $W_I$. For each $I$ such that $M_I\neq\varnothing$ there exists by \cite[Thm.1.4]{bfr} a finite open semialgebraic covering $\{W_{Ik}\}_{k=1}^{p_{I}}$ of $W_I$ equipped with Nash diffeomorphisms $\phi_{Ik}:W_{Ik}\to\R^n$ that maps $M_I\cap W_{Ik}$ onto the linear variety $\R^d\times\{0\}$ for each $k=1,\ldots,p_I$. By the real Nullstellensatz for Nash functions \cite[Thm.8.6.5]{bcr} the ideal $\Ii^\bullet(\R^d\times\{0\})=\{x_{d+1},\ldots,x_n\}\Nn(\R^n)$, because the ideal in the righthand side is a real prime ideal whose zero set is $\R^d\times\{0\}$. Fix $I,k$ and to lighten rotations write $N:=M_{Ik}$, $\psi:=(\psi_1,\ldots,\psi_n)=\phi_{Ik}$ and $U:=W_{Ik}$. The ideal $\Ii^\bullet(N)$ of $\Nn(U)$ is generated by $\{\psi_{d+1},\ldots,\psi_n\}$. As $N$ is a $d$-dimensional Nash manifold, the matrix $(\frac{\partial\psi_\ell}{\partial\x_j})_{d+1\leq\ell\leq n,1\leq j\leq n}$ has rank $n-d$ at each point of $N$. As $P_i|_N=0$ for each $i\in I$, there exist $b_{i\ell}\in\Nn(U)$ such that $P_i=\sum_{\ell=d+1}^nb_{i\ell}\psi_\ell$. Write $P:=(P_i)_{i\in I}$, $B:=(b_{i\ell})_{i\in I,d+1\leq\ell\leq n}$, $\psi':=(\psi_{d+1},\ldots,\psi_n)$ and observe that $P^t=B\psi'^t$. We compute the Jacobian matrix of $P$ using the latter formula and we obtain 
$$
\Big(\frac{\partial P_i}{\partial\x_j}(x)\Big)_{i\in I, 1\leq j\leq n}=B(x)\Big(\frac{\partial\psi_\ell}{\partial\x_j}(x)\Big)_{d+1\leq\ell\leq n,1\leq j\leq n}
$$
for each $x\in N$. As both matrices $(\frac{\partial P_i}{\partial\x_j}(x))_{i\in I, 1\leq j\leq n}$ and $(\frac{\partial\psi_\ell}{\partial\x_j}(x))_{d+1\leq\ell\leq n,1\leq j\leq n}$ have rank $n-d$ for each $x\in N$, we deduce $\det(B(x))\neq0$ for each $x\in N$. Define $U':=\{x\in U:\ \det(B(x))\neq0\}$, which is an open semialgebraic neighborhood of $N$ in $U$. Using Cramer solution to the system of linear equations $P_i=\sum_{\ell=d+1}^nb_{i\ell}\psi_\ell$ where $i\in I$ we find Nash functions $c_{\ell i}\in\Nn(U')$ such that $\psi_\ell=\sum_{i\in I}c_{\ell i}P_i$ for $\ell=d+1,\ldots,n$. Thus, the ideal $\Ii^\bullet(N)$ of $\Nn(U')$ is generated by $\Ii^*(V)$ and we substitute $U$ by $U'$. As $M$ is a coherent Nash subset of $W$, we have $\Ii^\bullet(N)\Nn(U')$ generates by Remark \ref{generalnash} $\Ii^\bullet_M$ on $U'$ as a finite sheaf of $\Nn_W$-ideals, so $\Ii^*(V)$ generates $\Ii^\bullet_M$ on $U'$ as a finite sheaf of $\Nn_W$-ideals.

Let us check $\Ii^\bullet(M)=\Ii^*(V)\Nn(W)$. Pick $f\in\Ii^\bullet(M)$ and consider the transporter sheaf of $\Nn_W$-ideals $(\Ii^*(Z)\Nn_W:f\Nn_W)$, which is the sheaf of $\Nn_W$-ideals associated to the presheaf of $\Nn_W$-ideals $(\Ii^*(Z)\Nn_W(A):f\Nn_W(A))=\{g\in\Nn(A):\ gf\in\Ii^*(Z)\Nn_W(A)\}$ for each open set $A\subset W$. As $\Ii^\bullet_{M}=\Ii^*(V)\Nn_W$ as finite sheafs of $\Nn_W$-ideals, we deduce that $(\Ii^*(Z)\Nn_W:f\Nn_W)=\Nn_W$ as finite sheaves of $\Nn_W$-ideals. By Theorem \ref{factfinitesheaf}(A) $(\Ii^*(Z)\Nn_W:f\Nn_W)=\Nn_W$ is generated by its global sections $H^0(W,(\Ii^*(Z)\Nn_W:f\Nn_W))=H^0(W,\Nn_W)$, so $f\in\Ii^*(Z)$. Thus, $\Ii^\bullet(M)=\Ii^*(V)\Nn(W)$. The remaining equalities in the statement follow straightforwardly from the previous one. 
\end{proof}

\begin{proof}[Semialgebraicity of the normalization {\em(\S\ref{innsss0})}]
We keep the notations introduced above in the proof of \S\ref{scns0}. Define $Z$ as the zero set in $V$ of $g_1,\ldots,g_p$, that is, $Z=\ZZ(g_1,\ldots,g_p,h_1,\ldots,h_k)$, and $X':=\psi(X)=M\cap Z$. As $\Sing(V)=\varnothing$, the complex algebraic set $\Sing(\widetilde{V})$ does not meet $\R^{n+q}$ and it is invariant under conjugation. Define $U:=\R^{n+q}\setminus (V\setminus M)$, which is an open semialgebraic subset of $\R^{n+q}$, because $M$ is open in $V$ and $V$ is closed in $\R^{n+q}$. Observe that $X'\subset M\subset U$ and both are closed in $U$. We claim: {\em $g_1,\ldots,g_p,h_1,\ldots,h_k$ is a system of generators of $\Ii^\bullet(X')$ in $\Nn(U)$}. 

Let $b\in\Ii^\bullet(X')$ and consider the restriction $b|_M$. As $f_1,\ldots,f_p$ is a system of generators of $\Ii^{\bullet}(X)$ in $\Nn(\R^n)$, we deduce that $g_1,\ldots,g_p$ is a system of generators of $\Ii^{\bullet}(X')$ in $\Nn(M)$, so there exists Nash functions $a_1,\ldots,a_p\in\Nn(M)$ such that $b|_M=g_1a_1+\cdots+g_pa_p$. As $M$ is a Nash manifold closed in $U$, there exists Nash functions $A_1,\ldots,A_p\in\Nn(U)$ such that $A_i|_U=a_i$ for $i=1,\ldots,r$. Thus, $b-A_1g_1-\cdots-A_pg_p=0$ on $M$. As $V$ is non-singular, there exist by Lemma \ref{genpol} Nash functions $B_1,\ldots,B_k\in\Nn(U)$ such that $b=A_1g_1+\cdots+A_pg_p+B_1h_1+\cdots+B_kh_k$ on $U$, so $\Ii^\bullet(X')=\{g_1,\ldots,g_p,h_1,\ldots,h_k\}\Nn(U)$, as claimed.

Consequently, $\Ii^\bullet(X')=\Ii^*(Z)\Nn(\R^{n+q})$, where $\Ii^*(Z)$ is the zero ideal of $Z$ in $\R[\x,\y]$. The sheaf of $\an_{\C^n}$-ideals we use to construct the analytic complexification of $X'$ is $\Ii_{X'}\otimes_R\C$, where $\Ii_{X'}|_{U}=\Ii(X')\an_{\R^{n+q}}|_{U}=\Ii^{\bullet}(X')\an_{\R^{n+q}}|_{U}=\Ii^*(Z)\an_{\R^{n+q}}|_{U}$. 

Define $\widetilde{Z}:=\ZZ(G_1,\ldots,G_p,H_1,\ldots,H_k)$, which is the Zariski closure of $Z$ in $\C^{n+q}$. As $\widetilde{Z}$ is the zero set of $\Ii^*(Z)\C[\x,\y]$, its zero ideal in $\C[\x,\y]$ is by Hilbert's Nullstellensatz and \cite[Ch.V,\! \S15,\! Prop.5]{b} $\Ii^*(\widetilde{Z})=\Ii^*(Z)\C[\x,\y]$. Thus, $(\Ii_{X'}\otimes_\R\C)|_{U}=\Ii^*(Z)\an_{\C^{n+q}}|_{U}=\Ii^*(\widetilde{Z})\an_{\C^{n+q}}|_{U}$.

The algebraic normalization $(Y_1,\eta)$ of $\widetilde{Z}$ is a pair constituted by an invariant complex algebraic subset $Y_1$ of $\C^{n+q+m}$ for some $m\geq1$ and the restriction $\eta$ to $Y_1$ of the projection $H:\C^{n+q+m}=\C^{n+q}\times\C^m\to\C^{n+q}$ onto the first factor \cite[Prop.3.11]{fg3}, which satisfies $\sigma_{n+q}\circ\eta=\eta\circ\sigma_{n+q+m}$. By Zariski Main Theorem \cite[\S V.6]{mo} $(Y_1,\eta)$ is the analytic normalization of $\widetilde{Z}$ (see also \cite[Thm.1.1 \& \S3.4]{abf4}).

We consider the real underlying algebraic structures of $\C^{n+q+m}\equiv\R^{2(n+q+m)}$, $\C^{n+q}\equiv\R^{2(n+q)}$, $Y_1$, $\widetilde{Z}$ and $\eta$. As $\eta^{-1}(X')$ and $\eta^{-1}(V\setminus M)$ are disjoint invariant closed semialgebraic subsets of $\C^{n+q+m}$, we deduce that $\Theta:=\C^{n+q+m}\setminus\eta^{-1}(V\setminus M)$ is an invariant open semialgebraic neighborhood of $\eta^{-1}(X')$ in $\C^{n+q+m}$. In addition, $\Omega:=\C^{n+q}\setminus(V\setminus M)$ is an invariant open semialgebraic neighborhood of $X'$ in $\C^{n+q}$. Let $\widetilde{Z}':=\widetilde{Z}\cap\Omega$, which is an open semialgebraic neighborhood of $X'$ in $\widetilde{Z}$ and satisfies $\widetilde{Z}'\cap\R^{n+q}=X'$. 

Let $\widetilde{X}'\subset\widetilde{Z}'$ be an invariant analytic complexification of $X'$. For each $x\in X'$ we have 
$$
(\Ii_{X'}\otimes_\R\C)_x=\Ii^*(\widetilde{Z})\an_{\C^{n+q},x}\subset\Ii(\widetilde{Z}')\an_{\C^{n+q},x}\subset\Ii(\widetilde{X}')\an_{\C^{n+q},x}=(\Ii_{X'}\otimes_\R\C)_x.
$$
Thus, $(\Ii_{X'}\otimes_\R\C)_x=\Ii(\widetilde{Z}')\an_{\C^{n+q},x}$ for each $x\in X'$. Consequently, $\widetilde{Z}'$ is a semialgebraic invariant analytic complexification of $X'$ and $(Y_1':=\eta^{-1}(\widetilde{Z}'),\eta':=\eta|_{Y_1'})$ is a semialgebraic invariant (analytic) normalization of $\widetilde{Z}'$. Observe that $Y_1'=Y_1\cap\Theta$.

Consider the composition $P:=\Pi\circ H:\C^{n+q+m}=\C^n\times\C^{q+m}\to\C^n$, which is the projection onto the first factor. Let $\widetilde{X}:=\ZZ(F_1,\ldots,F_p)\subset W_2$ and recall that $\Psi|_{\widetilde{X}}:\widetilde{X}\to\widetilde{Z}'\cap W_1$ is a Nash diffeomorphism. Define $Y:=\eta^{-1}(\widetilde{Z}'\cap W_1)$ and $\rho:=P|_{Y}$. It holds that $(Y,\rho)$ is an invariant normalization of $\widetilde{X}$. In addition, $\widetilde{X}$ is a semialgebraic subset of $\C^n$ and $Y$ is a semialgebraic subset of $\C^{n+q+m}$.
\end{proof} 

\section{Proof of Theorem \ref{sis} for general semialgebraic sets}\label{B}

We prove next Theorem \ref{sis} for general semialgebraic sets as a byproduct of the results \cite{bcm}. We explain in certain points how one should interpret and adapt the results in \cite{bcm} in order to prove implication (iii) $\Longrightarrow$ (i) of Theorem \ref{sis}. The proofs of the implications (i) $\Longrightarrow$ (ii) and (ii) $\Longrightarrow$ (iii) of Theorem \ref{sis} proposed above do not use that $S$ is locally compact. We denote $\Bb(x,\veps)$ the open ball of center $x\in\R^n$ and radius $\veps>0$. 

\begin{proof}[Proof of implication {\em (iii) $\Longrightarrow$ (i)} of Theorem {\em\ref{sis}} for general semialgebraic sets]\ 
The proof is conducted in several steps:

\noindent{\sc Step 1.} {\em For each $p\geq1$ there exists a closed semialgebraic set $C_p\subset\cl(S)\setminus S$ such that $f$ admits a $\Cc^p$ differentiable extension $H_p$ to $\R^n\setminus C_p$, the restriction $H_p|_{\cl(S)\setminus C_p}$ is semialgebraic and $C_p\subset C_{p+1}$.}

As $f$ is local Nash at each point $x\in S$, we deduce that $f$ is locally Lipschitz at each point $x\in S$. Let $S_0\subset\cl(S)$ be the set of points $x\in\cl(S)$ such that $f$ is locally Lipschitz in the intersection with $S$ of a neighborhood of $x$, that is, there exist $\veps>0$ and $K_x>0$ (depending on $x$) satisfying $|f(y_1)-f(y_2)|<K_x\|y_1-y_2\|$ for each $y_1,y_2\in M\cap\Bb(x,\veps)$. Thus, $S_0$ is an open semialgebraic subset of $\cl(S)$ that contains $S$. For each point $x\in S_0$ there exists an open semialgebraic neighborhood $V^x\subset S_0\subset\cl(S)$ of $x$ such that $f|_{V^x\cap S}$ is Lipschitz, so in particular $f|_{V^x\cap S}$ is uniformly continuous. This means that $f|_{V^x\cap S}$ admits a unique continuous extension to $\cl(V^x\cap S)$ for each $x\in S_0$. Consequently, there exists a (unique) continuous extension $F$ of $f$ to $S_0$. The function $F$ is semialgebraic, because its graph $\Gamma(F)=\cl(\Gamma(f))\cap(S_0\times\R)$, which is a semialgebraic set. 

As $S_0$ is locally compact, we may assume from the beginning that $S$ is closed in $\R^n$ (using Mostowski's Lemma \cite[Lem.6]{m2} as we have done before) and $f$ extends continuously to $\cl(S)$ (as a semialgebraic function). We denote $F:\cl(S)\to\R$ such extension. Fix a point $x\in S$. As $f$ is local Nash at $x$, there exists an open semialgebraic neighborhood $W^x\subset\R^n$ of $x$ and a Nash function $F_x$ on $W^x$ such that $F_x|_{S\cap W^x}=f|_{S\cap W^x}$. As $f$ extend to $\cl(M)$, we deduce $F_x|_{\cl(S)\cap W^x}=F|_{\cl(S)\cap W^x}$.

For each $p\geq1$ define $S_p\subset\cl(S)$ as the set of points $x\in\cl(S)$ such that there exist $\veps>0$ and $K_x>0$ depending on $x$ and (finitely many) continuous semialgebraic functions $F_\alpha:\cl(S)\cap\Bb(x,\veps)\to\R$ for each $\alpha\in\N^n$ with $|\alpha|\leq p$ such that
$$
\Big|F(y_1)-\sum_{|\alpha|\leq p}\frac{F_{\alpha}(y_2)}{\alpha!}(y_1-y_2)^\alpha\Big|<K_x\|y_1-y_2\|^{p+1}
$$
for each $y_1,y_2\in\cl(S)\cap\Bb(x,\veps)$. As $f$ is local Nash at each point of $S$, we deduce that $S\subset S_p$. We claim: {\em $S_p$ is an open semialgebraic neighborhood of $S$ in $\cl(S)$}. 

To prove the semialgebraicity of $S_p$ we only have to show that the existence of a continuous semialgebraic function $F_\alpha$ on $\cl(S)\cap\Bb(x,\veps)$ is a semialgebraic condition: 
\begin{itemize}
\item {\em Semialgebraicity.} There exists polynomials $P_i,Q_{ij}\in\R[\x]$ such that for each $y\in\cl(S)\cap\Bb(x,\veps)$ there exists $z\in\R$ satisfying $(y,z)\in\bigcup_{i=1}^\ell\{P_i=0,Q_{i1}>0,\ldots,Q_{iq}>0\}$ and if $z'\in\R$ satisfies $(y,z')\in\bigcup_{i=1}^\ell\{P_i=0,Q_{i1}>0,\ldots,Q_{iq}>0\}$, then $z=z'$. We denote $F_\alpha(y):=z$.
\item {\em Continuity.} For each $y\in\cl(S)\cap\Bb(x,\veps)$ and each $\lambda>0$ there exists $\delta>0$ such that $y'\in\cl(S)\cap\Bb(x,\veps)$ and $\|y'-y\|<\delta$ implies $|F_\alpha(y)-F_\alpha(y')|<\lambda$.
\end{itemize}
 
By \cite{kp2,th} the function $F|_{\cl(M)\cap\Bb(x,\veps)}$ extends to $\Bb(x,\veps)$ as a $\Cc^p$ semialgebraic function for each $x\in S_p$. Consequently, using a $\Cc^\infty$ partition of unity, we deduce that $f$ extends as a $\Cc^p$ differentiable function $h_p$ to an open neighborhood $V_p\subset\R^n\setminus C_p$ of $S_p$. Denote $C_p:=\cl(S_p)\setminus S_p=\cl(S)\setminus S_p$, which is a closed semialgebraic subset of $\R^n$, because $S_p$ is locally compact. Let $\sigma_p:\R^n\setminus C_p\to[0,1]$ be a $\Cc^\infty$ bump function such that $\sigma_p|_{S_p}=1$ and $\sigma_p|_{(\R^n\setminus C_p)\setminus V_p}=0$. Define $H_p:=\sigma_ph_p$ and extend it by zero to $(\R^n\setminus C_p)\setminus V_p$. Then $H_p$ is a $\Cc^p$ differentiable extension of $f$ to $\R^n\setminus C_p$ and the restriction $H_p|_{\cl(S)\setminus C_p}=H_p|_{S^p}=F|_{S_p}$ is a semialgebraic function.

By definition $S_{p+1}\subset S_p$, so $C_p=\cl(S)\setminus S_p\subset\cl(S)\setminus S_{p+1}= C_{p+1}$ for each $p\geq1$.

\noindent{\sc Step 2.} {\em Construction of the function $t$ that control the loss of differentiability.} By \cite[Thm.5.1 \& Rem.5.2]{bm2} (adapted to the semialgebraic case) there exists a Nash manifold $N$ of the same dimension as $\cl(S)$ and a proper Nash map $\psi:N\to\cl(S)$ such that $\psi(N)=\cl(S)$. By \cite[Cor.1.7]{bcm} and its proof applied to the semialgebraic function $F:\cl(S)\to\R$ (and $\psi:N\to\cl(S)$) there exists a function $t:\N\to\N$ such that if $H:\R^n\to\R$ is a $\Cc^{t(p)}$ differentiable extension of $F$ to $\R^n$, there exists a $\Ss^p$ function $G:\R^n\to\R$ such that $G|_{\cl(S)}=F$. The function $t$ depends by \cite[Rem.5.5 \& \S5, p.19]{bcm} on:
\begin{itemize}
\item[(1)] The Chevalley function $R:\cl(S)\times\N\to\N$ that itself depends on $\cl(S)$, the Nash function $\varphi$ and $F:\cl(S)\to\R$ (see \cite[Rem.5.5]{bcm}). For each $p\geq1$, the function $C(\cdot,p)$ is uniformly bounded on $N$ for each $p\geq1$, because we are in the semialgebraic setting \cite[\S6, p.22]{bcm}. Define $r(p):=\max\{R(b,p):\ b\in\cl(M)\}$ for each $p\geq1$.
\item[(2)] The exponent $\rho$ of regularity of $\cl(S)$, which exists by \cite[Lem.3.5]{bcm}.
\item[(3)] The maximum $\sigma$ of the \L ojasiewicz exponents of formulas that involves certain map $M^s_\varphi$, which depends on $\varphi$ and $R$, and the distances to the boundaries of certain strata $\Gamma$ of a suitable (finite) stratification of $N$. The quoted formulas are of the type $|\det(M^s_\varphi)(\cdot)|\geq C_\Gamma{\rm dist}(\cdot,\partial\Gamma)^\sigma$ for some constants $C_\Gamma>0$ (see \cite[Eq.(10), p.18]{bcm}). 
\end{itemize}
A possible choice for $t$ is $t(p)=c(\rho p)+\sigma$, see \cite[p.19]{bcm}. 

\noindent{\sc Step 3.} {\em Restriction of $t$ to open semialgebraic neighborhoods of $S$ in $\cl(S)$.}
If $S'\subset\cl(S)$ is an open semialgebraic neighborhood of $S$, it holds that the corresponding function $t'\leq t$, because: 
\begin{itemize}
\item The Chevalley function $R_{S'}$ of $S'\times\N$ is the restriction of $R$ to $S'\times\N$.
\item The exponent $\rho_{S'}$ of regularity of $S'$ is smaller than or equal to $\rho$.
\item The \L ojasiewicz exponents in this case are $\leq\sigma$, because they appear in formulas that involve $M^{s}_{\varphi}|_{\varphi^{-1}(S')}$, which depends on $\varphi|_{\varphi^{-1}(S')}:\varphi^{-1}(S')\to S'$ and $R_{S'}=R|_{S'\times\N}$, and the distances to the boundaries of certain strata $\Gamma'$ of a suitable stratification of $\varphi^{-1}(S')$. Each stratum $\Gamma'$ can be chosen contained in some stratum $\Gamma$ of the stratification of $N$ chosen in (3) above. The quoted formulas are of the type $|\det(M^{s}_{\varphi|_{\varphi^{-1}(S')}})(\cdot)|\geq C_{\Gamma'}{\rm dist}(\cdot,\partial\Gamma')^\sigma$ for some constants $C_{\Gamma'}>0$ (see \cite[Eq.(10), p.18]{bcm}).
\end{itemize}

Consequently, we can take $t'$ equal to $t$ in this case. 

\noindent{\sc Step 4.} {\em Smoothness of $f$.}
Let $C':=\cl(S)\setminus S'$ and let $h':\R^n\to\R$ be a continuous semialgebraic function such that $\ZZ(h')=C'$ and the restriction $h'|_{\R^n\setminus C'}$ is a Nash function (Mostowski's Lemma \cite[Lem.6]{m2}). Consider the Nash diffeomorphism 
$$
\varphi':\R^n\setminus C'\to M':=\{(x,t)\in\R^{n+1}:\ h'(x)t-1=0\}\subset\R^{n+1}. 
$$
By \cite[Cor.8.9.5]{bcr} there exist an open semialgebraic neighborhood $U'$ of $M'$ in $\R^{n+1}$ and a Nash retraction $\rho':U'\to M'$. Let $\sigma':\R^{n+1}\to[0,1]$ be a $\Cc^\infty$ bump function such that $\sigma'|_{M'}=1$ and $\sigma'|_{\R^{n+1}\setminus U'}=0$. Suppose that $g|_{S'}$ extends to a $\Cc^{t(p)}$ differentiable function $H'$ on $\R^n\setminus C'$. Thus, $G':=\sigma'\cdot(H'\circ\varphi'^{-1}\circ\rho'):\R^{n+1}\to\R$ is a $\Cc^{t(p)}$ differentiable function on $\R^{n+1}$, after extending it by $0$ outside $U'$. If $T':=\varphi(S')$, we have $G'|_{T'}=g\circ\varphi'^{-1}|_{T'}$. By \cite[Cor.1.7]{bcm} and the discussing above there exists an $\Ss^p$ function $G_p:\R^{n+1}\to\R$ such that $G_p|_{T'}=f\circ\varphi'^{-1}|_{T'}$. Define $F_p:=G_p\circ\varphi':\R^n\setminus C'\to\R$, which is a $\Ss^p$ function such that $F_p|_S=f$. 

If we apply recursively the previous strategy to $S':=S_{t(p)}=\cl(S)\setminus C_p$ and $H':=H_{t(p)}$ for each $p\geq 1$, we conclude $f\in\bigcap_{p\geq0}\Ss^p(S)=\Ss^{(\infty)}(S)$, as required.
\end{proof}

\subsection*{Acknowledgements}

The first author is supported by Spanish STRANO PID2021-122752NB-I00. The second author is supported by GNSAGA of INDAM, and partially supported by Spanish STRANO PID2021-122752NB-I00. This work was also supported by the ``National Group for Algebraic and Geometric Structures, and their Applications'' (GNSAGA - INDAM).

In addition, this article has been developed during several one month research stays of the first author in the Dipartimento di Matematica of the Universit\`a di Trento. The first author would like to express his gratitude to the department for the invitations and the very pleasant working conditions. 

The authors thank Prof. Kollar for the interest shown in our work and for very valuable exchanges of ideas that have notably improved the presentation of our article, making the final version of this article clearer and more precise. We have also understood that it is valuable to guide the reader to perform a simplified and shortened reading of our manuscript if he/she is only interested in having a converse of Cartan’s Theorem B.

\end{document}